\documentclass[11pt, DIV10,a4paper]{article}
\usepackage{natbib}
\newcommand {\ctn}{\citet} % change to \citet if using natbib
\usepackage{graphicx,subfigure,amsmath,latexsym,amssymb}
\usepackage{float,epsfig,multirow,rotating,times}
\usepackage{upgreek,wrapfig}
\usepackage[colorlinks,breaklinks,bookmarksopen,bookmarksnumbered]{hyperref}
\usepackage{comment}
\usepackage{url}
\usepackage{gensymb}
\usepackage{algorithm}

\newcommand{\bb}{\boldsymbol{b}}

\newcommand{\bg}{\boldsymbol{g}}
\newcommand{\bof}{\boldsymbol{f}}

\newcommand{\btheta}{\boldsymbol{\theta}}
\newcommand{\bLambda}{\boldsymbol{\Lambda}}

\newcommand{\bbeta}{\boldsymbol{\beta}}

\newcommand{\bSigma}{\boldsymbol{\Sigma}}

\newcommand{\bmu}{\boldsymbol{\mu}}
\newcommand{\bnu}{\boldsymbol{\nu}}

\newcommand{\bc}{\boldsymbol{c}}
\newcommand{\bD}{\boldsymbol{D}}

\newcommand{\bh}{\boldsymbol{h}}
\newcommand{\bH}{\boldsymbol{H}}

\newcommand{\bA}{\boldsymbol{A}}

\newcommand{\bx}{\bm{x}}
\newcommand{\bX}{\boldsymbol{X}}
\newcommand{\by}{\boldsymbol{y}}

\newcommand{\bz}{\boldsymbol{z}}

\newcommand{\bzero}{\boldsymbol{0}}

\newtheorem{theorem}{Theorem}

\newtheorem{lemma}[theorem]{Lemma}

\newtheorem{remark}[theorem]{Remark}

\newtheorem{assumption}{Assumption}[section]

\newenvironment{proof}[1][Proof]{\textbf{#1.} }{\ \rule{0.5em}{0.5em}}

\newcommand{\bm}{\mathbf}

\numberwithin{equation}{section}
\numberwithin{algo}{section}
\numberwithin{table}{section}
\numberwithin{figure}{section}

\usepackage{setspace}
\usepackage[mathscr]{euscript}
\usepackage[margin=1in]{geometry}
\begin{document}
%\renewcommand\baselinestretch{1.4}

%\normalsize

\title{\vspace{-0.8in}
%{A Bayesian Approach to Determination of Convergence, Divergence and Oscillation of Infinite Series
%with Application to Riemann Hypothesis}
{\bf Function Optimization with Posterior Gaussian Derivative Process}}
\author{Sucharita Roy and Sourabh Bhattacharya\thanks{
Sucharita Roy is an Assistant Professor and Head of the Department of Mathematics in St. Xavier's College, Kolkata. 
Sourabh Bhattacharya is a Professor in Interdisciplinary Statistical Research Unit, Indian Statistical
Institute, 203, B. T. Road, Kolkata 700108.
Corresponding e-mail: sourabh@isical.ac.in.}}
\date{\vspace{-0.5in}}
\maketitle%

\begin{abstract}
Function optimization is fundamental to virtually all scientific disciplines, yet no single method dominates across all problem classes. 
In this article we propose a novel Bayesian optimization algorithm specifically designed for functions whose first and second partial derivatives are available. 
Our approach models the objective function as a Gaussian process, which induces a Gaussian derivative process. 
Given a small set of initial function evaluations, we obtain the posterior distribution of the derivative process and then construct a posterior for the stationary points by setting the derivative to zero. 
A carefully designed recursive importance‑resampling scheme, combined with prior constraints on the gradient norm and Hessian definiteness, drives the algorithm progressively toward the true optima.

We provide a rigorous theoretical foundation, proving almost sure convergence of the algorithm to the true optima as the number of stages tends to infinity. 
This relies on novel results establishing almost sure uniform convergence of the Gaussian process and its derivative process posteriors to the true function and its derivatives under fixed‑domain infill asymptotics; rates of convergence are also derived. 
In addition, we give a Bayesian characterization of the number of optima using the information accumulated during the algorithm.

The practical effectiveness of our method is demonstrated on five diverse optimization problems, including finding maxima, minima, saddle points, and inconclusive cases. 
Problems range from simple one‑dimensional examples to challenging $50$‑ and $100$‑dimensional nonlinear least‑squares problems. 
Notably, on a real‑world Poisson regression (AIDS deaths data), our Bayesian procedure achieves a substantially smaller gradient norm at the maximum likelihood estimate than Fisher scoring, BFGS, simulated annealing, and multi‑start quasi‑Newton. %all of which we implemented in R. 
The posterior simulation nature of our algorithm enables it to explore neighborhoods of solutions obtained by other methods, consistently yielding more accurate results. 
	The code is written in C with MPI parallelisation and is available at 
	\url{https://github.com/Sourabh-Bhattacharya/FUNCTION_OPT_GDP} for implementing the AIDS data problem. 
	\\[2mm]
{\it Keywords:} Dirichlet process; Fixed‑domain infill asymptotics; Function optimization; Gaussian derivative process; Importance resampling; Parallel processing; Transformation based Markov Chain Monte Carlo (TMCMC).
\end{abstract}

\section{Introduction}
\label{sec:intro}

Function optimization is a cornerstone of scientific computing, with applications across virtually every discipline. Despite decades of research, no single optimization method dominates all problem classes; for any sufficiently large family of problems, it is remarkably difficult to identify a methodology that consistently outperforms others in terms of theoretical foundation, accuracy, computational efficiency, or robustness. Consequently, practitioners often resort to heuristic methods tailored to the specific problem at hand.

Among stochastic optimization methods, simulated annealing is a general‑purpose technique, but its practical success depends critically on carefully chosen temperature schedules and proposal distributions, both of which become extremely challenging to tune in moderate to high dimensions.

In this article, we propose and develop a novel Bayesian algorithm for optimizing functions whose first and second partial derivatives are available. Our method assumes that the objective function can be evaluated exactly (no observation noise) and that its first and second partial derivatives are available analytically or via automatic differentiation. Thus the setting is that of deterministic optimization, typical in computer experiments and engineering design, rather than stochastic optimization with noisy evaluations. Our approach embeds the objective function, together with its derivatives, into a random function framework driven by Gaussian processes and their induced derivative processes. With data comprising a set of input points and the corresponding function values, we first obtain the posterior derivative process. We then construct the posterior distribution of the solutions obtained by setting the random partial derivative functions to the null vector. This posterior is expected to emulate the stationary points of the objective function.

We further incorporate a uniform prior over the domain, constrained so that the first partial derivatives have small Euclidean norm and the Hessian matrix is positive definite (for minimization) or negative definite (for maximization). These constraints help the posterior solutions approximate the true optima even when the dataset is small. However, for very small datasets the prior may be too restrictive to allow effective posterior simulation. Therefore, we begin with a weaker restriction (for instance, bounding the gradient norm by a relatively large constant) to obtain an initial set of posterior solution simulations. Then, in successive iterative stages, we refine the prior restrictions (making them stricter), and at each stage we augment the dataset with new realizations (and their function values) that satisfy the current, tighter constraints. As the number of stages tends to infinity, the posterior distributions converge to the true optima. These ideas lead to a general, effective, and adaptive optimization algorithm.

The convergence of our algorithm hinges on the convergence of the posterior derivative process to the true derivatives of the objective function, which in turn depends on a suitable design of the input points. Under an appropriate fixed‑domain infill asymptotics setup, we prove almost sure uniform convergence of the posterior Gaussian process and its derivative process to the objective function and its derivatives. Notably, we also obtain rates of convergence under a specific infill design. To the best of our knowledge, these results are new and of independent interest. Using them, we prove almost sure convergence of our optimization algorithm to the true optima as the number of stages grows. As an aside, we also provide a Bayesian characterization of the number of optima, exploiting information accumulated by the algorithm.

We illustrate our method on five optimization problems that include finding maxima, minima, saddle points, and even inconclusive cases. The examples range from simple one‑dimensional problems to challenging 50‑ and 100‑dimensional problems. In each case we obtain encouraging and insightful results. Along the way we discuss various computational and accuracy issues. A key strength of our approach is its ability to achieve significantly more accurate solutions than existing optimization algorithms, thanks to the posterior simulation paradigm embedded in the method. Indeed, by exploring neighborhoods around any solution provided by other methods, our algorithm almost surely finds a point at least as close to the true optimum.

We note that existing Bayesian optimization methods using Gaussian processes typically treat the objective function as a black box and assume derivatives are unavailable 
(see, for example, \ctn{Frazier18} and references therein). Those methods are naturally less accurate when derivatives are accessible. Moreover, we are not aware of any convergence results for such derivative‑free Gaussian process methods, which rely on heuristic acquisition functions. Because many acquisition functions exist, each with its own advantages and drawbacks, those methods lack a solid theoretical foundation and may be unreliable in practice. In this article we do not further discuss those existing approaches.

The remainder of the paper is organized as follows. Section~\ref{sec:gp_derivative} provides the derivation of the posterior Gaussian derivative process. Section~\ref{sec:posterior_optima} derives the posterior distribution for the random optima obtained by setting the derivative process to zero, together with additional restrictions based on the objective function’s derivatives. Section~\ref{sec:conv} presents almost sure uniform convergence results for the Gaussian process and its derivative process posteriors, along with proofs. Section~\ref{sec:algo} introduces our general‑purpose Bayesian optimization algorithm. Section~\ref{sec:recursive} establishes a Bayesian characterization of the number of optima. Section~\ref{sec:exps} illustrates the algorithm on various optimization problems. Section~\ref{sec:conclusion} summarizes our contributions and offers concluding remarks.

\section{Posterior Gaussian derivative process}
\label{sec:gp_derivative}

\subsection{Details of the objective function}
\label{subsec:details_objective_function}
Consider any function $f:\mathbb R^d\mapsto \mathbb R$, where $\mathbb R$ is the real line and $d~(\geq 1)$ is the dimension 
of the input space, assumed to be finite. We further assume that the second order
partial derivatives of $f$, namely, $\partial^2f(\bx)/\partial x_i\partial x_j$ exist and are continuous for all $\bx\in\mathcal X\subseteq\mathbb R^d$ for $i,j=1,\ldots,d$.
The objective is to optimize the function $f(\bx)$ with respect to $\bx\in\mathcal X$. For theoretical purposes, we assume that $\mathcal X$ is compact; this assumption
is not required for implementation of our methodology.

\subsection{Data from the objective function}
\label{subsec:data_objective_function}
Assume that corresponding to arbitrary inputs $\{\bx_1,\ldots,\bx_n\}\in\mathcal X$, where $n>1$, the output vector 
$\bof_n=\left(f(\bx_1),\ldots,f(\bx_n)\right)^T$ is available. Here $^T$ denotes transpose.
Let $\bD_n=\{(\bx_i,f(\bx_i):i=1,\ldots,n\}$.

\subsection{Gaussian process representation of the objective function}
\label{subsec:gp_objective_function}
Let $g:\mathbb R^d\mapsto \mathbb R$ denote a random function such that given $\bD_n$, $g(\bx_i)=f(\bx_i)$ for $i=1,\ldots,n$. 

Since Gaussian processes have the above interpolation property, we model $g(\cdot)$ by a Gaussian process with mean function $\mu(\cdot)$ and covariance function
$\sigma^2c(\cdot,\cdot)$, where $\sigma^2$ is the process variance. In other words, $E\left[g(\bx)\right]=\mu(\bx)$ for any $\bx\in\mathcal X$ and 
$Cov(g(\bx),g(\by))=\sigma^2c(\bx,\by)$, for all $\bx,\by\in\mathcal X$. Let $\mu(\cdot)$ be continuous in $\mathcal X$ and
$c(\cdot,\cdot)$ be Lipschitz continuous on $\mathcal X\times\mathcal X$.
Then $g(\cdot)$ is actually a continuous-path Gaussian process with mean function $\mu(\cdot)$ and covariance function $\sigma^2c(\cdot,\cdot)$.

We shall use the notation $\bg_n$ to denote $(g(\bx_1),\ldots,g(\bx_n))^T$ when distribution of this vector will be considered and $\bof_n$ when this vector is conditioned upon.

\subsection{Gaussian derivative process}
\label{subsec:gp_derivative}
For $\bx=(x_1,\ldots,x_d)$ and $\by=(y_1,\ldots,y_d)$, let 
%$$\frac{\partial^4c(\bx^*,\by^*)}{\partial x_i\partial y_i\partial x_j\partial y_j}=
%\frac{\partial^4c(\bx,\by)}{\partial x_i\partial y_i\partial x_j\partial y_j}\bigg |_{\bx=\bx^*,by=\by^*}$$
%exist for all $\bx^*,\by^*\in\mathcal X$ and be Lipschitz continuous on $\mathcal X\times\mathcal X$ for $i,j=1,\ldots,d$. 
us assume that the second order mixed partial derivatives 
$$\frac{\partial^2c(\bx^*,\by^*)}{\partial x_i\partial y_i}=\frac{\partial^2c(\bx,\by)}{\partial x_i\partial y_i}\bigg |_{\bx=\bx^*,\by=\by^*}$$ 
are Lipschitz continuous on $\mathcal X\times\mathcal X$ for $i=1,\ldots,d$. 

With the above assumption on the covariance function and with the further assumption that $\mu(\cdot)$ is twice continuously differentiable
with continuous mixed second order partial derivatives, for  
$\bx=(x_1,\ldots,x_d)$, $$g'_i(\bx^*)=\frac{\partial g(\bx^*)}{\partial x_i}=\frac{\partial g(\bx)}{\partial x_i}\bigg |_{\bx=\bx^*},$$ 
%and $$g''_{ij}(\bx^*)=\frac{\partial^2 g(\bx^*)}{\partial x_i\partial x_j}=\frac{\partial^2 g(\bx)}{\partial x_i\partial x_j}\bigg |_{\bx=\bx^*},$$ 
corresponding
to the original continuous-path Gaussian process $g(\cdot)$ exists for $i=1,\ldots,d$, for all $\bx^*\in\mathcal X$. Specifically, for $i=1,\ldots,d$,
$g'_i(\cdot)=\partial g(\cdot)/\partial x_i$ is a continuous-path Gaussian process with mean function $\mu'_i(\cdot)=\partial \mu(\cdot)/\partial x_i$ and 
covariance function $\sigma^2\partial^2c(\cdot,\cdot)/\partial x_i\partial y_i$. 
%and $g''_{ij}(\cdot)=\partial^2 g(\cdot)/\partial x_i\partial x_j$ is a continuous-path Gaussian process with mean function 
%$\mu''_{ij}(\cdot)=\partial^2 \mu(\cdot)/\partial x_i\partial x_j$ and covariance function $\sigma^2\partial^4c(\cdot,\cdot)/\partial x_i\partial y_i\partial x_j\partial y_j$.

For general details on Gaussian and Gaussian derivative processes, see, for example, \ctn{Adler81}, \ctn{Adler07}.

\subsection{Joint distribution of Gaussian variables and Gaussian derivative variables}
\label{subsec:joint}

Note that, given $\bx^*,\by^*\in\mathcal X$, 
\begin{align}
&Cov\left(g'_i(\bx^*),g'_j(\bx^*)\right)=\sigma^2\frac{\partial^2 c(\bx,\by)}{\partial x_i\partial y_j} \bigg |_{\bx=\bx^*,\by=\bx^*};\label{eq:cov1}\\
&Cov\left(g'_i(\bx^*),g(\by^*)\right)=\sigma^2\frac{\partial c(\bx,\by)}{\partial x_i} \bigg |_{\bx=\bx^*,\by=\by^*}.\label{eq:cov2}
\end{align}	
%and
%$cov\left(g'_i(\bx),g(\by)\right)=\sigma^2\partial c(\bx,\by)/\partial x_i\partial x_j$ is non-zero for any $\bx^*\in\mathcal X$.
With the above covariance forms, for given $\bx^*\in\mathcal X$, we have 
%the joint distribution of $\left(g'_1(\bx^*),\ldots,g'_d(\bx^*),g(\bx_1),\ldots,g(\bx_n)\right)^T$ is given by
\begin{equation}
	\left(g'_1(\bx^*),\ldots,g'_d(\bx^*),g(\bx_1),\ldots,g(\bx_n)\right)^T
	\sim N_{d+n}\left(\bnu_{d+n},\sigma^2\bSigma^{\overline{d+n}\times\overline{d+n}}\right),
	\label{eq:mvn1}
\end{equation}
that is, the vector on the left hand side of (\ref{eq:mvn1}) has the $(d+n)$-variate normal distribution 
with mean vector $\bnu_{d+n}$ and covariance matrix $\sigma^2\bSigma^{\overline{d+n}\times\overline{d+n}}$. Here
\begin{equation}
	\bnu_{d+n}(\bx^*)=\left({\bmu'_{d}(\bx^*)}^T,\bmu^T_{n}\right)^T,
	\label{eq:bmu}
\end{equation}
where $\bmu'_{d}(\bx^*)=\left(\partial\mu(\bx^*)/\partial x_1,\ldots,\partial\mu(\bx^*)/\partial x_d\right)^T$ with $\partial\mu(\bx^*)/\partial x_i=
\frac{\partial\mu(\bx)}{\partial x_i}\bigg |_{\bx=\bx^*}$ for $i=1,\ldots,d$, and
$\bmu_{n}=\left(\mu(\bx_1),\ldots,\mu(\bx_n)\right)^T$.
Also,
\begin{equation}
	\bSigma^{\overline{d+n}\times\overline{d+n}}(\bx^*)=\begin{pmatrix}\bSigma^{d\times d}_{11}(\bx^*) & \bSigma^{d\times n}_{12}(\bx^*)\\ 
	\bSigma^{n\times d}_{21}(\bx^*) & \bSigma^{n\times n}_{22}\end{pmatrix},
	\label{eq:bSigma}
\end{equation}
where $\bSigma^{d\times d}_{11}$ is the $d$-th order correlation matrix with $(i,j)$-th element $\sigma^{-2}Cov\left(g'_i(\bx^*),g'_j(\bx^*)\right)$ 
where the covariance term is given by (\ref{eq:cov1}), 
%$\sigma^2\frac{\partial^2c(\bx,\by)}{\partial x_i\partial y_j}\bigg |_{\bx=\bx^*,\by=\bx^*}$, 
$\bSigma^{d\times n}_{12}(\bx^*)$ is the $d\times n$ matrix with $(i,j)$-th element $\sigma^{-2}Cov\left(g'_i(\bx^*),g(\bx_j)\right)$ 
where the covariance term is of the same form as (\ref{eq:cov2}) with
$\by^*$ replaced by $\bx_j$, $\bSigma^{n\times d}_{21}(\bx^*)$ is the transpose of $\bSigma^{d\times n}_{12}(\bx^*)$, and $\bSigma^{n\times n}_{22}$ is the $n\times n$ matrix
with $(i,j)$-th element $c(\bx_i,\bx_j)$, the correlation between $g(\bx_i)$ and $g(\bx_j)$. 
In our examples, we shall consider the squared exponential correlation function having the form
\begin{equation}
	c(\bx,\by)=\exp\left(-\frac{1}{2}(\bx-\by)^T\bLambda^{-1}(\bx-\by)\right),
	\label{eq:exp_corr1}
\end{equation}
for $\bx,\by\in\mathcal X$, where $\bLambda$ is a $d\times d$ diagonal matrix with positive diagonal elements $\lambda_i$; $i=1,\ldots,d$.
It follows that $\bSigma_{11}=\bLambda^{-1}$ and
for $j=1,\ldots,n$, the $j$-th column of $\bSigma_{12}(\bx^*)$ is $-\bLambda^{-1}(\bx^*-\bx_j)c(\bx^*,\bx_j)$.

\subsection{Posterior distribution of the Gaussian derivatives given the data and parameters}
\label{subsec:posterior1}
From (\ref{eq:mvn1}) it follows that the joint posterior distribution of $\bg'(\bx^*)=\left(g'_1(\bx^*),\ldots,g'_d(\bx^*)\right)^T$ 
given $\bx^*$, $\bD_n$, $\sigma^2$ and other parameters $\btheta$, is the following $d$-variate normal distribution:
\begin{equation}
	\pi\left(\bg'(\bx^*)|\sigma^2,\btheta,\bD_n\right)\equiv N_d\left(\tilde\bmu(\bx^*),\sigma^2\tilde\bSigma(\bx^*)\right),
	\label{eq:postpred1}
\end{equation}
where 
\begin{align}
	\tilde\bmu(\bx^*)&=\bmu'_d(\bx^*)+\bSigma_{12}(\bx^*)\bSigma^{-1}_{22}\left(\bof_n-\bmu_n\right);\label{eq:postmean1}\\
	\tilde\bSigma(\bx^*)&=\bSigma_{11}(\bx^*)-\bSigma_{12}(\bx^*)\bSigma^{-1}_{22}\bSigma_{21}(\bx^*).\label{eq:postvar1}
\end{align}

\subsection{Prior and posterior distributions of the parameters}
\label{subsec:prior_posterior_parameters}
In our examples we assume that $\mu(\bx)=\bh(\bx)^T\bbeta$, where $\bh(\bx)^T=(1,x_1,\ldots,x_d)$ and $\bbeta=(\beta_0,\beta_1,\ldots,\beta_d)^T\in \mathbb R^d$.
Let $\bH^{n\times\overline{d+1}}=\left(\bh(\bx_1),\ldots,\bh(\bx_n)\right)^T$.

\subsubsection{Priors for the parameters}
\label{subsubsec:prior1}
We assume that {\it a priori} 
\begin{equation}
	\pi(\bbeta|\sigma^2)\equiv N_{d+1}\left(\bbeta_0,\sigma^2\bSigma_0\right),
	\label{eq:prior1}
\end{equation}
where $\bbeta_0$ is the mean vector and $\bSigma_0$ is the positive definite covariance matrix, and 
\begin{equation}
\pi\left(\sigma^{-2}\right)\equiv \mathcal G(a,b), 
	\label{eq:prior2}
\end{equation}
the gamma distribution with mean $a/b$ and variance $a/b^2$, where $a,b>0$. 

\subsubsection{Posteriors for the parameters}
\label{subsubsec:posterior1}
Let us first obtain the posterior distribution of $\sigma^{-2}$ given $\bD_n$.
Note that
\begin{align}
	&\pi(\sigma^{-2}|\bD_n)\propto\pi\left(\sigma^{-2}\right)\pi\left(\bg_n|\sigma^{-2}\right)\notag\\
	&\qquad = \pi\left(\sigma^{-2}\right)\int\pi\left(\bg_n|\sigma^{-2},\bbeta\right)\pi\left(\bbeta|\sigma^{2}\right)d\bbeta.
	\label{eq:pos1}
\end{align}
To obtain $\pi\left(\bg_n|\sigma^{-2}\right)=\int\pi\left(\bg_n|\sigma^{-2},\bbeta\right)\pi\left(\bbeta|\sigma^{2}\right)d\bbeta$ note that
\begin{equation}
\pi\left(\bg_n|\sigma^{-2},\bbeta\right)\equiv N_n\left(\bH\bbeta,\sigma^2\bSigma_{22}\right)
	\label{eq:pos_D}
\end{equation}
and since 
$\pi\left(\bbeta|\sigma^{2}\right)$ has the normal distribution (\ref{eq:prior1}), it follows that
\begin{equation}
	\pi\left(\bg_n|\sigma^{-2}\right)\equiv N_n\left(\bH\bbeta_0,\sigma^2\left(\bH\bSigma_0\bH^T+\bSigma_{22}\right)\right).
	\label{eq:pos2}
\end{equation}
Combining (\ref{eq:pos2}) with (\ref{eq:pos1}) and (\ref{eq:prior2}) it follows that
\begin{equation}
	\pi(\sigma^{-2}|\bD_n)\equiv \mathcal G\left(a+\frac{d}{2},b+\frac{1}{2}\left(\bof_n-\bH\bbeta_0\right)^T\left(\bH\bSigma_0\bH^T+\bSigma_{22}\right)^{-1}
	\left(\bof_n-\bH\bbeta_0\right)\right).
	\label{eq:pos3}
\end{equation}
Also, combining (\ref{eq:pos_D}) and (\ref{eq:prior1}) it is easy to see that
\begin{equation}
	\pi\left(\bbeta|\bD_n,\sigma^2\right)\equiv 
	N_{d+1}\left(\left(\bH^T\bSigma^{-1}_{22}\bH+\bSigma^{-1}_0\right)^{-1}\left(\bH^T\bSigma^{-1}_{22}\bof_n+\bSigma^{-1}_0\bbeta_0\right),
	\sigma^2\left(\bH^T\bSigma^{-1}_{22}\bH+\bSigma^{-1}_0\right)^{-1}\right).
	\label{eq:pos4}
\end{equation}

\subsection{Marginal posterior distribution of the derivative process}
\label{subsec:marginal_posterior_derivative}
Now, from (\ref{eq:postpred1}) it follows that  
\begin{equation}
	\pi\left(\bg'(\bx^*)|\sigma^2,\bbeta,\bD_n\right)\equiv N_d\left(\bA\bbeta+\bSigma_{12}(\bx^*)\bSigma^{-1}_{22}(\bof_n-\bH\bbeta),\sigma^2\tilde\bSigma(\bx^*)\right),
	\label{eq:postpred2}
\end{equation}
where $\bA^{d\times\overline{d+1}}
=\begin{pmatrix}\bzero^{d\times 1} & \mathbb I_d\end{pmatrix}$. Here $\bzero^{d\times 1}$ is the $d$-dimensional null vector and $\mathbb I_d$ is the identity
matrix of order $d$. Integrating (\ref{eq:postpred2}) with respect to (\ref{eq:pos4}) we obtain
\begin{equation}
	\pi\left(\bg'(\bx^*)|\sigma^2,\bD_n\right)\equiv N_d\left(\hat\bmu'(\bx^*),\sigma^2\hat\bSigma(\bx^*)\right),
	\label{eq:postpred3}
\end{equation}
where
$\hat\bmu'(\bx^*)$ and $\hat\bSigma(\bx^*)$ are given by
\begin{align}
	&\hat\bmu'(\bx^*)=\bA\hat\bbeta+\bSigma_{12}(\bx^*)\bSigma^{-1}_{22}(\bof_n-\bH\hat\bbeta);\label{eq:postpred_mean1}\\
	&\hat\bSigma(\bx^*)=
	\tilde\bSigma(\bx^*)+\left(\bA-\bSigma_{12}(\bx^*)\bSigma^{-1}_{22}\bH\right)\left(\bH^T\bSigma^{-1}_{22}\bH+\bSigma^{-1}_0\right)^{-1}
	\left(\bA-\bSigma_{12}(\bx^*)\bSigma^{-1}_{22}\bH\right)^T,
	\label{eq:postpred_var1}
\end{align}
with
\begin{equation}
	\hat\bbeta=\left(\bH^T\bSigma^{-1}_{22}\bH+\bSigma^{-1}_0\right)^{-1}\left(\bH^T\bSigma^{-1}_{22}\bof_n+\bSigma^{-1}_0\bbeta_0\right).
	\label{eq:hat_beta}
\end{equation}
Integrating (\ref{eq:postpred3}) with respect to (\ref{eq:pos3}) we obtain
\begin{align}
	&\pi\left(\bg'(\bx^*)|\bD_n\right)\equiv%\notag\\
	t_d\left(\hat\bmu'(\bx^*),
	\frac{\left(a+\frac{d}{2}\right)\hat\bSigma(\bx^*)^{-1}}{b+\frac{1}{2}\left(\bof_n-\bH\bbeta_0\right)^T\left(\bH\bSigma_0\bH^T+\bSigma_{22}\right)^{-1}
	\left(\bof_n-\bH\bbeta_0\right)},2\left(a+\frac{d}{2}\right)\right),
	\label{eq:postpred4}
\end{align}
where for any $d$-dimensional vector $\bmu$, $d$-th order covariance matrix $\bSigma$, and $\alpha>0$, $t_d\left(\bmu,\bSigma^{-1},\alpha\right)$
is a $d$-variate Student's $t$ distribution with density at $\bx\in\mathbb R^d$ given by
$$t_d\left(\bx:\bmu,\bSigma^{-1},\alpha\right)=C\left[1+\alpha^{-1}(\bx-\bmu)^T\bSigma^{-1}(\bx-\bmu)\right]^{-\frac{(\alpha+d)}{2}},$$
where $$C=\frac{\Gamma\left(\frac{\alpha+d}{2}\right)}{\Gamma\left(\frac{\alpha}{2}\right)\left(\alpha\pi\right)^{\frac{d}{2}}},$$
with $\Gamma(\cdot)$ denoting the gamma function.

\section{Posterior distribution of random optima corresponding to the posterior derivative process}
\label{sec:posterior_optima}
	From (\ref{eq:postpred4}) we obtain the following posterior density at $\bg'(\bx^*)=\bzero$: 
	%to be given by %the density of the posterior (\ref{eq:postpred4}) is given by
\begin{align}
	\pi\left(\bg'(\bx^*)=\bzero|\bD_n\right)
	\propto\left[1+\frac{\hat\bmu'(\bx^*)^T\hat\bSigma(\bx^*)^{-1}\hat\bmu'(\bx^*)}
	{2b+\left(\bof_n-\bH\bbeta_0\right)^T\left(\bH\bSigma_0\bH^T+\bSigma_{22}\right)^{-1}\left(\bof_n-\bH\bbeta_0\right)}
	\right]^{-(a+d)}.
	\label{eq:postpred5}
\end{align}
Now, with prior $\pi(\bx^*)$ on $\bx^*$, the posterior of $\bx^*$, given $\bg'(\bx^*)$ and $\bD_n$ can be obtained as follows:
\begin{align}
	&\pi(\bx^*|\bg'(\bx^*),\bD_n)\propto\pi(\bx^*)\pi(\bg'(\bx^*),\bg_n|\bx^*)\notag\\
	&\qquad=\pi(\bx^*)\pi(\bg'(\bx^*)|\bD_n,\bx^*)\pi(\bg_n|\bx^*)\notag\\
	&\qquad=\pi(\bx^*)\pi(\bg'(\bx^*)|\bD_n,\bx^*)\pi(\bg_n)\notag\\
	&\qquad\propto\pi(\bx^*)\pi(\bg'(\bx^*)|\bD_n,\bx^*).
	\label{eq:inv1}
\end{align}
In the second step of (\ref{eq:inv1}), $\pi(\bg_n|\bx^*)$ is the marginal distribution of $\bg_n$, integrated over the parameters. Since this does
not depend upon $\bx^*$, we denoted this as $\pi(\bg_n)$ in the third step of (\ref{eq:inv1}). From (\ref{eq:inv1}) it then follows that
\begin{equation}
	\pi(\bx^*|\bg'(\bx^*)=\bzero,\bD_n)\propto\pi(\bx^*)\pi(\bg'(\bx^*)=\bzero|\bD_n,\bx^*).
	\label{eq:inv2}
\end{equation}
Now, $\pi(\bg'(\bx^*)=\bzero|\bD_n,\bx^*)$ will also depend upon parameters of the covariance function, which will be unknown generally. It is legitimate to estimate
them using the maximum likelihood estimation method (see, for example, \ctn{Santner03}) and treat them as fixed. %using the $R$ package ``mlegp".

The formula (\ref{eq:inv2}) holds for any prior $\pi(\bx^*)$ for $\bx^*$. However, we shall consider a uniform prior on $\mathcal X$ constrained 
by the first and second derivatives of $f(\cdot)$. The details are presented below.

\subsection{Prior for $\bx^*$}
\label{subsec:prior}

Without loss of generality, let us assume that our objective is to obtain the minima of the function $f(\cdot)$ on $\mathcal X$.
For $i,j=1,\ldots,d$, let $f''_{ij}(\bx^*)=\frac{\partial^2f(\bx)}{\partial x_i\partial x_j}\bigg |_{\bx=\bx^*}$ denote the second order partial derivatives
of the objective function $f(\cdot)$ at any $\bx^*\in\mathcal X$. Let $\bSigma''(\bx^*)$ stand for the $d\times d$ 
matrix of such second order partial derivatives at $\bx^*$ with $(i,j)$-th element $f''_{ij}(\bx^*)$. Let $\bSigma''(\bx^*)>0$ denote that $\bSigma''(\bx^*)$
is positive definite.
Then we consider the following prior for $\bx^*$:
\begin{equation}
	\pi(\bx^*)\propto I_{B(\epsilon)}(\bx^*),
	\label{eq:prior_x_star}
\end{equation}
where, for any set $A$ and vector $\bx$, $I_A(\bx)=1$ if $\bx\in A$ and zero otherwise. Also, for any $\bx$ and $\epsilon>0$, 
\begin{equation}
	B(\epsilon)=\mathcal X\cap\left\{\bx:\|\bof'(\bx)\|_d<\epsilon\right\}\cap\left\{\bx:\bSigma''(\bx)>0\right\},
	\label{eq:B}
\end{equation}
where $\|\cdot\|_d$ denotes the Euclidean norm in the $d$-dimensional Euclidean space.

\subsection{Interpretation of the posterior for $\bx^*$ as a pseudo-posterior}
\label{subsec:pseudo}
Before proceeding to the convergence theory, we comment on the nature of the target in (\ref{eq:inv2}). The event $\{\bg'(\bx^*)=\bzero\}$ has probability zero under the continuous distribution of the derivative process. Therefore, the expression $\pi(\bx^*|\bg'(\bx^*)=\bzero,\bD_n)$ is not a conditional probability in the strict Kolmogorov sense. Instead, it should be understood as a density-based target obtained by evaluating the conditional density of the derivative at zero, which serves as a surrogate posterior for the stationary point. This construction is akin to a Gibbs posterior or an energy-based formulation; it is well-defined as a density and leads to a valid sampling scheme. Such pseudo-posteriors have been successfully used in various contexts (see, for example, \ctn{Bissiri16}) and we adopt a similar perspective here.

\section{Almost sure uniform convergence of posterior Gaussian and Gaussian derivative processes}
\label{sec:conv}

	Consider the joint posterior distribution of 
	\begin{equation}
	\pi(g(\cdot),\bg'(\cdot)|\bD_n)=\pi(g(\cdot)|\bD_n)\pi(\bg'(\cdot)|g(\cdot),\bD_n).
		\label{eq:joint1}
	\end{equation}
	Then the marginal posterior $\pi(\bg'(\cdot)|\bD_n)$ is of the same form as (\ref{eq:postpred4}), and the marginal posterior distribution $\pi(g(\cdot)|\bD_n)$ 
	in (\ref{eq:joint1}) corresponds to the $t_1$ process, the form of which is not relevant for our purpose. 
        
	For $n\geq 1$, let $\bX_n$ denote the $n$ input points in $\bD_n$. 
	Note that even after marginalizing out the parameters of the Gaussian process with respect to their posteriors (here $\bbeta$ and $\sigma^2$), 
	the interpolation property of $g(\cdot)$ given $\bD_n$ is preserved. That is, the marginal posterior $\pi(g(\bx^*)|\bD_n)$ gives full posterior mass to $f(\bx^*)$
	if $\bx^*\in\bX_n$.
	%Indeed, given $\bbeta$ and $\sigma^2$ as
	%realizations from their respective posteriors given $\bD_n$, $\pi(g(\cdot)|\bD_n,\bbeta,\sigma^2)$ is a Gaussian process.
	%Let us consider any such realization of $\bbeta$ and $\sigma^2$ associated with non-null sets of their respective posterior probability measures. Also, given such
	%$\bbeta$ and $\sigma^2$, 
	
	Let $g_n(\cdot)$ denote any random function associated with any non-null set of the marginalized posterior measure of $g(\cdot)$ given $\bD_n$ 
	(here, the $t_1$ posterior measure). Also, let $\bg'_n(\cdot)$ denote 
	any $d$-dimensional random function associated with any non-null set of the marginalized posterior measure of $\bg'(\cdot)$ given $\bD_n$, 
	the form of which is provided explicitly by (\ref{eq:postpred4}). 
	Theorems \ref{theorem:infill}, \ref{theorem:infill2}, \ref{theorem:theorem1} and \ref{theorem:theorem2} prove almost sure uniform convergence of $g_n(\cdot)$ and
	$\bg'_n(\cdot)$ to $f(\cdot)$ and $\bof'(\cdot)$ respectively, as $n\rightarrow\infty$. In particular, Theorems \ref{theorem:theorem1} and \ref{theorem:theorem2}
	also provide rates of such convergences. 

Before presenting the theorems, we state a set of standing assumptions that will be used throughout this section.

\begin{assumption}[Assumptions for convergence results]
\label{ass:main}
We assume the following:
\begin{itemize}
\item[(A1)] The domain $\mathcal X\subset\mathbb R^d$ is compact and convex.
\item[(A2)] The objective function $f$ is twice continuously differentiable on an open set containing $\mathcal X$, with continuous mixed second-order partial derivatives.
\item[(A3)] The mean function $\mu(\cdot)$ of the Gaussian process prior is twice continuously differentiable with continuous mixed second-order partial derivatives.
\item[(A4)] The correlation function $c(\cdot,\cdot)$ is such that the mixed partial derivatives $\frac{\partial^2 c(\bx,\by)}{\partial x_i\partial y_i}$ exist and are Lipschitz continuous on $\mathcal X\times\mathcal X$ for all $i=1,\ldots,d$. For the higher-order convergence results (Theorems \ref{theorem:theorem1} and \ref{theorem:theorem2}), we additionally require the existence and Lipschitz continuity of $\frac{\partial^4 c(\bx,\by)}{\partial x_i\partial y_i\partial x_j\partial y_j}$ for all $i,j$.
\item[(A5)] The input points $\bX_n$ satisfy a fixed-domain infill asymptotics condition: for any $\bx\in\mathcal X$, there exists a sequence $\bx_n\in\bX_n$ such that 
	\begin{equation}
		\|\bx_n-\bx\|_d\rightarrow 0~\mbox{as}~n\rightarrow\infty, 
		\label{eq:infill_limit}
	\end{equation}
		and moreover points of the form $\bx_n+\bh_n$ with $\bh_n\rightarrow\bzero$ also belong to $\bX_n$.
\item[(A6)] The correlation matrices $\bSigma_{22}$ and their derivatives are uniformly non-degenerate: there exists a constant $c_0>0$ such that the smallest eigenvalue of $\bSigma_{22}$ is bounded below by $c_0$ for all $n$.
\end{itemize}
\end{assumption}
Assumption (A6) is used in two essential places. First, the posterior distributions derived in Sections 
\ref{subsec:posterior1} and \ref{subsec:marginal_posterior_derivative}  
explicitly depend on $\bSigma_{22}^{-1}$; the uniform lower bound on the smallest eigenvalue guarantees that 
$\bSigma_{22}$ is invertible for every $n$ and that its inverse remains numerically stable as $n$ increases. 
Second, in the proofs of Theorems \ref{theorem:infill2} and \ref{theorem:theorem1} we require a uniform bound on 
the second derivatives $g''_{ijn}(\cdot)$ over $n$. Such a bound follows from the uniform boundedness of the conditional 
covariance $\tilde\bSigma(\bx^*)$ in (\ref{eq:postvar1}) and the stability of the posterior of $\bg'(\bx^*)$; both rely 
on $\bSigma_{22}^{-1}$ not becoming arbitrarily large. Consequently, (A6) ensures that the almost sure uniform convergence 
results hold without pathological degeneracy of the design points.

These assumptions are standard in the literature on fixed-domain asymptotics for Gaussian processes (see, for example, \ctn{Stein99}) and are sufficient for the proofs that follow.

\begin{lemma}
	\label{lemma:lemma1}
	Consider a sequence of real-valued continuous functions $\{f_n\}_{n=1}^{\infty}$ on any compact set $\mathcal X$ such that  
	$f_n(\bx)\rightarrow f(\bx)$ for all $\bx\in\mathcal X$, where $f$ is some real-valued continuous function on $\mathcal X$. Then
	$$\underset{n\rightarrow\infty}{\lim}\underset{\bx\in\mathcal X}{\sup}~|f_n(\bx)|<\infty.$$
\end{lemma}
\begin{proof}
	Note that $f_n$, for $n\geq 1$, and $f$ are actually uniformly continuous since $\mathcal X$ is compact.
Now let us first consider an arbitrary $\bx_1\in\mathcal X$. Then due to pointwise convergence of $f_n$ to $f$, for any $\epsilon>0$, 
	there exists $n_1\geq 1$ such that for $n\geq n_1$, $|f_n(\bx_1)-f(\bx_1)|<\epsilon$. Moreover, due to uniform continuity of $f_n$ and $f$, 
	there exists an open neighborhood
	$\mathcal N(\bx_1)$ of $\bx_1$ such that $|f_n(\bx)-f(\bx)|<\epsilon_1$ for all $\bx\in\mathcal N(\bx_1)$, where $\epsilon_1$ is some positive finite constant.
	Since $f$ is continuous on the compact set $\mathcal X$, it is uniformly bounded. Hence, $\underset{\bx\in\mathcal N(\bx_1)}{\sup}~|f_n(\bx)|<M_1$
	for all $n\geq n_1$, where $M_1$ is some positive finite constant.

	Now consider another point $\bx_2\in\mathcal X\backslash\mathcal N(\bx_1)$. Then similar argument shows that 
	$\underset{\bx\in\mathcal N(\bx_2)}{\sup}~|f_n(\bx)|<M_2$ for all $n\geq n_2\geq n_1$, 
	where $\mathcal N(\bx_2)$ is some appropriate open neighborhood of $\bx_2$ and $M_2$ is some
	positive finite constant.

	Thus, starting with $\mathcal N(\bx_1)$ and the associated bound $\underset{\bx\in\mathcal N(\bx_1)}{\sup}~|f_n(\bx)|<M_1$,
	continuing the procedure for $i\geq 2$, we can construct neighborhoods $\mathcal N(\bx_i)$ with $\bx_i\in\mathcal X\backslash\cup_{j=1}^{i-1}\mathcal N(\bx_j)$ 
	and bounds $M_i$ such that for all $n\geq n_i\geq n_{i-1}\geq \cdots\geq n_2\geq n_1$, 
	$\underset{\bx\in\mathcal N(\bx_i)}{\sup}~|f_n(\bx)|<M_i$. 
	Note that $\mathcal X\subseteq \cup_{i=1}^{\infty}\mathcal N(\bx_i)$. That is, the set of neighborhoods $\{\mathcal N(\bx_i):i=1,2,\ldots\}$
	constitutes an open cover for $\mathcal X$. Since $\mathcal X$ is compact, there exists a finite sub-cover for $\mathcal X$, say,
	$\{\mathcal N(\bx_{i_j}):j=1,2,\ldots,K\}$, where $K$ is finite. Now, by our construction, for $n\geq n_{i_j}$, 
	$\underset{\bx\in\mathcal N(\bx_{i_j})}{\sup}~|f_n(\bx)|<M_{i_j}$,
	for $j=1,\ldots,K$. Let $n_0=\max\{n_{i_j}:j=1,\ldots,K\}$ and $M=\max\{M_{i_j}:j=1,\ldots,K\}$. 
	Then for all $n\geq n_0$, $\underset{\bx\in\mathcal X}{\sup}~|f_n(\bx)|<M<\infty$.
\end{proof}

\begin{theorem}
\label{theorem:infill}
Consider a fixed-domain infill asymptotics framework satisfying Assumption (A5). Under Assumptions (A1)–(A4) and (A6), for almost all sequences $\{g_n(\cdot)\}_{n=1}^{\infty}$,
	\begin{align}
		&\underset{\bx\in\mathcal X}{\sup}~|g_n(\bx)-f(\bx)|\rightarrow 0,~\mbox{as}~n\rightarrow\infty.
		\label{eq:infill1}
		%&\underset{\bx\in\mathcal X}{\sup}~\|\bg'(\bx)-\bof'(\bx)\|_d\rightarrow 0,
		%\label{eq:infill2}
	\end{align}
%	as $n\rightarrow\infty$, almost surely with respect to the posterior (\ref{eq:postpred4}).
\end{theorem}
\begin{proof}
%	Consider the joint posterior distribution of 
%	\begin{equation}
%	\pi(g(\cdot),\bg'(\cdot)|\bD_n)=\pi(g(\cdot)|\bD_n)\pi(\bg'(\cdot)|g(\cdot),\bD_n).
%		\label{eq:joint1}
%	\end{equation}
%	Then the marginal posterior $\pi(\bg'(\cdot)|\bD_n)$ is of the same form as (\ref{eq:postpred4}), and the marginal posterior distribution $\pi(g(\cdot)|\bD_n)$ 
%	of $g(\cdot)$ in (\ref{eq:joint1}) is the $t_1$ process, the form of which is not relevant for our purpose. 
%	Note that even after marginalizing out the parameters of the Gaussian process with respect to their posteriors (here $\bbeta$ and $\sigma^2$), 
%	the interpolation property of $g(\cdot)$ given $\bD_n$ is preserved. That is, the marginal posterior $\pi(g(\bx^*)|\bD_n)$ gives full posterior mass to $f(\bx^*)$
%	if $\bx^*\in\bD_n$.
%	
	Consider any $\bx\in\mathcal X$. Then there exists $\bx_n\in\bX_n$ for $n\geq 1$ satisfying (\ref{eq:infill_limit}). Now by Taylor's series expansion up to the 
	first order,
	\begin{equation}
	g_n(\bx_n)=g_n(\bx)+(\bx_n-\bx)^T\bg'_n(\bc_n), 
		\label{eq:taylor1}
	\end{equation}
	where $\bc_n$ lies on the line joining $\bx$ and $\bx_n-\bx$.

	Now, for $i=1,\ldots,d$, consider the $i$-th partial derivative $g'_{in}(\cdot)$ of $g_n(\cdot)$. With any sequence $h_{in}\rightarrow 0$ as $n\rightarrow\infty$, 
	we have
	\begin{equation}
		\frac{g_n(x_{1n},\ldots,x_{i-1,n},x_{in}+h_{in},x_{i+1,n},\ldots,x_{dn})-g_n(\bx_n)}{h_{in}}=g'_{in}(\bx^*_n),
		\label{eq:der1}
	\end{equation}
	where $\bx^*_n=(x_{1n},\ldots,x_{i-1,n},x^*_{in},x_{i+1,n},\ldots,x_{dn})$; here $x^*_{in}$ lies between $x_{in}$ and $x_{in}+h_{in}$. 
	Since $(x_{1n},\ldots,x_{i-1,n},x_{in}+h_{in},x_{i+1,n},\ldots,x_{dn})^T\in\bX_n$ and $\bx_n\in\bX_n$, 
	$g_n(x_{1n},\ldots,x_{i-1,n},x_{in}+h_{in},x_{i+1,n},\ldots,x_{dn})=f(x_{1n},\ldots,x_{i-1,n},x_{in}+h_{in},x_{i+1,n},\ldots,x_{dn})$
	and $g_n(\bx_n)=f(\bx_n)$, almost surely. Hence, from (\ref{eq:der1}) it follows that 
	\begin{equation}
		g'_{in}(\bx^*_n)=f'_i(\bz_n),~\mbox{almost surely}, 
	\label{eq:der2}
        \end{equation}
	with
	$\bz_n=(x_{1n},\ldots,x_{i-1,n},z_{in},x_{i+1,n},\ldots,x_{dn})$, where $z_{in}$ lies between $x_{in}$ and $x_{in}+h_{in}$.
	Clearly, $\bz_n\rightarrow\bx$, as $n\rightarrow\infty$. Hence, taking limits of both sides of (\ref{eq:der2}) as $n\rightarrow\infty$, and using
	continuity of $f'_i(\cdot)$, yields
	\begin{equation}
		\underset{n\rightarrow\infty}{\lim}~g'_{in}(\bx^*_n)=f'_i(\bx),~\mbox{almost surely}. 
	\label{eq:der3}
        \end{equation}
	Now, by the hypothesis of Lipschitz continuity of the second order mixed partial derivatives of the correlation function ensures existence and 
	sample path continuity of the partial derivatives $g'_{in}(\cdot)$, for $i=1,\ldots,d$, for any $n\geq 1$. 
	Since $\mathcal X$ is compact, $g'_{in}(\cdot)$ are uniformly continuous on $\mathcal X$, for $i=1,\ldots,d$, for any $n\geq 1$.
	Uniform continuity of $g'_{in}(\cdot)$ for all $n\geq 1$ implies that for any $\epsilon>0$,
	$|g'_{in}(\bx^*_n)-g'_{in}(\bx)|<\epsilon$, whenever $\|\bx^*_n-\bx\|_d<\delta$, where $\delta~(>0)$ depends upon $\epsilon$ only.
	Now, since $\bx^*_n\rightarrow\bx$ as $n\rightarrow\infty$, there exists $n_0~(\geq 1)$ depending upon $\delta$ such that 
	$\|\bx^*_n-\bx\|_d<\delta$ for $n\geq n_0$. Further, using (\ref{eq:der3}) we obtain for any $\bx\in\mathcal X$, the following: 
	\begin{equation}
         \underset{n\rightarrow\infty}{\lim}~g'_{in}(\bx)=\underset{n\rightarrow\infty}{\lim}~g'_{in}(\bx^*_n)+\underset{n\rightarrow\infty}{\lim}~
		(g'_{in}(\bx)-g'_{in}(\bx^*_n))=f'_i(\bx),~\mbox{almost surely}.
		\label{eq:der4}
	\end{equation}
	That is, $g'_{in}(\cdot)$ converges pointwise to $f'_i(\cdot)$ almost surely, as $n\rightarrow\infty$. Moreover, $g'_{in}(\cdot)$ is almost surely continuous
	on $\mathcal X$ for all $n\geq 1$ and $f'_i(\cdot)$ is continuous on $\mathcal X$. Since $\mathcal X$ is compact, we invoke Lemma \ref{lemma:lemma1}
	to conclude that there exists a positive, finite constant $M$ depending upon $f'_i(\cdot)$; $i=1,\ldots,d$ such that
	\begin{equation}
		\underset{n\rightarrow\infty}{\lim}~\underset{\bx\in\mathcal X}{\sup}~|g'_{in}(\bx)|<M,~\mbox{almost surely},~\mbox{for}~i=1,\ldots,d.
		\label{eq:unibound1}
	\end{equation}
	Hence, using the Cauchy-Schwartz inequality in (\ref{eq:taylor1}), boundedness of the partial derivatives $g'_{in}(\cdot)$ for $i=1,\ldots,d$
	for large enough $n$ and (\ref{eq:infill_limit}), we obtain
	\begin{equation}
	\left|g_n(\bx_n)-g_n(\bx)\right|=|(\bx_n-\bx)^T\bg'_n(\bc_n)|\leq \|\bx_n-\bx\|_d\times\|\bg'_n(\bc_n)\|_d\rightarrow 0,~\mbox{as}~n\rightarrow\infty.
		\label{eq:lim1}
	\end{equation}
	Hence, using (\ref{eq:lim1}) and continuity of $f(\cdot)$ we obtain, for any $\bx\in\mathcal X$,
	\begin{equation}
        \underset{n\rightarrow\infty}{\lim}~g_n(\bx)=\underset{n\rightarrow\infty}{\lim}~g_n(\bx_n)=\underset{n\rightarrow\infty}{\lim}~f(\bx_n)
		=f(\underset{n\rightarrow\infty}{\lim}~\bx_n)=f(\bx),
		\label{eq:pointwise}
	\end{equation}
	proving pointwise convergence of $g_n(\cdot)$ to $f(\cdot)$.

	Thus, we have shown that $g_n(\cdot)$ converges pointwise to $f(\cdot)$ on $\mathcal X$ almost surely, as $n\rightarrow\infty$ (equation (\ref{eq:pointwise})), 
	and also that the partial derivatives of $g_n(\cdot)$ are uniformly bounded in the limit almost surely (equation (\ref{eq:unibound1})). The latter also implies
	that $g_n(\cdot)$ is almost surely Lipschitz continuous on $\mathcal X$.
	Since $\mathcal X$ is compact, by the stochastic Arzel\`a--Ascoli lemma (see, for example, Theorem 1.6.6 of \ctn{vanDerVaart96}), which states that if a sequence of random continuous functions is pointwise convergent and has uniformly bounded derivatives (hence equicontinuous and tight), then uniform convergence follows, we conclude that (\ref{eq:infill1}) holds.
\end{proof}

\begin{theorem}
\label{theorem:infill2}
	Consider the same setup as in Theorem \ref{theorem:infill} with the additional smoothness condition on the correlation function (fourth-order mixed derivatives Lipschitz). Then, for almost all sequences $\{\bg^{\prime}_n(\cdot)\}_{n=1}^\infty$, 
	\begin{align}
		&\underset{\bx\in\mathcal X}{\sup}~\|\bg^{\prime}_n(\bx)-\bof^{\prime}(\bx)\|_d\rightarrow 0,~\mbox{as}~n\rightarrow\infty.
		\label{eq:infill2}
	\end{align}
\end{theorem}
\begin{proof}
Note that for $i=1,\ldots,d$, pointwise convergence of $g^{\prime}_{in}(\cdot)$ to $f^{\prime}_i(\cdot)$ as $n\rightarrow\infty$, 
is already shown by (\ref{eq:der4}), in the proof of Theorem \ref{theorem:infill}. 
Hence, if we can show that for $i,j=1,\ldots,d$, the absolute values of the second order partial derivatives $|g^{\prime\prime}_{ijn}(\cdot)|$ 
are uniformly bounded on $\mathcal X$ as $n\rightarrow\infty$, 
then this would imply that  $g^{\prime}_{in}(\cdot)$ are almost surely Lipschitz continuous on $\mathcal X$ for large enough $n$. Since $\mathcal X$ is compact,
this would then imply by the stochastic Arzel\`a–Ascoli lemma (Theorem~1.6.6 of \ctn{vanDerVaart96}) that $\underset{\bx\in\mathcal X}{\sup}~|g^{\prime}_{in}(\bx)-f^{\prime}_i(\bx) |\rightarrow 0$, almost surely,
s $n\rightarrow\infty$, for $i=1,\ldots,d$, which is equivalent to (\ref{eq:infill2}).
Hence, in the rest of the proof, we show that $\underset{n\rightarrow\infty}{\lim}\underset{\bx\in\mathcal X}{\sup}~|g^{\prime\prime}_{ijn}(\bx)|<\infty$.
The argument is given for $i=j$; mixed partials follow similarly.
%
%We prove that for every fixed $\mathbf{x}\in\mathcal{X}$ and all $i,j\in\{1,\dots,d\}$,
%\[
%	\lim_{n\to\infty} g^{\prime\prime}_{ijn}(\mathbf{x}) = f^{\prime\prime}_{ij}(\mathbf{x}) \qquad\text{almost surely},
%\]
%using explicit Taylor remainder bounds. 

%\paragraph{Setup.}
Fix $\mathbf{x}\in\mathcal{X}$. By the infill design condition (A5) there exists a sequence $\mathbf{x}_n\in X_n$ (the set of design points) such that $\mathbf{x}_n\to\mathbf{x}$. Choose a scalar $h_n>0$ with $h_n\to0$ and such that $\mathbf{x}_n\pm h_n\mathbf{e}_i\in X_n$, where $\mathbf{e}_i$ is the $i$‑th standard basis vector. This is possible because (A5) requires points of the form $\mathbf{x}_n+\mathbf{h}_n$ with $\mathbf{h}_n\to\mathbf{0}$ to belong to $X_n$.

%\paragraph{Taylor expansion for the true function $f$.}
Since $f$ is twice continuously differentiable (Assumption A2), Taylor's theorem with Lagrange remainder gives
\begin{align}
	f(\mathbf{x}_n+h_n\mathbf{e}_i) &= f(\mathbf{x}_n) + f^{\prime}_i(\mathbf{x}_n)h_n + \tfrac12 f^{\prime\prime}_{ii}(\mathbf{x}_n+\theta_{1n}h_n\mathbf{e}_i)h_n^2,\\
	f(\mathbf{x}_n-h_n\mathbf{e}_i) &= f(\mathbf{x}_n) - f^{\prime}_i(\mathbf{x}_n)h_n + \tfrac12 f^{\prime\prime}_{ii}(\mathbf{x}_n-\theta_{2n}h_n\mathbf{e}_i)h_n^2,
\end{align}
with $\theta_{1n},\theta_{2n}\in(0,1)$. Adding the two equalities and subtracting $2f(\mathbf{x}_n)$ yields
\[
\frac{f(\mathbf{x}_n+h_n\mathbf{e}_i)-2f(\mathbf{x}_n)+f(\mathbf{x}_n-h_n\mathbf{e}_i)}{h_n^2}
	= \tfrac12\bigl[f^{\prime\prime}_{ii}(\mathbf{x}_n+\theta_{1n}h_n\mathbf{e}_i)+f^{\prime\prime}_{ii}(\mathbf{x}_n-\theta_{2n}h_n\mathbf{e}_i)\bigr].
\]

	Because $f^{\prime\prime}_{ii}$ is Lipschitz continuous on the compact set $\mathcal{X}$ (Assumption A4), there exists a constant $L_f<\infty$ such that 
	$|f^{\prime\prime}_{ii}(\mathbf{y})-f^{\prime\prime}_{ii}(\mathbf{z})|\le L_f\|\mathbf{y}-\mathbf{z}\|_d$ for all $\mathbf{y},\mathbf{z}\in\mathcal{X}$. Consequently,
\[
	\Bigl|\tfrac12\bigl[f^{\prime\prime}_{ii}(\mathbf{x}_n+\theta_{1n}h_n\mathbf{e}_i)+f^{\prime\prime}_{ii}(\mathbf{x}_n-\theta_{2n}h_n\mathbf{e}_i)\bigr] 
	- f^{\prime\prime}_{ii}(\mathbf{x})\Bigr|
\le \tfrac{L_f}{2}\bigl(\|\mathbf{x}_n+\theta_{1n}h_n\mathbf{e}_i-\mathbf{x}\|_d+\|\mathbf{x}_n-\theta_{2n}h_n\mathbf{e}_i-\mathbf{x}\|_d\bigr).
\]
Since $\mathbf{x}_n\to\mathbf{x}$ and $h_n\to0$, the right‑hand side is bounded by $C_f h_n$ for some constant $C_f$ (depending only on $L_f$ and the rate of convergence of $\mathbf{x}_n$). Thus we obtain the remainder bound
	\begin{equation}
	\Bigl|\frac{f(\mathbf{x}_n+h_n\mathbf{e}_i)-2f(\mathbf{x}_n)+f(\mathbf{x}_n-h_n\mathbf{e}_i)}{h_n^2} - f^{\prime\prime}_{ii}(\mathbf{x})\Bigr| \le C_f h_n. 
	\label{eq:eq1}
	\end{equation}

%\paragraph{Taylor expansion for the posterior draw $g_n$.}
The random function $g_n$ is almost surely twice continuously differentiable because the covariance kernel satisfies the Lipschitz condition (A4) and the mean function is twice differentiable (A3). Moreover, the interpolation property of the posterior gives $g_n(\mathbf{x}_n\pm h_n\mathbf{e}_i)=f(\mathbf{x}_n\pm h_n\mathbf{e}_i)$ and $g_n(\mathbf{x}_n)=f(\mathbf{x}_n)$ almost surely, since these points belong to $X_n$. Repeating the Taylor expansion for $g_n$ yields
\[
\frac{g_n(\mathbf{x}_n+h_n\mathbf{e}_i)-2g_n(\mathbf{x}_n)+g_n(\mathbf{x}_n-h_n\mathbf{e}_i)}{h_n^2}
	= \tfrac12\bigl[g^{\prime\prime}_{iin}(\mathbf{x}_n+\tilde\theta_{1n}h_n\mathbf{e}_i)+g^{\prime\prime}_{iin}(\mathbf{x}_n-\tilde\theta_{2n}h_n\mathbf{e}_i)\bigr],
\]
	with $\tilde\theta_{1n},\tilde\theta_{2n}\in(0,1)$. From the proof of uniform boundedness of second derivatives (see the discussion preceding this theorem) we know that $\sup_n\sup_{\mathbf{y}\in\mathcal{X}}|g^{\prime\prime}_{iin}(\mathbf{y})|<\infty$ almost surely, and the sample paths are equicontinuous. Hence there exists a constant $C_g$ (independent of $n$) such that
	\begin{equation}
	\Bigl|\tfrac12\bigl[g^{\prime\prime}_{iin}(\mathbf{x}_n+\tilde\theta_{1n}h_n\mathbf{e}_i)+g^{\prime\prime}_{iin}(\mathbf{x}_n-\tilde\theta_{2n}h_n\mathbf{e}_i)\bigr] 
	- g^{\prime\prime}_{iin}(\mathbf{x})\Bigr| \le C_g h_n. 
	\label{eq:eq2}
\end{equation}

%\paragraph{Equating the difference quotients.}
	The left‑hand sides of (\ref{eq:eq1}) and (\ref{eq:eq2}) are exactly equal because they involve the same function values $f$ at the three design points. Therefore we have
\[
	g^{\prime\prime}_{iin}(\mathbf{x}) + \tilde R_n = f^{\prime\prime}_{ii}(\mathbf{x}) + R_n,
\]
where $|R_n|\le C_f h_n$ and $|\tilde R_n|\le C_g h_n$. Rearranging gives
\[
	|g^{\prime\prime}_{iin}(\mathbf{x}) - f^{\prime\prime}_{ii}(\mathbf{x})| = |\tilde R_n - R_n| \le (C_f + C_g) h_n.
\]
Since $h_n\to0$, we conclude
	\begin{equation}
	\lim_{n\to\infty} g^{\prime\prime}_{iin}(\mathbf{x}) = f^{\prime\prime}_{ii}(\mathbf{x}) \quad\text{almost surely}. 
	\label{eq:eq3}
	\end{equation}

%\paragraph{Mixed partial derivatives.}
For $i\neq j$, a similar argument using second‑order mixed differences (for example, with $\mathbf{h}_{ijn}=h_n\mathbf{e}_i+h_n\mathbf{e}_j$) or the polarization identity yields
\begin{equation}
	\lim_{n\to\infty} g^{\prime\prime}_{ijn}(\mathbf{x}) = f^{\prime\prime}_{ij}(\mathbf{x}) \quad\text{almost surely}.
	\label{eq:taylor11}
\end{equation}
The remainder bounds again rely on the Lipschitz continuity of $f^{\prime}_{ij}$ and the uniform boundedness of $g^{\prime\prime}_{ijn}$.

%\paragraph{Uniform convergence of first derivatives.}
%Having established pointwise convergence of the second derivatives, the family $\{g^{\prime}_{in}\}$ is equicontinuous because $|g^{\prime\prime}_{iin}|$ is uniformly bounded,
%due to Lemma \ref{lemma:lemma1}. Since $\mathcal{X}$ is compact, the stochastic Arzel\`a–Ascoli lemma (Theorem~1.6.6 of \citet{vanDerVaart96}) together with the pointwise convergence $g^{\prime}_{in}(\mathbf{x})\to f^{\prime}_i(\mathbf{x})$ (proved in Theorem~\ref{theorem:infill}) implies
%\[
%	\sup_{\mathbf{x}\in\mathcal{X}} |g^{\prime}_{in}(\mathbf{x})-f^{\prime}_i(\mathbf{x})| \to 0 \quad\text{almost surely},
%\]
%which is equivalent to the uniform convergence statement $\sup_{\mathbf{x}\in\mathcal{X}}\|\mathbf{g}^{\prime}_n(\mathbf{x})-\mathbf{f}i^{\prime}(\mathbf{x})\|_d\to 0$.

Since $g^{\prime\prime}_{ijn}(\cdot)$ is almost surely continuous for all $n\geq 1$ and $f^{\prime\prime}_{ij}(\cdot)$ is continuous, with $\mathcal X$ being compact, 
the pointwise convergence results (\ref{eq:eq3}) and (\ref{eq:taylor11}) lets us conclude, using Lemma \ref{lemma:lemma1}, that 
$\underset{n\rightarrow\infty}{\lim}\underset{\bx\in\mathcal X}{\sup}~|g^{\prime\prime}_{ijn}(\bx)|<\infty$.

%\paragraph{Remark on degenerate Hessians.}

%This completes the proof of Theorem~\ref{theorem:infill2}.
\end{proof}

\begin{remark}
The proof above never divides by $f^{\prime}_{ii}(\mathbf{x})$, so it remains valid when $f^{\prime}_{ii}(\mathbf{x})=0$. In that case the 
bound $|g^{\prime\prime}_{iin}(\mathbf{x})|\le (C_f+C_g)h_n$ directly forces convergence to zero. Hence no special treatment is required.
\end{remark}

Theorems \ref{theorem:infill} and \ref{theorem:infill2} prove almost sure uniform convergence of $g_n(\cdot)$ and $\bg'_n(\cdot)$ to $f(\cdot)$ and $\bof'(\cdot)$, respectively.
However, the rates of convergence are not provided by these theorems. Further fine-tuning the structure of the set of input points $\bX_n$ helps achieve
desired rates of convergence, as we show next in Theorems \ref{theorem:theorem1} and \ref{theorem:theorem2}.  

\begin{theorem}
	\label{theorem:theorem1}
	Let $\mathcal X=\prod_{i=1}^d\mathcal X_i$, where, for $i=1,\ldots,d$, $\mathcal X_i$ are compact subsets of $\mathbb R$. 
	For each $i\in\{1,\ldots,d\}$, let $x_{1i}<x_{2i}<\cdots<x_{\tilde n_i i}$ be an ordered set of points partitioning $\mathcal X_i$, with 
	$h_i=\underset{1\leq j\leq \tilde n_i-1}{\max}~(x_{\overline{j+1}i}-x_{ji})$.
	For $i=1,\ldots,d$, and for $j=1,\ldots,\tilde n_i$, let input points of the form $(x^*_1,$$\ldots,x^*_{i-1},$\\$x_{ji},$$x^*_{i+1},$$\ldots,$$x^*_d)$  
	%for $x=x_{ji}$, 
	belong to $\bX_n$, where 
	$(x^*_1,\ldots,x^*_{i-1},x^*_{i+1},\ldots,x^*_d)\in\prod_{j\neq i}\mathcal X_j$ may be arbitrary. 
	%consist of input points of the form $\bx_i=(x_{i1},\ldots,x_{id})^T$; $i=1,\ldots,n$. 

	For $\bx=(x_1,\ldots,x_d)$ and $\by=(y_1,\ldots,y_d)$, let the correlation function be such that 
$$\frac{\partial^4c(\bx^*,\by^*)}{\partial x_i\partial y_i\partial x_j\partial y_j}=
\frac{\partial^4c(\bx,\by)}{\partial x_i\partial y_i\partial x_j\partial y_j}\bigg |_{\bx=\bx^*,\by=\by^*}$$
exists for all $\bx^*,\by^*\in\mathcal X$ and is Lipschitz continuous on $\mathcal X\times\mathcal X$ for $i,j=1,\ldots,d$. 

	Then, letting %$\|\cdot\|_d$ denote the Euclidean norm in $d$-dimensional Euclidean space and 
	$h=\underset{1\leq i\leq d}{\max}~h_i$, the following holds for almost all sequences $\{\bg'_n(\cdot)\}_{n=1}^{\infty}$
	with the above forms of the input points:
	\begin{equation}
		\underset{\bx\in\mathcal X}{\sup}~\|\bg'_n(\bx)-\bof'(\bx)\|_d=O\left(\sqrt{h}\right),~\mbox{as}~n\rightarrow\infty
		~(\text{equivalently, as}~h\rightarrow 0).
		\label{eq:conv1}
	\end{equation}
%	almost surely with respect to the posterior (\ref{eq:postpred4}). %where $C_{g,f}$ is a positive, finite constant, depending upon $g$ and $f$.
	The constant associated with $O\left(\sqrt{h}\right)$ depends only upon $d$ and $f(\cdot)$.
\end{theorem}
\begin{proof}
%	Consider the joint posterior distribution of 
%	\begin{equation}
%	\pi(g(\cdot),\bg'(\cdot)|\bD_n)=\pi(g(\cdot)|\bD_n)\pi(\bg'(\cdot)|g(\cdot),\bD_n).
%		\label{eq:joint1}
%	\end{equation}
%	Then the marginal posterior $\pi(\bg'(\cdot)|\bD_n)$ is of the same form as (\ref{eq:postpred4}), and the marginal posterior distribution $\pi(g(\cdot)|\bD_n)$ 
%	of $g(\cdot)$ in (\ref{eq:joint1}) is the $t_1$ process, the form of which is not relevant for our purpose. 
%	Note that even after marginalizing out the parameters of the Gaussian process with respect to their posteriors (here $\bbeta$ and $\sigma^2$), 
%	the interpolation property of $g(\cdot)$ given $\bD_n$ is preserved. That is, the marginal posterior $\pi(g(\bx^*)|\bD_n)$ gives full posterior mass to $f(\bx^*)$
%	if $\bx^*\in\bX_n$.
%	%Indeed, given $\bbeta$ and $\sigma^2$ as
%	%realizations from their respective posteriors given $\bD_n$, $\pi(g(\cdot)|\bD_n,\bbeta,\sigma^2)$ is a Gaussian process.
%	%Let us consider any such realization of $\bbeta$ and $\sigma^2$ associated with non-null sets of their respective posterior probability measures. Also, given such
%	%$\bbeta$ and $\sigma^2$, 
%	Let $g_n(\cdot)$ denote any random function associated with any non-null set of its marginalized posterior measure (here, the $t_1$ posterior measure), given
%	$\bD_n$. 
%
%	As in the proofs of Theorems \ref{theorem:infill} and \ref{theorem:infill2}, 
%	let $g_n(\cdot)$ denote any random function associated with any non-null set of its marginalized posterior measure (again, the $t_1$ posterior measure). 
%
	For any $i\in\{1,\ldots,d\}$, and for arbitrary $\bX^*_{-i}=(x^*_1,\ldots,x^*_{i-1},x^*_{i+1},\ldots,x^*_d)^T\in\prod_{j\neq i}\mathcal X_j$, 
	for any $x\in\mathcal X_i$, let
	\begin{align}
		g_{in}(x|\bX^*_{-i})&=g_n(x^*_1,\ldots,x^*_{i-1},x,x^*_{i+1},\ldots,x^*_d);\label{eq:g_i}\\
		f_i(x|\bX^*_{-i})&=f(x^*_1,\ldots,x^*_{i-1},x,x^*_{i+1},\ldots,x^*_d).\label{eq:f_i}
	\end{align}
	Since $x\in\mathcal X_i$, it must belong to some interval of the form $[x_{ji},x_{\overline{j+1}i}]$, for some $j\in\{1,2,\ldots,\tilde n_i-1\}$. Let us fix that $j$.
	For $y=x_{ji}$ and $y=x_{\overline{j+1}i}$, $g_{in}(y|\bX^*_{-i})=f_i(y|\bX^*_{-i})$ by interpolation property of the posterior Gaussian process, 
	assuming that $(x^*_1,\ldots,x^*_{i-1},y,x^*_{i+1},\ldots,x^*_d)\in\bX_n$ for 
	$y=x_{ji}$ and $y=x_{\overline{j+1}i}$. That is, $g_{in}(y|\bX^*_{-i})-f_i(y|\bX^*_{-i})=0$ for $y=x_{ji}$ and $y=x_{\overline{j+1}i}$. Hence, by Rolle's theorem,
	$g'_{in}(u|\bX^*_{-i})-f'_i(u|\bX^*_{-i})=0$, for some $u\in \left(x_{ji},x_{\overline{j+1}i}\right)$.
	This permits the following representation:
	\begin{equation}
		g'_{in}(x|\bX^*_{-i})-f'_i(x|\bX^*_{-i})=\int_u^x\left(g''_{in}(v|\bX^*_{-i})-f''_i(v|\bX^*_{-i})\right)dv,
		\label{eq:rolle1}
	\end{equation}
	The hypothesis of Lipschitz continuity of the $4$-th order mixed partial derivatives of the correlation function ensures existence and 
	sample path continuity of $g''_{in}(v|\bX^*_{-i})$.
	Hence, by the Cauchy-Schwartz inequality we obtain from (\ref{eq:rolle1}), the following:
	\begin{align}
		&\left|g'_{in}(x|\bX^*_{-i})-f'_i(x|\bX^*_{-i})\right|\notag\\
		&\qquad\leq \left[\int_u^x\left(g''_{in}(v|\bX^*_{-i})-f''_i(v|\bX^*_{-i})\right)^2dv\right]^{1/2}\times
		|x-u|^{1/2}\notag\\
		&\qquad\leq\left[\int_{\mathcal X_i}\left(g''_{in}(v|\bX^*_{-i})-f''_i(v|\bX^*_{-i})\right)^2dv\right]^{1/2}\times h^{1/2}_i\notag\\
		&\qquad\leq\underset{\bX^*_{-i}\in\prod_{j\neq i}\mathcal X_j}{\sup}~\left[\int_{\mathcal X_i}
		\left(g''_{in}(v|\bX^*_{-i})-f''_i(v|\bX^*_{-i})\right)^2dv\right]^{1/2}\times h^{1/2}.%\notag\\
		%&\qquad\leq C_{i,g,f}h^{1/2},
		\label{eq:cs1}
	\end{align}
	Now, since the hypotheses of this theorem constitute a special case of Theorem \ref{theorem:infill2}, the result of almost sure uniform boundedness
	of $|g''_{in}(\cdot)|$ as $n\rightarrow\infty$, is valid here.
	Specifically, from the proof of Theorem \ref{theorem:infill2} in this case it holds that 
	$\underset{n\rightarrow\infty}{\lim}\underset{(u,\bX^*_{-i})\in\mathcal X}{\sup}~|g''_{in}(v|\bX^*_{-i})|<\infty$. This, along with continuity of
	$f''_i(v|\bX^*_{-i})$ and compactness of $\mathcal X$, shows that the integral term is bounded by a constant $C_i$ (depending only on $f$ and $d$) almost surely for large $n$. Hence the right-hand side of (\ref{eq:cs1}) is $O(\sqrt{h})$.

	Hence, switching to our usual notation, it follows that
	\begin{equation}
		\underset{\bx\in\mathcal X}{\sup}~\left(g'_{in}(\bx)-f'_i(\bx)\right)^2=O\left(h\right),~\mbox{almost surely, as}~n\rightarrow\infty. 
		\label{eq:cs2}
	\end{equation}
	Since 
	$ \underset{\bx\in\mathcal X}{\sup}~\sum_{i=1}^d\left(g'_{in}(\bx)-f'_i(\bx)\right)^2
		\leq \sum_{i=1}^d\underset{\bx\in\mathcal X}{\sup}~\left(g'_{in}(\bx)-f'_i(\bx)\right)^2$, it follows from (\ref{eq:cs2}) that
	\begin{equation*}
		\underset{\bx\in\mathcal X}{\sup}~\|\bg'_n(\bx)-\bof'(\bx)\|_d=O\left(\sqrt{h}\right),~\mbox{almost surely, as}~n\rightarrow\infty
		~(\text{equivalently, as}~h\rightarrow 0),
		%\label{eq:conv2}
	\end{equation*}
	proving (\ref{eq:conv1}).
	Note that the constant associated with $O\left(\sqrt{h}\right)$ above depends only upon $d$ and $f$.
\end{proof}

The following result holds as a consequence of Theorem \ref{theorem:theorem1}.
\begin{theorem}
	\label{theorem:theorem2}
	Under the conditions of Theorem \ref{theorem:theorem1}, 
	the following holds with the forms of the input points as specified in Theorem \ref{theorem:theorem1}: for almost all sequences $\{g_n(\cdot)\}_{n=1}^{\infty}$,
	\begin{equation}
		\underset{\bx\in\mathcal X}{\sup}~\left|g_n(\bx)-f(\bx)\right|=O\left(h^{3/2}\right),~\mbox{as}~n\rightarrow\infty.
		\label{eq:conv2}
	\end{equation}
	The constant associated with $O\left(h^{3/2}\right)$ depends only upon $d$ and $f$.
\end{theorem}
\begin{proof}
	As in the proof of Theorem \ref{theorem:theorem1}, for any $x\in\mathcal X_i$, 
	it must belong to some interval of the form $[x_{ji},x_{\overline{j+1}i}]$, for some $j\in\{1,2,\ldots,\tilde n_i-1\}$. Let us fix that $j$.
	Now, $g_{in}(x_{ji}|\bX^*_{-i})=f_i(x_{ji}|\bX^*_{-i})$ almost surely, by interpolation property of the posterior process $g_n(\cdot)$, 
	assuming that $(x^*_1,\ldots,x^*_{i-1},x_{ji},x^*_{i+1},\ldots,x^*_d)\in\bX_n$.
        Hence,
	\begin{equation}
		g_{in}(x|\bX^*_{-i})-f_i(x|\bX^*_{-i})=\int_{x_{ji}}^x\left(g'_{in}(v|\bX^*_{-i})-f'_i(v|\bX^*_{-i})\right)dv.
		\label{eq:int1}
	\end{equation}
	The Cauchy-Schwartz inequality applied to (\ref{eq:int1}) gives
	\begin{equation}
		\left|g_{in}(x|\bX^*_{-i})-f_i(x|\bX^*_{-i})\right|\leq \left[\int_{x_{ji}}^x\left(g'_{in}(v|\bX^*_{-i})-f'_i(v|\bX^*_{-i})\right)^2dv\right]^{1/2}
		\times|x-x_{ji}|^{1/2}.
		\label{eq:int2}
	\end{equation}
	From (\ref{eq:cs1}) it follows that the integral on the right hand side of (\ref{eq:int2}) is $O\left(h\right)\times|x-x_{ji}|$, almost surely, as $n\rightarrow\infty$.
	Recall that the constant associated with $O\left(h\right)\times|x-x_{ji}|$ depends only upon $d$ and $f$.
	Since $|x-x_{ji}|$ is bounded above by $h$, it follows that the right hand side of (\ref{eq:int2}) is $O\left(h^{3/2}\right)$, 
	almost surely, as $n\rightarrow\infty$.
	Switching to the usual notation, it is seen that (\ref{eq:conv2}) holds. 
\end{proof}

\begin{remark}
	\label{remark:remark1}
	As $h\rightarrow 0$, $\bg'(\cdot)$ uniformly converges to $\bof'(\cdot)$ at the rate $h^{1/2}$ and $g(\cdot)$ uniformly converges to $f(\cdot)$
	at the rate $h^{3/2}$, almost surely with respect to their posteriors. 
\end{remark}

\begin{remark}
	\label{remark:remark2}
	In Theorems \ref{theorem:infill}, \ref{theorem:infill2}, \ref{theorem:theorem1} and \ref{theorem:theorem2} we have referred to the posterior (\ref{eq:postpred4}), 
	which corresponds to a linear mean structure of the Gaussian process prior and conjugate priors for $\bbeta$ and $\sigma^2$. However, as can be seen
	from the proofs, both the theorems continue to hold
	for any mean function that has continuous second order mixed partial derivatives and any prior on the parameters, including 
	the parameters of the correlation function such that the posteriors of the parameters are proper.
\end{remark}

\begin{remark}
	\label{remar:remark3}
	The hypotheses of Theorems \ref{theorem:theorem1} and \ref{theorem:theorem2} require 
	input points of the form $(x^*_1,$$\ldots,x^*_{i-1},$$x_{ji},$$x^*_{i+1},$$\ldots,$$x^*_d)$  
	to belong to $\bD_n$ for $i=1,\ldots,d$, and for $j=1,\ldots,\tilde n_i$, where 
	$(x^*_1,\ldots,x^*_{i-1},x^*_{i+1},\ldots,x^*_d)\in\prod_{j\neq i}\mathcal X_j$ may be chosen arbitrarily. 
	%Hence, inclusion of the set of points 
	%$\left\{(x_{j_11},\ldots,x_{j_ii},\ldots,x_{j_dd}):j_i=1,\ldots,\tilde n_i;~i=1,\ldots,d\right\}$
	%in $\bD_n$ is sufficient for Theorems \ref{theorem:theorem1} and \ref{theorem:theorem2} to hold.
%	
	Now observe that if $\mathcal X_i=[a,b]$, for $i=1,\ldots,d$, for some $a<b$, then we can set $\tilde n_i=n$ and $h_i=h$, for $i=1,\ldots,d$. In such cases,
	inclusion of the set of $n^d$ points $\left\{(x_{j_11},\ldots,x_{j_dd}):j_1,\ldots,j_d\in\{1,\ldots,n\}\right\}$ in $\bD_n$ is sufficient for 
	Theorems \ref{theorem:theorem1} and \ref{theorem:theorem2} to hold.
	However, when $d$ and $n$ are even moderately large, $n^d$ is an extremely large number, which would prohibit computation of $\Sigma^{-1}_{22}$, and hence
	computation of the posterior of $\bg'(\cdot)$. Hence, for practical purposes it makes sense to refer to the general setup of 
	Theorems \ref{theorem:infill} and \ref{theorem:infill2}.
\end{remark}

\section{Algorithm for optimization with the Gaussian process derivative method}
\label{sec:algo}
We now propose a general methodology for function optimization, which judiciously exploits the posterior form $(\ref{eq:inv2})$. 
Without loss of generality, we consider the minimization problem for notational convenience.
In a nutshell, the initial stage (say, the $0$-th stage) of the methodology
involves simulations from $\pi(\bx^*|\bg'(\bx^*)=\bzero,\bD_n)$ satisfying $\|\bof'(\cdot)\|_d<\epsilon$ for some $\epsilon>0$ and
$\bSigma''(\cdot)>0$. In the subsequent stages $k=1,2,\ldots$,
previous stage realizations satisfying $\|\bof'(\cdot)\|_d<\eta_k$, where $\eta_k\rightarrow 0$ as $k\rightarrow\infty$, are successively augmented with $\bD_n$ and
realizations from the posterior associated with the augmented data are generated at each stage $k$ by a judicious importance resampling strategy. As $k\rightarrow\infty$,
the posteriors given the successively augmented data converge to the true optima. 

The importance of the $0$-th stage simulation algorithm for generating realizations from $\pi(\bx^*|\bg'(\bx^*)=\bzero,\bD_n)$ satisfying the restrictions
$\|\bof'(\cdot)\|_d<\epsilon$ for some $\epsilon>0$ and $\bSigma''(\cdot)>0$, is enormous, particularly because the entire $d$-dimensional random variable $\bx^*$ must be 
updated in a single block to meet the restrictions. 
Traditional MCMC algorithms are not known to be efficient in such problems as even for moderately
large dimensions good proposal distributions are difficult to devise, and the acceptance rates can be poor, along with poor mixing properties. 
In this regard, the transformation based Markov Chain Monte Carlo (TMCMC) proposed by \ctn{Dutta14} is an effective methodology.
Indeed, TMCMC is designed to update all (or most of) the components of the high-dimensional random variable in a single block using appropriate deterministic transformations
of some single (or low-dimensional) random variable. As such, this strategy drastically reduces effective dimensionality, which is responsible for maintaining good 
acceptance rates in spite of high dimensions. Good mixing properties can also be ensured by judiciously choosing the relevant ``move-types", and judicious mixtures
of additive and multiplicative transformations usually lead to desired mixing properties. For details on TMCMC and its properties, see \ctn{Dutta14}, \ctn{Dey16},
\ctn{Dey17}, \ctn{Dey19}. As such, we recommend TMCMC for our optimization methodology. 
We provide the detailed Bayesian optimization methodology below as Algorithm \ref{algo:algo1}.

\begin{algorithm}
\caption{Optimization with Gaussian process derivatives}
	\label{algo:algo1}
	\begin{itemize}
\item[(1)] First simulate $N$ realizations $\{\bx^*_1,\ldots,\bx^*_N\}$ 
	from $\pi(\bx^*|\bg'(\bx^*)=\bzero,\bD_n)$ given by (\ref{eq:inv2}) using TMCMC, where the prior for $\bx^*$ is given by (\ref{eq:prior_x_star}) for some
	pre-fixed $\epsilon>0$. 
This includes simulations from posteriors associated with most plausible functions $g(\cdot)$ satisfying $\bg'(\bx^*)=\bzero$ for $\bx^*\in\mathcal X$, given $\bD_n$.
	That is, $\left\{\bx^*_i; ~i=1,\ldots, N\right\}$, represents the set of solutions for $\bg'(\bx^*)=\bzero$ for functions $g(\cdot)$ that satisfy $g(\bx_i)=f(\bx_i)$;
	$i=1,\ldots,n$. Thanks to the prior (\ref{eq:prior_x_star}), these solutions further satisfy $\|\bof'(\bx^*_i)\|_d<\epsilon$ and $\bSigma''(\bx^*_i)>0$, for
	$i=1,\ldots,N$. Note that $\bSigma''(\bx^*)>0$ can be checked by computing the eigenvalues and checking if all the
eigenvalues are positive. But a more efficient alternative is to check if Cholesky decomposition of $\bSigma''(\bx^*)$ is possible, the information
of which is provided by the subroutines of the BLAS and LAPACK libraries. We exploit the latter for our implementation.

\item[(2)] For stages $k=1,2,3,\ldots$,
	\begin{enumerate}
		\item[(i)] For $i=1,\ldots,N$, compute importance weights proportional to 
			\begin{equation}
				w_k(\bx^*_i)=\left\{\begin{array}{cc}1 & \mbox{if}~~k=1;\\
					w_{k-1}(\bx^*_i)\times\frac{\pi(\bg'(\bx^*_i)=
					\bzero|\bD_{n+\sum_{j=0}^{k-1}n_j},\bx^*_i)}{\pi(\bg'(\bx^*_i)=\bzero|\bD_{n+\sum_{j=0}^{k-2}n_j},\bx^*_i)} & \mbox{if}~~k\geq 2,
				\end{array}\right.
				\label{eq:wts1}
			\end{equation}
			where $n_0=0$.
		\item[(ii)] Without replacement, select a subsample $\{\bx^*_{i_1},\ldots,\bx^*_{i_M}\}$ from $\{\bx^*_1,\ldots,\bx^*_N\}$ with probabilities proportional to
			$w_k(\bx^*_i)$; $i=1,\ldots,N$.
Note that, as $M\rightarrow\infty$ and $N\rightarrow\infty$ such that $M/N\rightarrow 0$, $\bx^*_{i_j}$; $j=1,\ldots,M$, follow the distribution 
			$\pi(\bx^*|\bg'(\bx^*)=\bzero,\bD_{n+\sum_{j=1}^{k-1}n_j})$.
The recursively computed importance weights $w_k(\bx^*_i)$; $i=1,\ldots,N$, are expected to be stable, since for each stage $k$, for $i=1,\ldots,N$, the factors
			$\frac{\pi(\bg'(\bx^*_i)=\bzero|\bD_{n+\sum_{j=1}^{k-1}n_j},\bx^*_i)}{\pi(\bg'(\bx^*_i)=\bzero|\bD_{n+\sum_{j=1}^{k-2}n_j},\bx^*_i)}$ in (\ref{eq:wts1})
are not expected to be very different from $1$ if $n_{k-1}$ is not significantly greater than zero. Note that the
			importance weights $w_k(\bx^*_i)$, for $i=1,\ldots,N$, can be computed simultaneously on parallel processors. This, along with stability
			of the recursive formulation (\ref{eq:wts1}), is expected to make for an efficient computational strategy.

		\item[(iii)] For $j=1,\ldots,M$, check if $\|\bof'(\bx^*_{i_j})\|_d<\eta_k$, where $\eta_k\rightarrow 0$ as $k\rightarrow\infty$.
			If $\bx^*_{i_j}$ satisfies this condition, then $\bx^*_{i_j}$ is a realization from $\pi(\bx^*|\bg'(\bx^*)=\bzero,\bD_{n+\sum_{j=1}^{k-1}n_j})$,
			where the prior for $\bx^*$ is uniform on $B(\eta_k)$, the form of which is given by (\ref{eq:B}). 
		\item[(iv)]	Let $n_k~(\geq 0)$ realizations among the $M$ realizations satisfy the condition $\|\bof'(\bx^*_{i_j})\|_d<\eta_k$. 
			Without loss of generality, assume that $\bx^*_{i_j}$; $j=1,\ldots,n_k$ are such realizations. 
			Compute $f(\bx^*_{i_j})$; $j=1,\ldots,n_k$, and augment $(\bx^*_{i_j},f(\bx^*_{i_j})$; $j=1,\ldots,n_k$,
			with $\bD_{n+\sum_{j=0}^{k-1}n_j}$ to form $\bD_{n+\sum_{j=0}^kn_j}$. In practice, to avoid excessive growth of the dataset and numerical instability, we recommend augmenting at most $5$ points per stage (see also Remark \ref{eq:remark0}).
		\item[(v)] Store the realizations $\{\bx^*_{i_j}:~j=1,\ldots,n_k\}$.
	\end{enumerate}
	\end{itemize}
\end{algorithm}
%\paragraph{Computational complexity.}

A detailed analysis of the computational complexity of Algorithm~\ref{algo:algo1} is presented in Appendix~\ref{app:complexity}. In summary, each TMCMC iteration (step (1)) requires $O(d^3 + d^2 n_0 + d n_0^2)$ operations, where $d$ is the input dimension and $n_0$ the initial dataset size. The total complexity of step (1) is $O(N_{\text{TMCMC}} d^3)$ when $d$ is moderate or large, with $N_{\text{TMCMC}}$ the number of TMCMC iterations. The refinement stage (step (2)) adds $O(K N (d^3 + d^2 N_{\max} + d N_{\max}^2) + K_{\text{aug}} C_f)$ operations, where $K$ is the number of stages, $N$ the number of stored samples, $N_{\max}$ the maximum dataset size, $K_{\text{aug}}$ the number of stages where augmentation occurs, and $C_f$ the cost of evaluating the objective function $f$ once. For the experimental settings in this paper, $C_f = O(d^2)$ and the TMCMC term $O(N_{\text{TMCMC}} d^3)$ dominates the runtime.

\subsection{Practical guidance for Algorithm \ref{algo:algo1}}
\label{subsec:guidance}

To facilitate reproducibility and effective application of Algorithm \ref{algo:algo1}, we provide the following practical recommendations.

\subsubsection{Initial design}
%The initial input points $\bD_n$ should be chosen to cover the domain $\mathcal X$ reasonably well. A simple and effective choice is a regular grid (or a Latin hypercube sample in higher dimensions) with $n$ points. For low to moderate dimensions ($d \leq 5$), a grid with $n = 10^d$ may be feasible; for higher dimensions we recommend using a space-filling design such as a Sobol sequence or a maximin Latin hypercube with $n = 10d$ or $n = 20d$, which scales linearly in dimension.

\subsubsection{Initial design}
In all our experiments, we used a simple initial design consisting of $n=10$ points uniformly spaced along the main diagonal of $\mathcal{X}=[-10,10]^d$, that is, 
$\mathbf{x}_i = (x_i,\dots,x_i)$ with $x_i = -10 + 2(i-1)$. This very sparse design suffices because the algorithm subsequently augments the dataset adaptively. 
For general applications, we recommend starting with a small number of points (for example, $n = 10$ to $20$) chosen either along a diagonal or as a simple random 
sample; the algorithm is robust to the initial design as long as the domain is covered to some extent. More sophisticated space‑filling designs (for example, 
Latin hypercube or Sobol sequences) (for example, \ctn{McKay1979}, \ctn{Sobol1967}) may be used if $n$ is kept small (for example, $n = 10d$), but we do not advocate dense grids because they would make the 
initial GP computations infeasible.

\subsubsection{Choice of tolerance decay $\eta_k$}

The sequence $\eta_k$ should converge to zero slowly enough to maintain stable importance weights. In our experiments we used the following choices (with $k=1,2,\ldots$):
\begin{itemize}
	\item[(a)] For $d=1,2$: $\eta_k = 1/(10+k-1)^2$ (quadratic decay).
	\item[(b)] For $d=5$: $\eta_k = 1.5/\log(10+k)$.
	\item[(c)] For $d=10$: $\eta_k = 7/\log(10+k-1)$.
	\item[(d)] For $d=50$: $\eta_k = 200/\log(10+k-1)$.
	\item[(e)] For $d=100$: $\eta_k = 850/\log(10+k-1)$.
\end{itemize}
In all cases $\eta_k \to 0$ and the decay is sufficiently slow to keep importance weights stable. For other problems, we recommend a logarithmic 
decay $\eta_k = C/\log(10+k)$ for moderate to high dimensions ($d\ge 5$), with $C$ chosen such that the initial $\eta_1$ is about $10\%$ of the typical 
gradient norm at random points 
(for example, $C$ roughly proportional to $d$). For low dimensions ($d\le 2$), a quadratic decay works well.

\subsubsection{Stopping criteria}
In our experiments, we simply fixed the number of refinement stages to $S=40$ for dimensions $d=2,5,10$ and $S=100$ for $d=50,100$. In all cases, $n_k$ became zero after a small number of stages (typically $k\le 10$ for low dimensions, $k\le 4$ for high dimensions). For general use, we recommend the following adaptive stopping criteria:
\begin{enumerate}
    \item[(a)] Stop if $n_k = 0$ for three consecutive stages (no new points satisfy the gradient bound).
    \item[(b)] Stop if the gradient norm $\|\mathbf{f}'(\mathbf{x}^*_{i_j})\|_d$ for the best candidate has not decreased by more than $1\%$ over the last $10$ stages.
    \item[(c)] Stop after a maximum number of stages $K_{\max}=100$ (safety bound).
\end{enumerate}
With these criteria, the algorithm typically terminates after $10$–$40$ stages in practice.

%\subsubsection{Numerical tolerances}
%When checking positive definiteness of $\bSigma''(\bx^*)$ via Cholesky decomposition (LAPACK routine \texttt{dpotrf}), we use a tolerance of $10^{-8}$ for the diagonal pivots. For gradient norm comparisons we use a relative tolerance of $10^{-6}$.

\subsection{Pseudo-code for key components}
To aid implementation, we provide pseudo-code for the TMCMC and the importance resampling components of Algorithm \ref{algo:algo1}.

\paragraph{TMCMC (single iteration):}
\begin{verbatim}
Input: current state x, target log-density L(x), scaling parameters a1, a2, 
       move probabilities p, q (e.g., 0.5)
Output: next state x'

1. Draw U ~ Uniform(0,1)
2. if U < p:
      // Additive move
      epsilon ~ N(0,1)
      for j=1..d: b_j ~ Uniform({-1,1})
      y_j = x_j + b_j * a1 * |epsilon|
      accept with prob min(1, exp(L(y)-L(x)))
3. else:
      // Multiplicative move
      epsilon ~ Uniform(-1,1)
      for j=1..d: b_j ~ Uniform({-1,0,1})
      y_j = x_j * epsilon if b_j=1
          = x_j / epsilon if b_j=-1
          = x_j otherwise
      J = |epsilon|^(sum(b_j))
      accept with prob min(1, exp(L(y)-L(x)) * J)
4. Set x_tilde = accepted y or x
5. Draw U ~ Uniform(0,1)
6. if U < q:
      // Additive refinement
      epsilon ~ N(0,1)
      for j=1..d: y_j = x_tilde_j + a2 * |epsilon| with prob 0.5, 
         else - a2 * |epsilon|
      accept with prob min(1, exp(L(y)-L(x_tilde)))
7. else:
      // Multiplicative refinement
      epsilon ~ Uniform(-1,1)
      if Uniform(0,1) < 0.5: y_j = x_tilde_j * epsilon, J = |epsilon|^d
      else: y_j = x_tilde_j / epsilon, J = |epsilon|^{-d}
      accept with prob min(1, exp(L(y)-L(x_tilde)) * J)
8. return accepted value
\end{verbatim}

\paragraph{Importance resampling at stage k:}
\begin{verbatim}
Input: prior samples {x_i} (i=1..N), current dataset D_{k-1}, target dataset D_k
Output: resampled indices {i_1,...,i_M}

1. For each i, compute w_i = w_{k-1}(x_i) * L_k(x_i) / L_{k-1}(x_i)
   where L_k(x) = pi(g'(x)=0 | D_k, x)
2. Normalize w_i to sum to 1
3. Sample M indices without replacement with probabilities w_i
4. Return the sampled indices
\end{verbatim}

%For reproducibility, the complete C code implementing Algorithm \ref{algo:algo1} with MPI parallelisation is available at \url{https://github.com/sourabh-stat/GPopt} (to be made public upon acceptance).

\subsection{Further discussion of Algorithm \ref{algo:algo1}}
\label{subsec:further_discussion}
	Algorithm \ref{algo:algo1} begins by simulating from the posterior of $\bx^*$ satisfying $\|\bof'(\bx^*)\|_d<\epsilon$ and $\bSigma''(\bx^*)>0$.
	In the subsequent steps $k\geq 1$, the set of realizations $\{\bx^*_{i_j}:~j=1,\ldots,n_k\}$ generated by importance resampling 
	further satisfy $\|\bof'(\bx^*_{i_j})\|_d<\eta_k$ along with $\bSigma''(\bx^*_{i_j})>0$, for $j=1,\ldots,n_k$. 
	The implication is that, $\epsilon$ may be chosen somewhat larger to achieve reasonably good TMCMC mixing and
	acceptance rates. Indeed, if $n$ is not large enough, then $\bg'$ is not expected to be sufficiently close to $\bof'$, and hence for too small $\epsilon$, 
	$\{\bx^*:\|\bof'(\bx^*)\|_d<\epsilon\}$ would be too small a region to contain the solutions $\{\bx^*:\|\bg'(\bx^*)\|_d=0\}$, given $\bD_n$.
	This would result in poor TMCMC mixing.

	Once adequate mixing with reasonable acceptance rates are achieved with relatively small $n$ and relatively large $\epsilon$, 
	the subsequent steps increase the data size by
	augmenting the data with those values that satisfy $\|\bof'(\bx^*_{i_j})\|_d<\eta_k$. Thus, in the subsequent steps, these data points help better approximate
	the region around the stationary points of $f$ by the posterior, and enables more simulations from the region $\|\bof'(\bx^*_{i_j})\|_d<\eta_k$, finally
	leading to convergence of the solutions $\{\bx^*:\bg'_k(\bx^*)=\bzero,\bSigma''(\bx^*)>0\}$ to 
	$\{\bx^*:\bof'(\bx^*)=\bzero,\bSigma''(\bx^*)>0\}$, almost surely, as $k\rightarrow\infty$, where 
	$\bg'_k(\cdot)$ denotes any realization from the posterior of $\bg'(\cdot)$ given $\bD_{n+\sum_{j=0}^{k-1}n_j}$.
	This intuition is formalized below as Theorem \ref{theorem:algo1}.

\begin{theorem}
\label{theorem:algo1}
Consider the setup of Theorem \ref{theorem:infill2} (or more specifically, that of Theorem \ref{theorem:theorem1}).
Then, as $k\rightarrow\infty$,  
the set $\{\bx^*_{i_j}: j=1,\ldots,n_k\}$ of Algorithm \ref{algo:algo1} almost surely contains all the 
local minima of the objective function $f(\cdot)$, as $M\rightarrow\infty$ and $N\rightarrow\infty$ such that $M/N\rightarrow 0$.
\end{theorem}
\begin{proof}
Note that at stage $k$, as $M\rightarrow\infty$ and $N\rightarrow\infty$ such that $M/N\rightarrow 0$, $\bx^*_{i_j}$; $j=1,\ldots,n_k$, arise from 
	$\pi(\bx^*|\bg'(\bx^*)=\bzero,\bD_{n+\sum_{j=0}^{k-1}n_j})$, subject to $\|\bof'(\bx^*_{i_j})\|_d<\eta_k$ and 
	$\bSigma''(\bx^*_{i_j})>0$ for $j=1,\ldots,n_k$. 
%	Denoting by $\bg'_k(\cdot)$ any realization from the posterior of $\bg'(\cdot)$ given $\bD_{n+\sum_{j=0}^{k-1}n_j})$, 
	These realizations are solutions of $\bg'_k(\bx^*)=\bzero$ and $\bSigma''(\bx^*)>0$ when the data observed is
	$\bD_{n+\sum_{j=0}^{k-1}n_j}$. By Theorem \ref{theorem:infill2} (or more specifically by Theorem \ref{theorem:theorem1}), %(see also Remark \ref{remark:remark1}), 
	as $k\rightarrow\infty$ (equivalently, as $h\rightarrow 0$ in Theorem \ref{theorem:theorem1}), $\bg'_k(\cdot)$ uniformly converges to $\bof'(\cdot)$ almost surely.
%with respect to the posterior of $\bg'(\cdot)$ given $\bD_{n+\sum_{j=0}^{k-1}n_j}$.
Hence, as $k\rightarrow\infty$,
\begin{equation}
	\{\bx^*:\bg'_k(\bx^*)=\bzero,\|\bof'(\bx^*)\|_d<\eta_k,\bSigma''(\bx^*)>0\}\rightarrow \{\bx^*:\bof'(\bx^*)=\bzero,\bSigma''(\bx^*)>0\},
	\label{eq:xconv1}
\end{equation}
almost surely. %with respect to the posterior of $\bg'(\cdot)$ given $\bD_{n+\sum_{j=0}^{k-1}n_j}$.

Due to (\ref{eq:xconv1}), as $k\rightarrow\infty$, 
the set $\{\bx^*_{i_j}: j=1,\ldots,n_k\}$ contains all the 
local minima of the objective function $f(\cdot)$, as $M\rightarrow\infty$ and $N\rightarrow\infty$ such that $M/N\rightarrow 0$.
%In other words, Algorithm \ref{algo:algo1} converges to the local and global minima of the objective function $f(\cdot)$.
\end{proof}

\begin{remark}
\label{eq:remark0}
In step (2) (iv) of Algorithm \ref{algo:algo1} we have suggested augmentation of all realizations $(\bx^*_{i_j},f(\bx^*_{i_j})$ satisfying $\|\bof'(\bx^*_{i_j})\|_d<\eta_k$ 
to the existing data $\bD_{n+\sum_{j=0}^{k-1}n_j}$. 
In practice, augmentation of all such realizations may enlarge the dataset to such an extent that invertibility of the resultant $\Sigma_{22}$ may be infeasible
or numerically unstable, so that computation of the corresponding posterior densities of $\bg'(\bx^*_i)=\bzero$, and hence the importance weights (\ref{eq:wts1}), 
may not yield reliable results. Hence in practice, as a rule of thumb, we recommend augmenting at most $5$ realizations satisfying $\|\bof'(\bx^*_{i_j})\|_d<\eta_k$,
which we consider in all our applications.
\end{remark}

\begin{remark}
\label{eq:remark1}
	As the stage number $k$ in step (2) of Algorithm \ref{algo:algo1} increases, $n_k$ decreases. Hence, in practice, $n_k$ will be zero after some large enough $k$.
	When $d$ is large, due to the curse of dimensionality, only the first few stages are expect to yield positive $n_k$.
\end{remark}

\begin{remark}
\label{eq:remark2}
	In step (1) of Algorithm \ref{algo:algo1}, that is, in the TMCMC step, as well as in any stage $k$ of step (2) of the algorithm, provided
	that $n_k$ is sufficiently large, desired credible regions of the respective posterior distributions of $\bx^*$ can be obtained. These quantify the stage-wise
	uncertainty in {\it a posteriori} learning about the optima, given $\|\bof'(\cdot)\|_d<\epsilon$ or $\|\bof'(\cdot)\|_d<\eta_k$. As $k\rightarrow\infty$, the uncertainty
	decreases, and the credibility regions shrink to the points representing the true optima.
	However, as mentioned in Remark \ref{remark:remark1}, in practice, particularly for large $d$, $n_k$ would be zero for most stages $k$, which would
	preclude computation of credible regions for most stages.
\end{remark}

\begin{remark}
\label{eq:remark3}
Note that if it is known beforehand that there is a single global minimum of $f(\cdot)$ on $\mathcal X$, 
then step (2) of Algorithm \ref{algo:algo1} is not required. It is then sufficient to
report $\bx^*_{i^*}$ as the (approximate) minimizer of $f(\cdot)$, where  $i^*=\arg\min\{f(\bx^*_i):i=1,\ldots,N\}$.
\end{remark}

\section{Bayesian characterization of the number of local minima of the objective function with recursive posteriors}
\label{sec:recursive}

Steps (2) (iii) and (2) (iv) can be combined to obtain a Bayesian characterization of the number of local minima of the objective function.
In this regard, for stage $j$, let us define $Y_j=\sum_{r=1}^MI_{B(\eta_j)}(\bx^*_{i_r})$, where $B(\eta_j)$ is given by (\ref{eq:B}). 
%where $B(\bof',\eta_k)=\left\{\bx:\|\bof'(\bx)\|_d<\eta_k\right\}$.
%$I(\cdot)$ is the indicator function defined by
%$I\left(\left\{\bSigma''(\bx^*_{i_r})>0\right\}\right)=1$ if $\bSigma''(\bx^*_{i_r})>0$, and zero otherwise.
Thus, $\{Y_j=m\}$ with probability $p_{mj}$, for $m=0,1,2,\ldots,M$.
Since $M\rightarrow\infty$, we allow $Y_j$ to take values on the entire set of non-negative integers.
That is, we set 
\begin{equation}
	P\left(Y_j=m\right)=p_{mj};~m=0,1,2,\ldots,
	\label{eq:mult1}
\end{equation}
the infinite-dimensional multinomial distribution, where $0\leq p_{mj}\leq 1$ for $m=0,1,2,\ldots$ and $j\geq 1$. Further, $\sum_{m=0}^{\infty}p_{mj}=1$ for all $j\geq 1$.
We assume that the true probabilities $p_{m0}\in [0,1]$; $m=0,1,2,\ldots$, such that $\sum_{m=0}^{\infty}p_{m0}=1$, are unknown. 
Indeed, if $f(\cdot)$ has finite number of local minima, then 
there must exist $\tilde m\geq 0$ such that $p_{\tilde m 0}=1$ and $p_{m0}=0$ for $m\neq\tilde m$. For infinite number of local minima, we must have
$p_{m0}=0$ for any finite integer $m\geq 0$.

We adopt the approach of \ctn{Roy17} based on Dirichlet process (\ctn{Ferguson73}) to 
obtain the posterior distribution of the infinite set of parameters $\{p_{mk};~m=0,1,2,\ldots\}$.
\ctn{Roy17} came up with the following strategy for characterizing possibly infinite number of limit points of infinite series, which we usefully
employ for our current purpose, albeit with a definition of $Y_j$ that is different from that of \ctn{Roy17}.  

Let $\mathcal Y=\left\{1,2,\ldots\right\}$ and let $\mathcal B\left(\mathcal Y\right)$ denote the Borel
$\sigma$-field on $\mathcal Y$ (assuming every singleton of $\mathcal Y$ is an open set). Let $\mathcal P$
denote the set of probability measures on $\mathcal Y$. Then, at the first stage,
\begin{equation}
	\pi\left(Y_1|P_1\right)\equiv P_1,
\label{eq:Y_DP}
\end{equation}
where $P_1\in\mathcal P$. We assume that
$P_1$ is the following Dirichlet process:
\begin{equation}
	\pi\left(P_1\right)\equiv DP\left(G\right),
\label{eq:DP}
\end{equation}
where, the probability measure $G$ is such that, for every $j\geq 1$,
\begin{equation}
G\left(Y_j=m\right)=\frac{1}{2^m}.
\label{eq:G}
\end{equation}
Then 
\begin{equation*}
	\pi\left(P_1|y_1\right)\equiv DP\left(G+\delta_{y_1}\right),
%\label{eq:posterior_DP1}
\end{equation*}
where, for any $x$, $\delta_{x}$ denotes point mass at $x$.

At the second stage, we set the prior for $P_2$ to be the posterior given $y_1$ corresponding to $DP\left(\left(1+\frac{1}{2^2}\right)G\right)$ prior for $P_1$.
That is, $\pi\left(P_2\right)\equiv DP\left((1+\frac{1}{2^2})G+\delta_{y_1}\right)$. Then, with respect to this prior for $P_2$, the posterior of $P_2$ is given by
\begin{equation*}
	\pi\left(P_2|y_2\right)\equiv DP\left(\left(1+\frac{1}{2^2}\right)G+\delta_{y_1}+\delta_{y_2}\right).
%\label{eq:posterior_DP2}
\end{equation*}

Continuing this recursive process we obtain that, at the $k$-th stage, 
the posterior of $P_k$ to be a Dirichlet process, given by
\begin{equation}
	\pi\left(P_k|y_k\right)\sim DP\left(\sum_{j=1}^k\frac{1}{j^2}G+\sum_{j=1}^k\delta_{y_j}\right).
\label{eq:posterior_DP}
\end{equation}
%where $\delta_{y_j}$ denotes point mass at $y_j$.
It follows from (\ref{eq:posterior_DP}) that
\begin{align}
E\left(p_{mk}|y_{k}\right)&=
\frac{\frac{1}{2^m}\sum_{j=1}^k\frac{1}{j^2}+\sum_{j=1}^k I\left(y_{j}=m\right)}
{\sum_{j=1}^k\frac{1}{j^2}+k};%\sum_{\ell=1}^M\sum_{j=1}^k\mathbb I\left(y_{j,n}=\ell\right)};
\label{eq:mean_DP}\\
Var\left(p_{mk}|y_{k}\right)&=
\frac{\left(\sum_{j=1}^k\frac{1}{j^2}+\sum_{j=1}^k I\left(y_{j}=m\right)\right)
\left((1-\frac{1}{2^m})\sum_{j=1}^k\frac{1}{j^2}+k-\sum_{j=1}^k I\left(y_{j}=m\right)\right)}
{\left(\sum_{j=1}^k\frac{1}{j^2}+k\right)^2\left(\sum_{j=1}^k\frac{1}{j^2}+k+1\right)}.
\label{eq:var_DP}
\end{align}

The theorem below characterizes the number of local minima of
the objective function $f(\cdot)$ in terms of the limit of the marginal posterior
probabilities of $p_{mk}$, denoted by $\pi_m\left(\cdot|y_k\right)$,  
as $k\rightarrow\infty$. 
\begin{theorem}
\label{theorem:post_minima1}
Assume the conditions of Theorem \ref{theorem:theorem1}, and in Algorithm \ref{algo:algo1}, assume that $M\rightarrow\infty$ and $N\rightarrow\infty$ such that
$M/N\rightarrow 0$.
Then $f(\cdot)$ has $\tilde m~(\geq 0)$ local minima if and only if 
\begin{equation}
	\pi_{\tilde m}\left(\mathcal N_1|y_{k}\right)\rightarrow 1,~\mbox{almost surely with respect to the posterior}~(\ref{eq:inv2}),
	\label{eq:con1}
\end{equation}
	as $k\rightarrow\infty$.
	In the above, $\mathcal N_1$ is any neighborhood of $1$ (one). 
%with $p_{m,0}$ satisfying
%$0\leq p_{m,0}\leq 1$ for $m=1,2,\ldots$ such that $\sum_{m=1}^{\infty}p_{m,0}=1$, with at most finite number
%of $m$ such that $p_{m,0}=0$.
%where $\delta_{\{0\}}$ denotes the point mass at zero.
%In the above, $\mathcal N_{p_{m,0}}$ is any neighborhood of $p_{m,0}$.
\end{theorem}
\begin{proof}
	First, let us assume that $f(\cdot)$ has $\tilde m~(\geq 0)$ local minima. Then, by Theorem \ref{theorem:algo1}, Algorithm \ref{algo:algo1} converges
	to the $\tilde m$ local minima almost surely, as $k\rightarrow\infty$, provided that $M\rightarrow\infty$ and $N\rightarrow\infty$ such that $M/N\rightarrow 0$.
	Hence, there exists $j_0\geq 1$, such that for $j\geq j_0$, $y_j=\tilde m_0$.
	Thus, almost surely, $I(y_j=\tilde m)=1$, for $j\geq j_0$.
Consequently, it easily follows from (\ref{eq:mean_DP}) and (\ref{eq:var_DP}) that almost surely, as $k\rightarrow\infty$,
\begin{align}
E\left(p_{\tilde mk}|y_{k}\right)&\rightarrow 1,~~\mbox{and}
\label{eq:mean_DP_convergence}\\
Var\left(p_{\tilde mk}|y_{k}\right)&= O\left(\frac{1}{k}\right)\rightarrow 0.
\label{eq:var_DP_convergence}
\end{align}
	Now let $\mathcal N_1$ denote any neighborhood of $1$, and let $\epsilon~(>0)$ be sufficiently small such that $\mathcal N_1\supseteq \{1-p_{\tilde mk}<\epsilon\}$.
	%Combining (\ref{eq:mean_DP_convergence}) and (\ref{eq:var_DP_convergence}) with Chebychev's inequality ensures that (\ref{eq:con1}) holds.
	Then using Markov's inequality we obtain 
	\begin{align}
		&\pi_{\tilde m}\left(\mathcal N_1|y_{k}\right)\geq \pi_{\tilde m}\left(1-p_{\tilde mk}<\epsilon|y_{k}\right)\notag\\
		&\qquad = 1-\pi_{\tilde m}\left(1-p_{\tilde mk}\geq\epsilon|y_{k}\right)\notag\\
		&\qquad\geq 1-\frac{E\left(1-p_{\tilde mk}|y_{k}\right)^2}{\epsilon^2}\notag\\
		&\qquad = 1-\frac{1-2E\left(p_{\tilde mk}|y_{k}\right)+E\left(p^2_{\tilde mk}|y_{k}\right)}{\epsilon^2}.
		\label{eq:markov1}
	\end{align}
	Now, as $k\rightarrow\infty$, $E\left(p_{\tilde mk}|y_{k}\right)\rightarrow 1$ by (\ref{eq:mean_DP_convergence}), and
	$E\left(p^2_{\tilde mk}|y_{k}\right)=Var\left(p_{\tilde mk}|y_{k}\right)+\left[E\left(p_{\tilde mk}|y_{k}\right)\right]^2\rightarrow 1$ by 
	(\ref{eq:mean_DP_convergence}) and (\ref{eq:var_DP_convergence}). Hence, the right hand side of (\ref{eq:markov1}) converges to $1$ almost surely, as
	$k\rightarrow\infty$. This proves (\ref{eq:con1}).

	Now assume that (\ref{eq:con1}) holds for any neighborhood $\mathcal N_1$ of $1$. Let us fix $\eta\in (0,1)$. 
	Then given any $\epsilon\in (0,1-\eta)$,
	\begin{equation}
	\pi_{\tilde m}\left(1-p_{\tilde mk}<\epsilon|y_{k}\right)\rightarrow 1, 
		\label{eq:markov2}
	\end{equation}
	almost surely as $k\rightarrow\infty$. The left hand side of (\ref{eq:markov2}) admits the following Markov's inequality:
         \begin{align}
           &\pi_{\tilde m}\left(1-p_{\tilde mk}<\epsilon|y_{k}\right)=\pi_{\tilde m}\left(p_{\tilde mk}>1-\epsilon|y_{k}\right) 
	      <\frac{E\left(p_{\tilde mk}|y_{k}\right)}{1-\epsilon}.
		 \label{eq:markov3}
	  %&\pi_{\tilde m}\left(1-p_{\tilde mk}<\epsilon|y_{k}\right)=\pi_{\tilde m}\left(p_{\tilde mk}>1-\epsilon|y_{k}\right) 
	 %	 <\frac{E\left(p^2_{\tilde mk}|y_{k}\right)}{(1-\epsilon)^2}.
         % \label{eq:markov4}
	\end{align}
	Due to (\ref{eq:markov2}), validity of (\ref{eq:markov3}) %and (\ref{eq:markov4}) 
	for all $\epsilon\in (0,1-\eta)$, and almost sure upper boundedness of $p_{\tilde mk}$ %and $p^2_{\tilde mk}$ 
	by $1$, it follows that
	\begin{equation}
		E\left(p_{\tilde mk}|y_{k}\right)\rightarrow 1,~\mbox{almost surely},~\mbox{as}~k\rightarrow\infty.
		\label{eq:meandp1}
	\end{equation}
	%$E\left(p_{\tilde mk}|y_{k}\right)\rightarrow 1$ and $E\left(p^2_{\tilde mk}|y_{k}\right)\rightarrow 1$, almost surely, as $k\rightarrow\infty$.
	%Hence, it also follows that $Var\left(p_{\tilde mk}|y_{k}\right)=E\left(p^2_{\tilde mk}|y_{k}\right)-\left[E\left(p_{\tilde mk}|y_{k}\right)\right]^2\rightarrow 0$, 
	%almost surely, as $k\rightarrow\infty$.
%	
	Now, if $f(\cdot)$ is assumed to have $m^*$ local minima where $m^*\neq\tilde m$, then due to (\ref{eq:mean_DP_convergence}) we must have
	\begin{equation}
		E\left(p_{m^*k}|y_{k}\right)\rightarrow 1,~\mbox{almost surely},~\mbox{as}~k\rightarrow\infty.
		\label{eq:dct0}
	\end{equation}
	%Consequently, $E\left(p_{\tilde mk}|y_{k}\right)\rightarrow 0$, amost surely, as $k\rightarrow\infty$,
	%since $\sum_{m=1}^{\infty}E\left(p_{mk}|y_{k}\right)=1$, almost surely, for all $k\geq 1$.
	Also note that since $0\leq p_{mk}\leq 1$ for all $m$ and $k$, the dominated convergence theorem ensures the following, almost surely:
	\begin{equation*}
	1=\underset{k\rightarrow\infty}{\lim}~\sum_{m=1}^{\infty}E\left(p_{mk}|y_{k}\right)
		=\sum_{m=1}^{\infty}\underset{k\rightarrow\infty}{\lim}~E\left(p_{mk}|y_{k}\right),
		%\label{eq:dct}
	\end{equation*}
	which implies, due to (\ref{eq:dct0}), that $E\left(p_{\tilde mk}|y_{k}\right)\rightarrow 0$, almost surely, as $k\rightarrow\infty$.
	But this would contradict (\ref{eq:meandp1}). %that here $E\left(p_{\tilde mk}|y_{k}\right)\rightarrow 1$, amost surely, as $k\rightarrow\infty$.
	Hence, $f(\cdot)$ must have $\tilde m$ local minima.
\end{proof}

\section{Experiments}
\label{sec:exps}
We consider application of Algorithm \ref{algo:algo1} to $5$ different optimization problems, ranging from simple to challenging, several of which are problems
of finding both maxima and minima, and one is concerned with saddle points and inconclusiveness in addition to maximum. 
Encouragingly, all our experiments bring forth the versatility of Algorithm \ref{algo:algo1} in capturing all the optima, saddlepoints, as well as inconclusiveness 
of the problems.

%It is important to mention the role of TMCMC in Algorithm \ref{algo:algo1}. First note that the entire $d$-dimensional random variable $\bx^*$ must be updated
%in a single block to meet the requirement $\bx^*\in B(\epsilon)$. Traditional MCMC algorithms are not known to be efficient in such problems as even for moderately
%large dimensions good proposal distributions are difficult to devise, and the acceptance rates can be poor, along with poor mixing properties. TMCMC, on the other hand,
%is designed to update all (or most of) the components of the high-dimensional random variable in a single block using appropriate deterministic transformations
%of some single (or low-dimensional) random variable. As such, this strategy drastically reduces effective dimensionality, which is responsible for maintaining good 
%acceptance rates in spite of high dimensions. Good mixing properties can also be ensured by judiciously choosing the relevant ``move-types", and judicious mixtures
%of additive and multiplicative transformations usually lead to desired mixing properties. For details on TMCMC and its properties, see \ctn{Dutta14}, \ctn{Dey16},
%\ctn{Dey17}, \ctn{Dey19}. 
As regards TMCMC, in our examples, we expectedly find that in the less challenging, low-dimensional problems, additive TMCMC is sufficient, while
in the more challenging cases, we consider appropriate mixtures of additive and multiplicative moves, followed by a further move of a specialized mixture
of additive and multiplicative transformations to improve mixing.

In most of our experiments, we discard the first $10^5$ TMCMC iterations as burn-in and store every $10$-th TMCMC realization in the next $5\times 10^5$ iterations,
to obtain $50000$ realizations before proceeding to step (2) of Algorithm \ref{algo:algo1}. 
However, in the $4$-th example, due to inadequate mixing, we discard the first $10^6$ iterations and store every $10$-th realization in 
the next $5\times 10^6$ iterations to obtain $5\times 10^5$ TMCMC realizations.
For the resample size in step (2), we set $M=1000$.

For the posterior of $\bg'(\cdot)=\bzero$ given by (\ref{eq:postpred5}), we set $a=b=0.1$, $\bbeta_0=\bzero$ and $\bSigma_0=\mathbb I_d$, for all the examples, where
$d$ is the dimension relevant to the problem.
In all the examples, we also set $\mathcal X=[-10,10]^d$ and the initial input size $n=10$. With $\bx_i=(x_{i1},\ldots,x_{id})$ for $i=1,\ldots,n$, we choose the inputs points 
as $x_{ik}=-10+2(i-1)$, for $i=1,\ldots,n$ and for $k=1,\ldots,d$. We then evaluate $f(\cdot)$ at each $\bx_i$; $i=1,\ldots,n$, to form $\bD_n$
with $n=10$. As we shall demonstrate, this strategy, in conjunction with the rest of the methodology, leads to adequate estimation of the optima in our examples. 

\subsection{Example 1}
\label{subsec:example1}
We begin with a simple example, where the goal is to obtain the maxima and minima of the function
\begin{equation}
	f(x)=2x^3-3x^2-12x+6.
	\label{eq:ex1}
\end{equation}
Here $f'(x)=6(x-2)(x+1)$ and $f''(x)=6(2x-1)$.
Hence, this function has a maximum at $x=-1$ and minimum at $x=2$.

We apply step (1) of Algorithm \ref{algo:algo1} with %$\mathcal X=[-10,10]$ and 
$\epsilon=1$, implementing additive TMCMC with 
equal move-type probabilities for forward and backward transformations.
Specifically, at iteration $t=1,2,\ldots$, letting $x^{(t-1)}$ denote the TMCMC realization at iteration $t-1$, 
we draw $\varepsilon\sim N(0,1)$ and consider the transformation $y=x^{(t-1)}+b|\varepsilon|$,
where $b$ takes the values $1$ and $-1$ with equal probabilities. We set $x^{(t)}=y$ with probability 
\begin{equation*}
	\alpha=\min\left\{1,\frac{\pi(y)\pi(g'(y)=0|\bD_n,y)}{\pi(x^{(t-1)})\pi(g'(x^{(t-1)})=0|\bD_n,x^{(t-1)})}\right\},
	%\label{eq:acc1}
\end{equation*}
and set $x^{(t)}=x^{(t-1)}$ with the remaining probability.
This is of course same as the ordinary random walk Metropolis algorithm since the dimension here is $d=1$; 
see \ctn{Dutta14} for ramifications and detailed discussions. 
\begin{figure}
	\centering
	\subfigure [TMCMC for maximum.]{ \label{fig:ex1_max}
	\includegraphics[width=7.5cm,height=6.5cm]{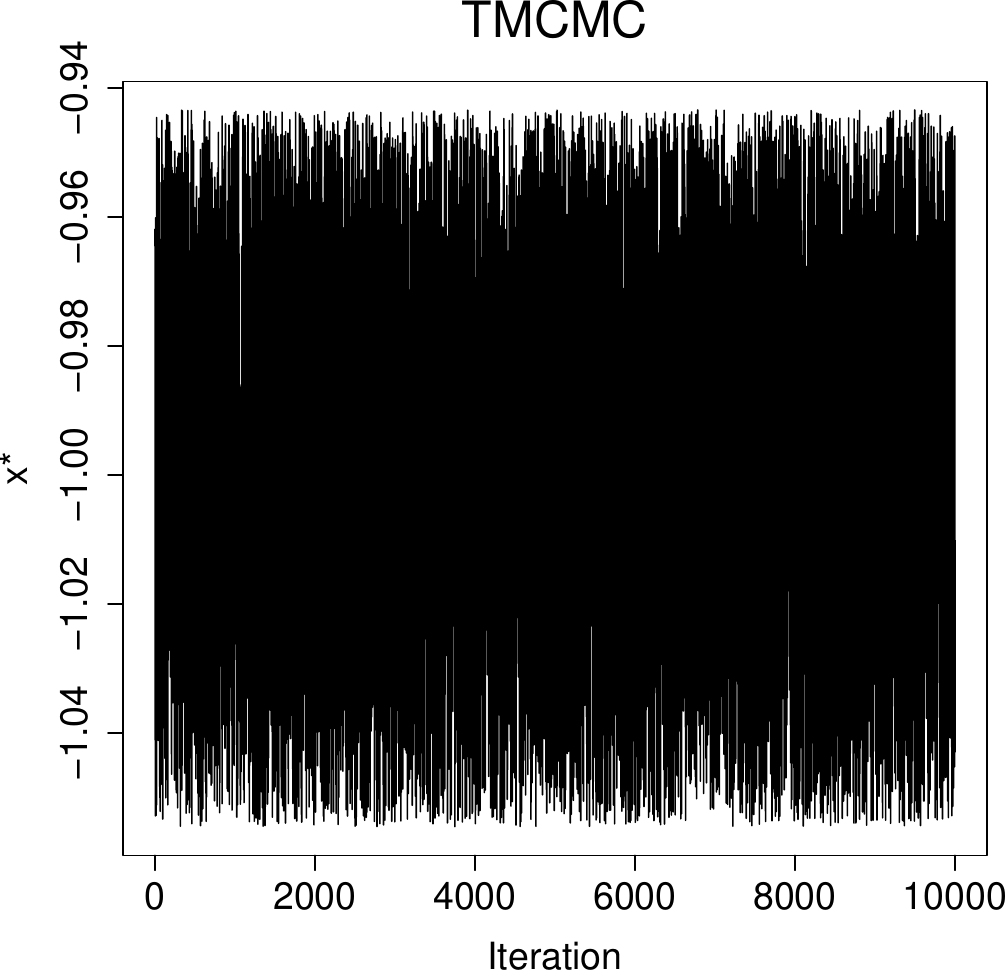}}
	\hspace{2mm}
	\subfigure [TMCMC for minimum.]{ \label{fig:ex1_min}
	\includegraphics[width=7.5cm,height=6.5cm]{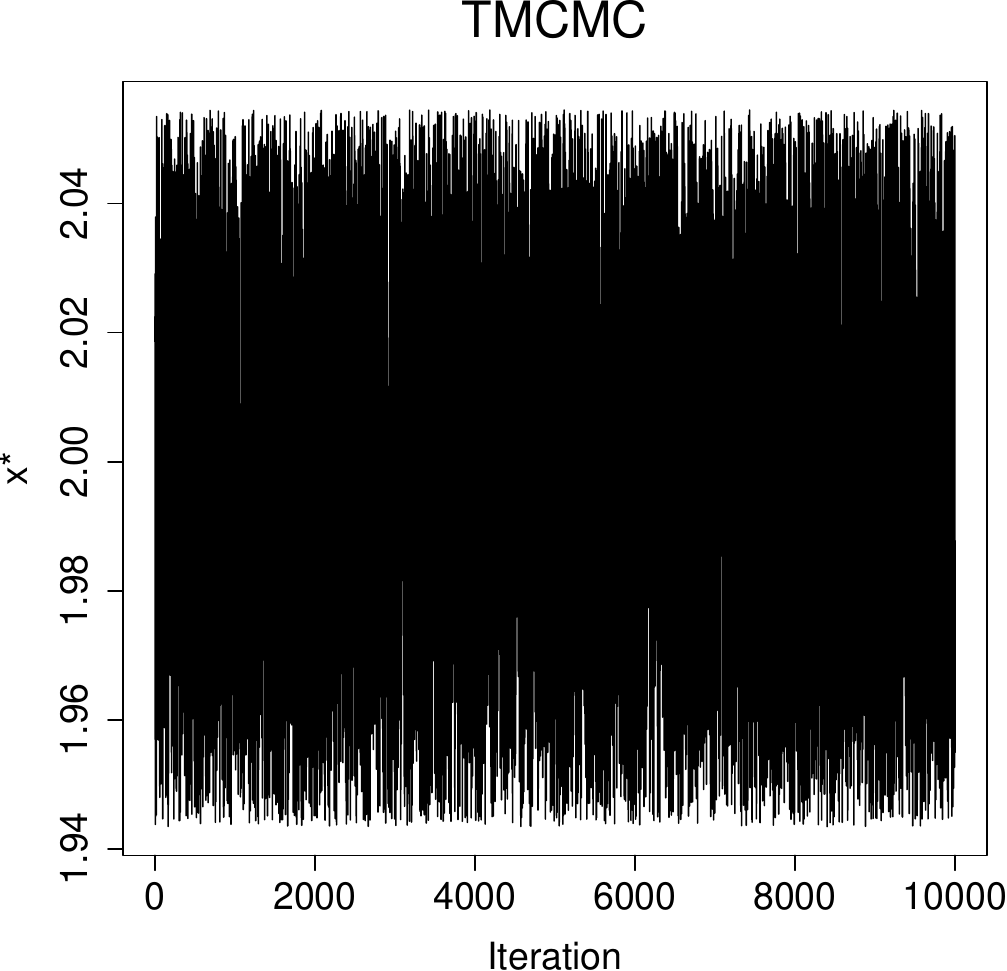}}
	\caption{TMCMC trace plots for Example 1.}
	\label{fig:ex1}
\end{figure}
Figure \ref{fig:ex1} shows the trace plots of $x^*$ for maximum and minimum, 
where to reduce the figure file size, we thinned the original TMCMC sample of size $N=50000$ to $10000$
by displaying every $5$-th realization. The trace plots indicate adequate mixing properties of TMCMC. 
We run step (2) of Algorithm \ref{algo:algo1} for $k=1,\ldots,S$ stages with $S=40$, setting $\eta_k=1/(10+k-1)^2$, 
computing the importance weights for the $N$ TMCMC realizations at each stage on $100$ $64$-bit cores in a VMWare parallel computing environment.
The cores have 2.80 GHz speed, and have access to 1 TB memory. All our codes are written in C, using the Message Passing Interface (MPI) protocol for parallel processing. 
As such, our entire exercise is completed in about $2$ minutes.
We obtain $\hat x_{\max}=-0.999995$ and
$\hat x_{\min}=2.000023$, as our estimates
of the maximum and the minimum, respectively, which are quite accurate. 

\subsection{Example 2}
\label{subsec:example2}
We now consider maximization and minimization of the function $f(x)=\sin(x)$ for $x\in[-10,10]$.
Here, the true maxima are $x_{\max}=\{\frac{\pi}{2}-2\pi=-4.712389,\frac{\pi}{2}=1.570796, \frac{\pi}{2}+2\pi=7.853982\}$, and the minima are
$x_{\min}=\{-7.853982,-1.570796,4.712389\}$.

We implement Algorithm \ref{algo:algo1} in the same way as in Example 1. Figure \ref{fig:ex2} displays the TMCMC trace plots, thinned to $10000$ realizations
to reduce figure file sizes. Clear tri-modality can be visualized from both the trace plots. After implementing step (2) of Algorithm \ref{algo:algo1} for $S=40$ stages
on our VMWare parallel computing architecture,
we obtain $\hat x_{\max}=\{-4.712581,1.570655,7.860879\}$ and $\hat x_{\min}=\{-7.854051,-1.570423,4.713222\}$, which turn out to be adequate approximations to the truths.
Again, the entire exercise takes about $2$ minutes.
\begin{figure}
	\centering
	\subfigure [TMCMC for maximum.]{ \label{fig:ex2_max}
	\includegraphics[width=7.5cm,height=6.5cm]{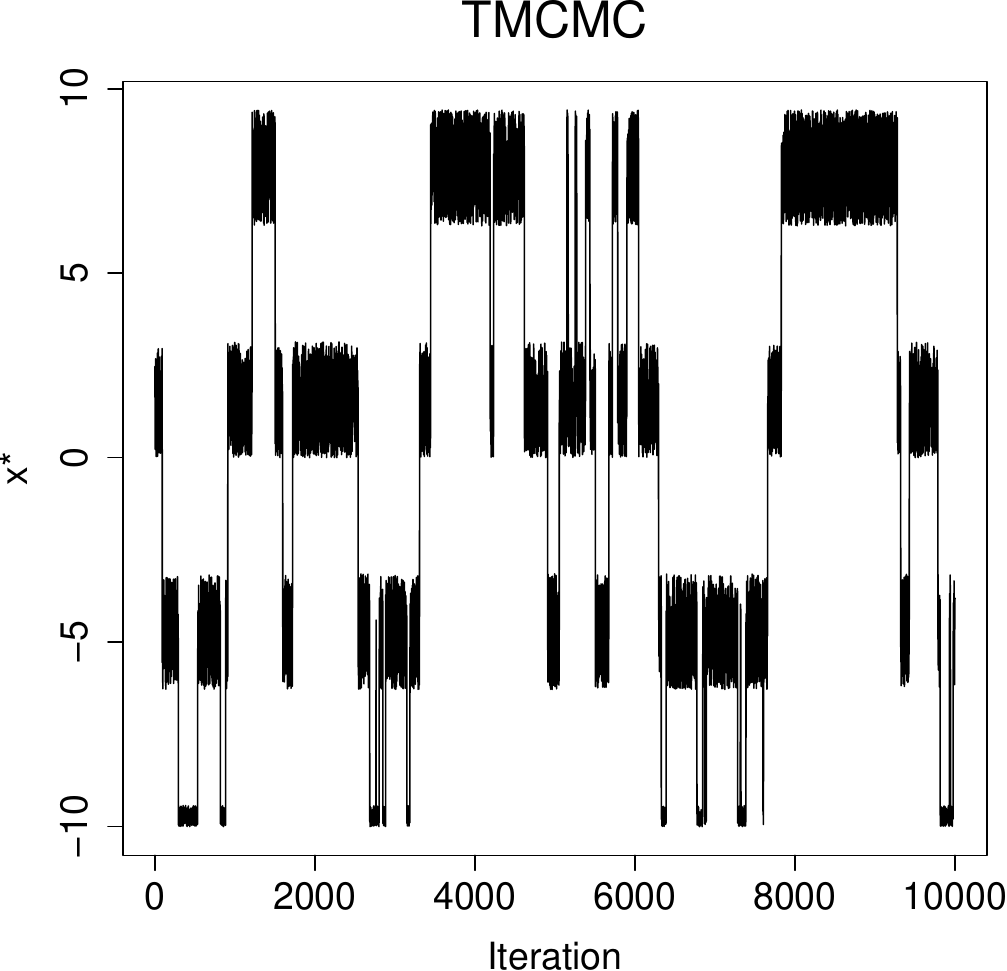}}
	\hspace{2mm}
	\subfigure [TMCMC for minimum.]{ \label{fig:ex2_min}
	\includegraphics[width=7.5cm,height=6.5cm]{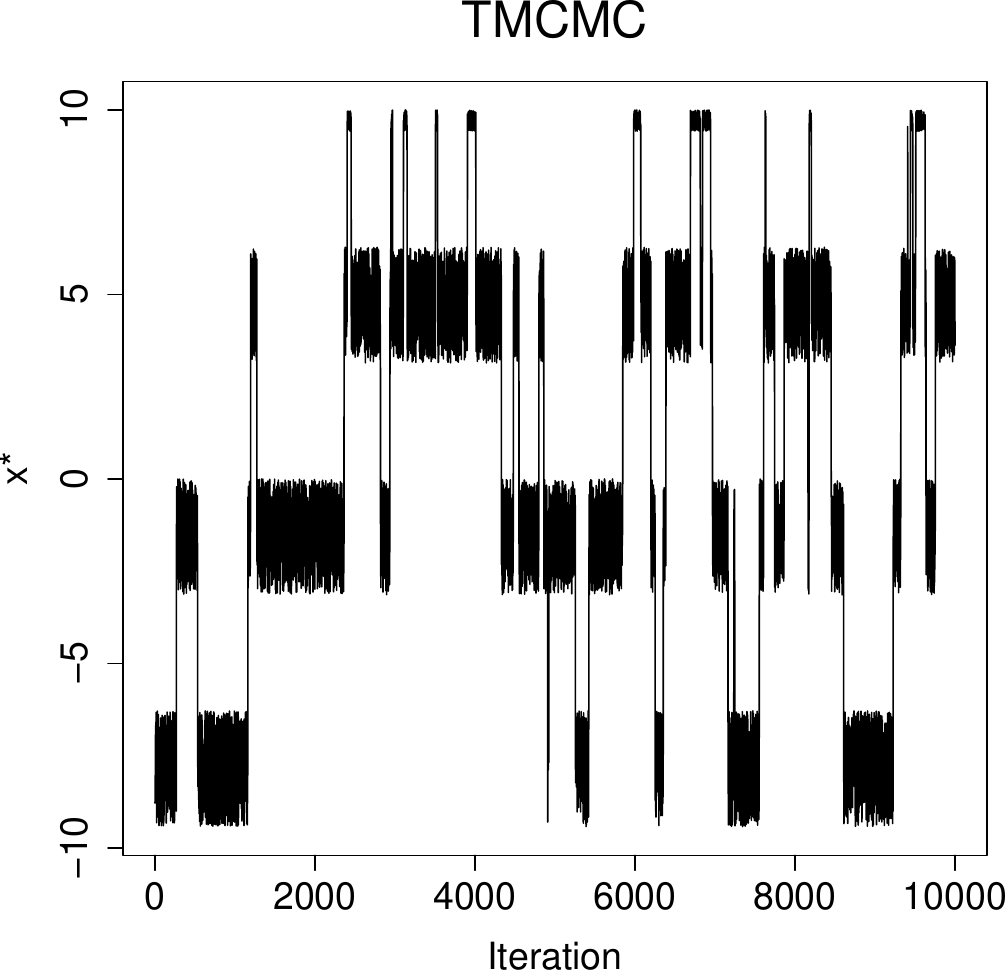}}
	\caption{TMCMC trace plots for Example 2.}
	\label{fig:ex2}
\end{figure}

\subsection{Example 3}
\label{subsec:example3}
Let us now consider a two-dimensional example, given by
$f(x_1,x_2)=x_1x_2(x_1+x_2)(1+x_2)$.

The first derivatives are given by 
$f'_1(x_1,x_2)=x_2(2x_1+x_2)(x_2+1)$ and $f'_2(x_1,x_2)=x_1(3x^2_2+2x_2(x_1+1)+x_1)$.

The second derivatives are $f''_{11}(x_1,x_2)=2x_2(x_2+1)$, $f''_{12}(x_1,x_2)=f''_{21}(x_1,x_2)=4x_1x_2+3x^2_2+2(x_1+x_2)$ and $f''_{22}(x_1,x_2)=2x_1(3x_2+x_1+1)$.

Consider the determinant $D(x_1,x_2)=f''_{11}(x_1,x_2)f''_{22}(x_1,x_2)-\left[f''_{12}(x_1,x_2)\right]^2$.
Now, if $(a,b)$ is any critical point of $f(\cdot)$ satisfying $f'_1(a,b)=0$ and $f'_2(a,b)=0$, then $(a,b)$ is a local maximum if $D(a,b)>0$ and 
$f''_{11}(a,b)<0$; $(a,b)$ is a local minimum if $D(a,b)>0$ and $f''_{11}(a,b)>0$; $(a,b)$ is a saddle point if $D(a,b)<0$. Furthermore, if $D(a,b)=0$,
then $(a,b)$ may be either maximum, minimum or even a saddle point, that is, the derivative test remains inconclusive in such cases.

In this example, it is easy to verify that there are four critical points $(0,0)$, $(0,-1)$, $(1,-1)$ and $\left(\frac{3}{8},-\frac{3}{4}\right)$.
The last point is a local maximum; $(0,-1)$ and $(1,-1)$ are saddle points, and the derivative test remains inconclusive about $(0,0)$.

We implement Algorithm \ref{algo:algo1} using the above conditions on $D(\bx^*)$ and $f''_{11}(\bx^*)>0$, along with the condition $\|\bof'(\bx^*)\|_2<\epsilon$,
with $\epsilon=1$, in the prior for $\bx^*=(x^*_1,x^*_2)$, for detection of maxima, minima, saddle points and inconclusiveness.  

%We apply step (1) of Algorithm \ref{algo:algo1} 
%We set $\mathcal X=[-10,10]^2$ and $\epsilon=1$, and 
We implement additive TMCMC with 
equal move-type probabilities for forward and backward transformations.
In these cases, at iteration $t=1,2,\ldots$, letting $\bx^{(t-1)}$ denote the TMCMC realization at iteration $t-1$, 
we draw $\varepsilon\sim N(0,1)$ and consider the transformation $\by=\bx^{(t-1)}+\bb|\varepsilon|$,
where, $\bb=(b_1,b_2)^T$ and each of $b_1$ and $b_2$ independently takes the values $1$ and $-1$ with equal probabilities. We set $\bx^{(t)}=\by$ with probability 
\begin{equation*}
	\alpha=\min\left\{1,\frac{\pi(\by)\pi(\bg'(\by)=\bzero|\bD_n,\by)}{\pi(\bx^{(t-1)})\pi(\bg'(\bx^{(t-1)})=\bzero|\bD_n,\bx^{(t-1)})}\right\},
	\label{eq:acc2}
\end{equation*}
and set $\bx^{(t)}=\bx^{(t-1)}$ with the remaining probability. Note that unlike the one-dimensional setup, this additive TMCMC is no longer equivalent to
random walk Metropolis (see \ctn{Dutta14} for details).

Implementation of Algorithm \ref{algo:algo1} for obtaining the maximum and the saddle points took about $7$ minutes to complete on our VMWare, implemented on $100$ cores. 

\subsubsection{Maximum}
\label{subsubsec:max}
Figure \ref{fig:ex3_max} displaying the TMCMC trace plots for maxima finding, indicates quite adequate mixing. 
Running step (2) of Algorithm \ref{algo:algo1} for $S=40$ steps yields
$\hat\bx_{max}=(0.376858, -0.752406)$, which is reasonably close to the true maximum.
\begin{figure}
	\centering
	\subfigure [TMCMC for maximum: first co-ordinate.]{ \label{fig:ex3_max1}
	\includegraphics[width=7.5cm,height=6.5cm]{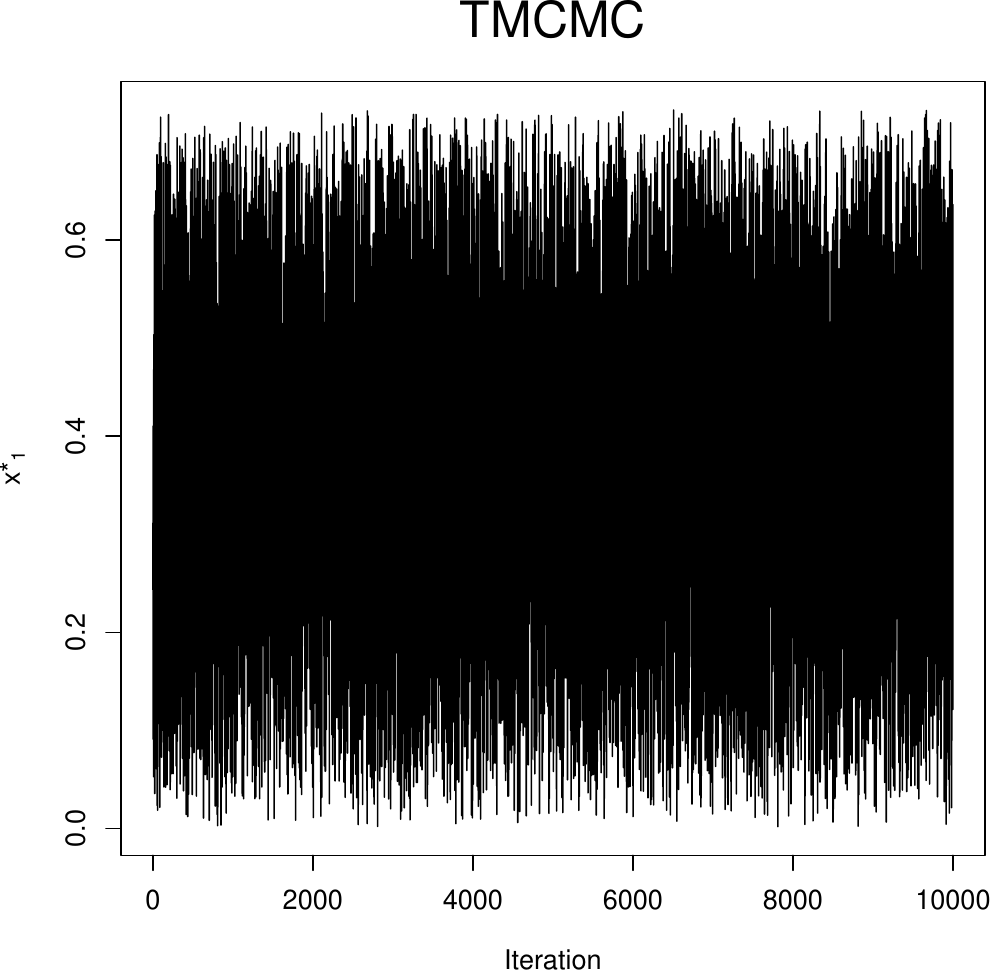}}
	\hspace{2mm}
	\subfigure [TMCMC for maximum: second co-ordinate.]{ \label{fig:ex3_max2}
	\includegraphics[width=7.5cm,height=6.5cm]{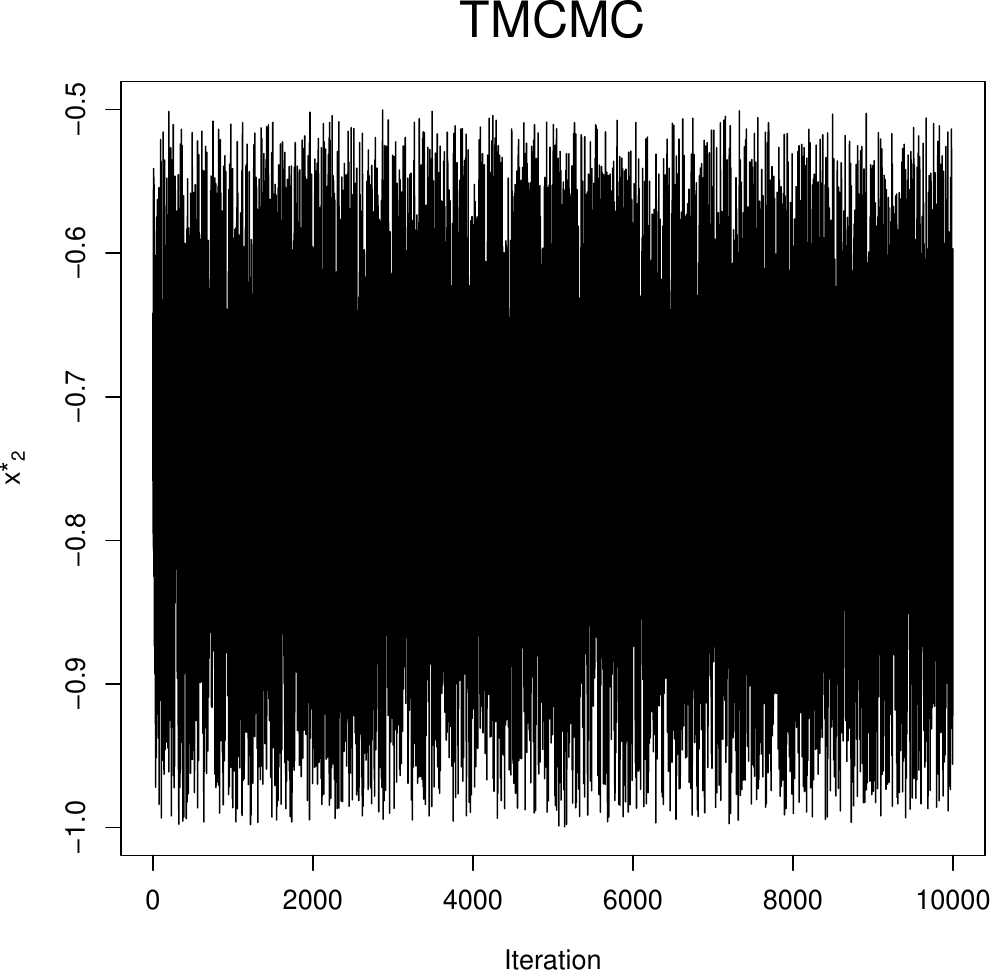}}
	\caption{TMCMC trace plots for Example 3 for finding maxima.}
	\label{fig:ex3_max}
\end{figure}

\subsubsection{Saddle points}
\label{subsubsec:saddle}
Our initial TMCMC investigations revealed two modal regions roughly around $(0.1,-1.1)$ and $(1.1,-1.1)$. 
For better exploration of the two modal regions, we implemented two separate TMCMC runs beginning at the above two points, and continued Algorithm \ref{algo:algo1}
to ultimately obtain two separate results after $S=40$ stages at step (2), associated with the two different starting points.
Figures \ref{fig:ex3_saddle1} and \ref{fig:ex3_saddle2} display the TMCMC trace plots associated with the two different starting points.  

Running step (2) of Algorithm \ref{algo:algo1} for $S=40$ steps yields
$\hat\bx^{(1)}_{saddle}=(-0.000309, -0.999580)$ and $\hat\bx^{(2)}_{saddle}=(0.999433, -0.999915)$ as estimates of two saddle points, 
which are both reasonably close to the true saddle points.
\begin{figure}
	\centering
	\subfigure [TMCMC for first saddle point: first co-ordinate.]{ \label{fig:ex3_sad1}
	\includegraphics[width=7.5cm,height=6.5cm]{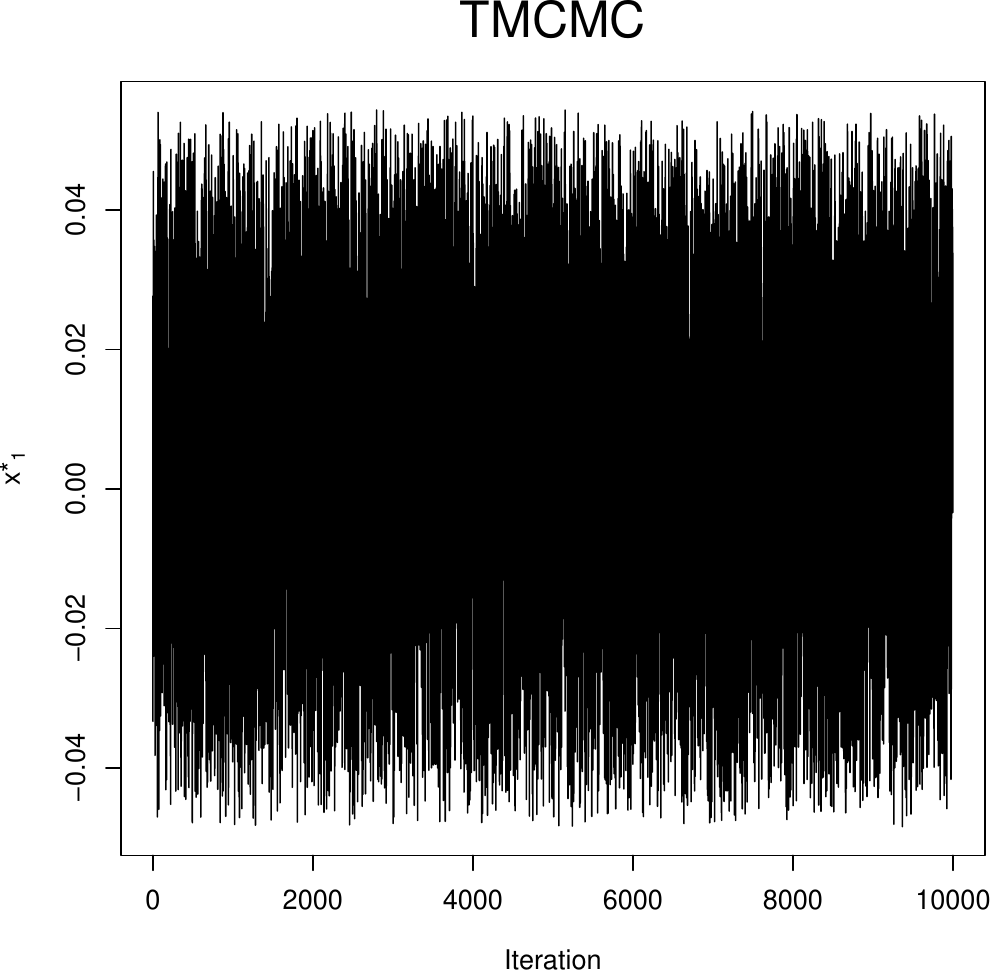}}
	\hspace{2mm}
	\subfigure [TMCMC for first saddle point: second co-ordinate.]{ \label{fig:ex3_sad2}
	\includegraphics[width=7.5cm,height=6.5cm]{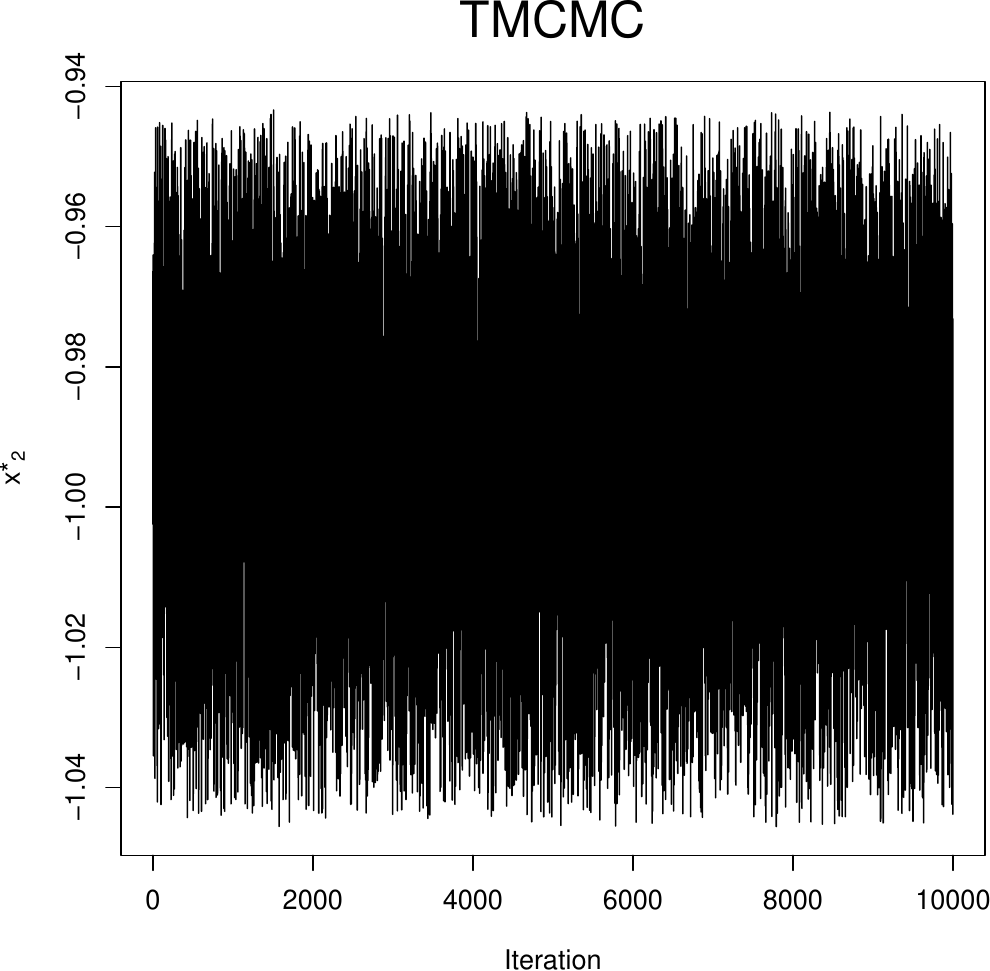}}
	\caption{TMCMC trace plots for Example 3 for finding the first saddle point.}
	\label{fig:ex3_saddle1}
\end{figure}

\begin{figure}
	\centering
	\subfigure [TMCMC for second saddle point: first co-ordinate.]{ \label{fig:ex3_sad3}
	\includegraphics[width=7.5cm,height=6.5cm]{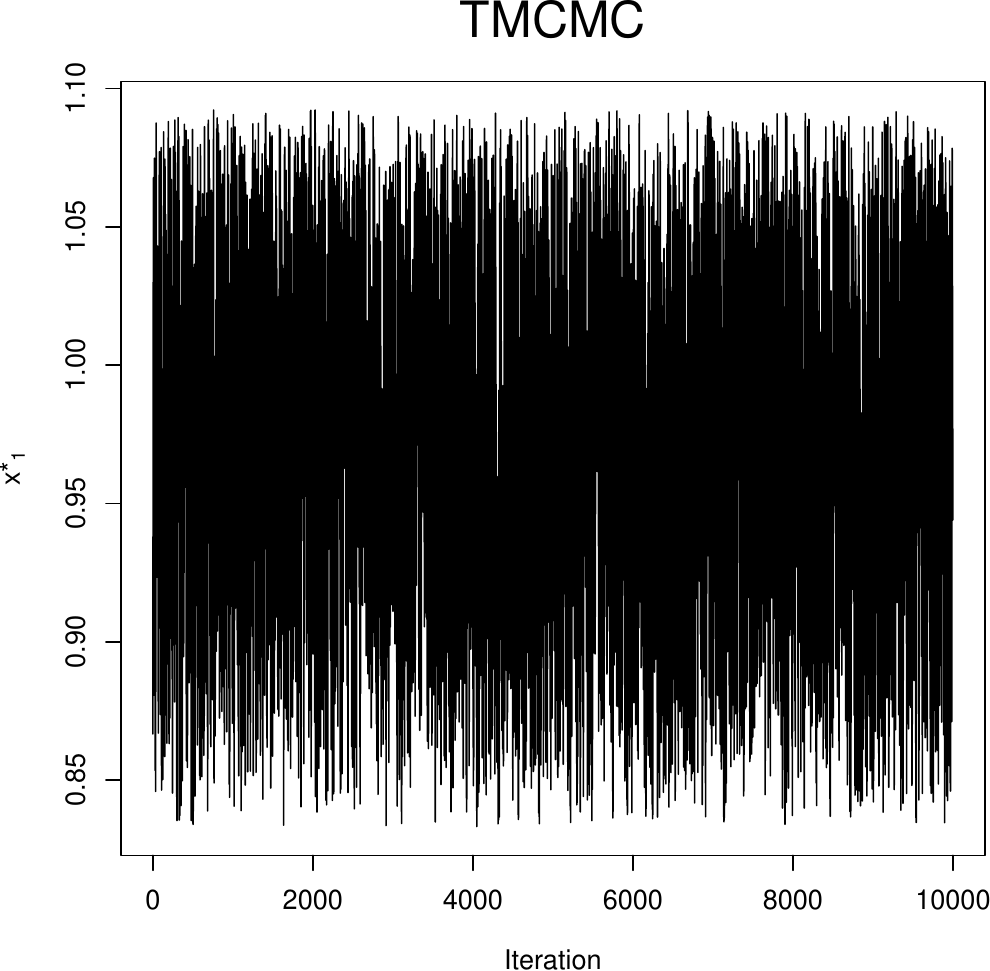}}
	\hspace{2mm}
	\subfigure [TMCMC for second saddle point: second co-ordinate.]{ \label{fig:ex3_sad4}
	\includegraphics[width=7.5cm,height=6.5cm]{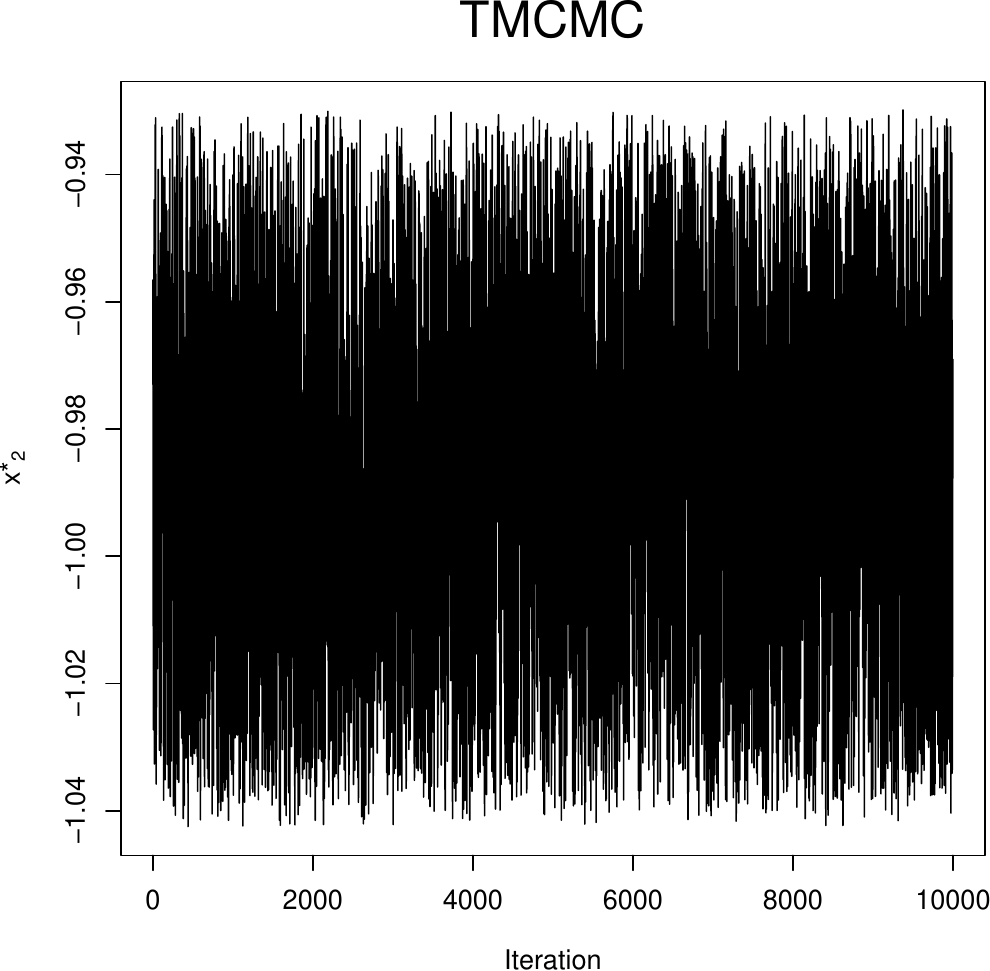}}
	\caption{TMCMC trace plots for Example 3 for finding the second saddle point.}
	\label{fig:ex3_saddle2}
\end{figure}

\subsubsection{Inconclusiveness}
\label{subsubsec:incon}
Investigation of situations where $D(a,b)=0$ for any critical point $(a,b)$ yielded the TMCMC trace plots displayed as Figure \ref{fig:ex3_incon}.
On completion of step (2) of Algorithm \ref{algo:algo1}, we obtain $\hat\bx_{incon}=(-0.000087, -0.000265)$ as the estimate of the stationary point
regarding which conclusion can not be drawn using the second derivatives. Observe that $\hat\bx_{incon}$ quite adequately estimates the actual point $(0,0)$
where conclusion fails.
\begin{figure}
	\centering
	\subfigure [TMCMC for inconclusiveness: first co-ordinate.]{ \label{fig:ex3_incon1}
	\includegraphics[width=7.5cm,height=6.5cm]{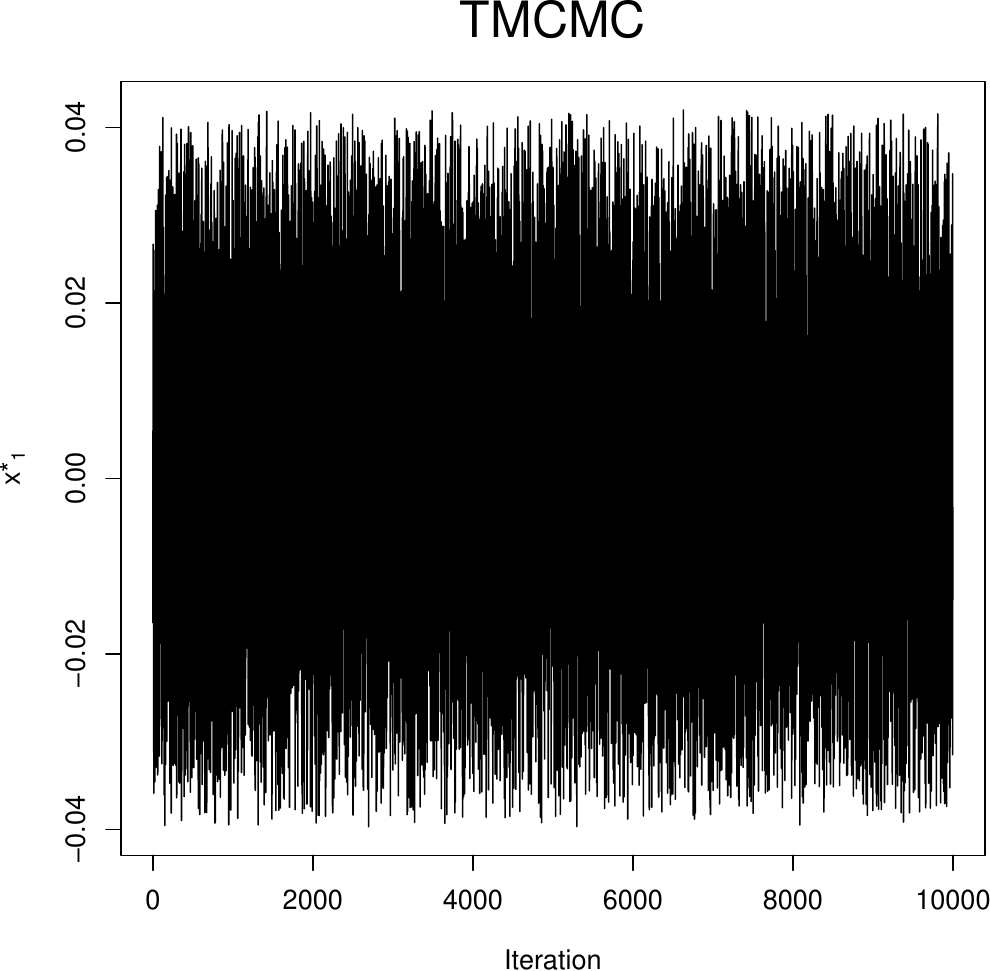}}
	\hspace{2mm}
	\subfigure [TMCMC for inconclusiveness: second co-ordinate.]{ \label{fig:ex3_incon2}
	\includegraphics[width=7.5cm,height=6.5cm]{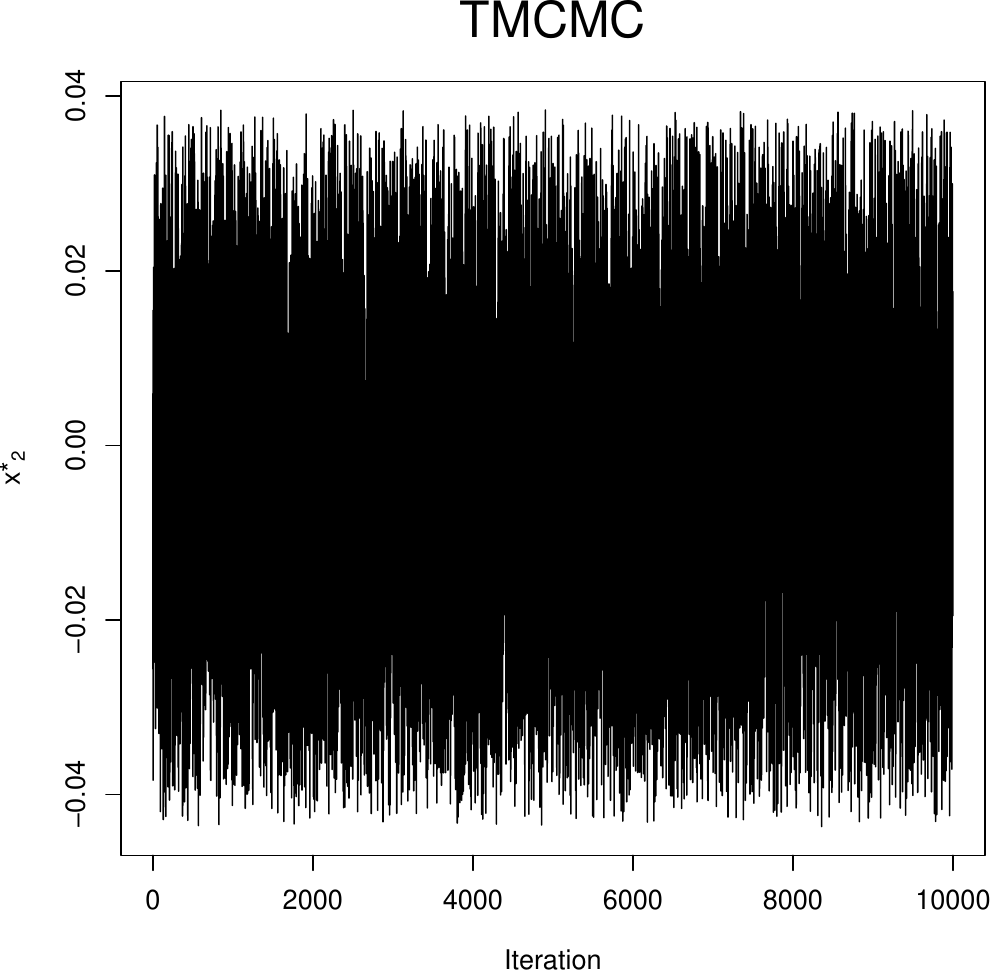}}
	\caption{TMCMC trace plots for Example 3 for investigating inconclusiveness.}
	\label{fig:ex3_incon}
\end{figure}

\subsubsection{Minimum}
\label{subsubsec:min}
Our attempt to implement TMCMC with the restrictions $D(\bx^*)>0$ and $f''_{11}(\bx^*)>0$ with $\|\bof'(\bx^*)\|_2<\epsilon$ did not yield any acceptance,
even for arbitrary initial values. In other words, we could not obtain any solution that satisfies all the above restrictions, and hence conclude that
there is no critical point on $\mathcal X$ that satisfy the above restrictions.

\subsection{Example 4}
\label{subsec:example4}
Table 14.2 of \ctn{Lange10} reports quarterly data on AIDS deaths in Australia during $1983-1986$, which is considered for illustration of fitting Poisson regression model.
Specifically, for $i=1,\ldots,14$, \ctn{Lange10} considers the model $Y_i\sim Poisson(\lambda_i)$, with $\lambda_i=\exp\left(\beta_0+i\beta_1\right)$.
\ctn{Lange10} computed the maximum likelihood estimate (MLE) of $\bbeta=(\beta_0,\beta_1)$ using Fisher's scoring method, which is equivalent to Newton's method in this case.
The final estimate obtained obtained by \ctn{Lange10} is $\hat\bbeta_{MLE}=(0.3396,0.2565)$. 

In our notation, the function to maximize is
$f(x_1,x_2)=-\sum_{i=1}^{14}\exp(x_1+ix_2)+\sum_{i=1}^{14}y_i(x_1+ix_2)$, with respect to $\bx=(x_1,x_2)$.
Note that this is a concave maximization problem and hence the second derivative is irrelevant.
We thus consider the only constraint $\|\bof'(\bx^*)\|_2<\epsilon$ for our implementation, with $\epsilon=1$. However, additive TMCMC did not exhibit adequate
mixing properties in this example, and hence we consider a mixture of additive and multiplicative TMCMC that is expected to improve mixing by using
a mixture of localised moves of additive TMCMC and non-local (``random dive") moves of multiplicative TMCMC (see \ctn{Dutta12}, \ctn{Dey16} for details). 
We strengthen the mixture TMCMC strategy with a further step of a mixture of specialized 
additive and multiplicative moves, which has parallels with \ctn{Liu00}. The detailed TMCMC algorithm, for general dimension $d$, is provided below as Algorithm \ref{algo:tmcmc}.

\begin{algorithm}
\caption{Mixture TMCMC}
\label{algo:tmcmc}
\begin{itemize}
	\item[(1)] Fix $p,q\in (0,1)$. Set an initial value $\bx^{(0)}$.
\item[(2)] For $t=1,\ldots,N$, do the following:
	\begin{enumerate}
		\item Generate $U\sim U(0,1)$. 
			\begin{enumerate}
				\item If $U<p$, then do the following:
	\begin{enumerate}
		\item[(i)] Generate $\varepsilon\sim N(0,1)$, $b_j\stackrel{iid}{\sim}U(\{-1,1\})$ for $j=1,\ldots,d$, and set
			$y_j=x^{(t-1)}_j+b_ja^{(1)}_j|\varepsilon|$, for $j=1,\ldots,d$. Here $a^{(1)}_j$ are positive scaling constants. 
		\item[(ii)] Evaluate 
                        \begin{equation*}
	                 \alpha_1=\min\left\{1,\frac{\pi(\by)\pi(\bg'(\by)=\bzero|\bD_n,\by)}{\pi(\bx^{(t-1)})\pi(\bg'(\bx^{(t-1)})=\bzero|\bD_n,\bx^{(t-1)})}\right\}.
	                  %\label{eq:acc1}
                        \end{equation*}
		\item[(iii)] Set $\tilde\bx^{(t)}=\by$ with probability $\alpha_1$, else set $\tilde\bx^{(t)}=\bx^{(t-1)}$. 

	\end{enumerate}
\item If $U\geq p$, then
	\begin{enumerate}
		\item[(i)] Generate $\varepsilon\sim U(-1,1)$, $b_j\stackrel{iid}{\sim}U(\{-1,0,1\})$ for $j=1,\ldots,d$, and set
			$y_j=x^{(t-1)}_j\varepsilon$ if $b_j=1$, $y_j=x^{(t-1)}_j/\varepsilon$ if $b_j=-1$ and $y_j=x^{(t-1)}_j$ if $b_j=0$, for $j=1,\ldots,d$. 
			Calculate $|J|=|\varepsilon|^{\sum_{j=1}^db_j}$.
		\item[(ii)] Evaluate 
                        \begin{equation*}
	                 \alpha_2=\min\left\{1,\frac{\pi(\by)\pi(\bg'(\by)=\bzero|\bD_n,\by)}{\pi(\bx^{(t-1)})\pi(\bg'(\bx^{(t-1)})=\bzero|\bD_n,\bx^{(t-1)})}\times|J|\right\}.
	                  %\label{eq:acc1}
                        \end{equation*}
		\item[(iii)] Set $\tilde\bx^{(t)}=\by$ with probability $\alpha_2$, else set $\tilde\bx^{(t)}=\bx^{(t-1)}$. 

	\end{enumerate}
			\end{enumerate}
		
\item Generate $U\sim U(0,1)$. 
			\begin{enumerate}
				\item If $U<q$, then do the following
	\begin{enumerate}
		\item[(i)] Generate $\tilde U\sim U(0,1)$ and $\varepsilon\sim N(0,1)$. If $\tilde U<1/2$, set
			$y_j=\tilde x^{(t)}_j+a^{(2)}_j|\varepsilon|$, for $j=1,\ldots,d$; else, set $y_j=\tilde x^{(t)}_j-a^{(2)}_j|\varepsilon|$, for $j=1,\ldots,d$. 
			Here $a^{(2)}_j$ are positive scaling constants. 
		\item[(ii)] Evaluate 
                        \begin{equation*}
	                \alpha_3=\min\left\{1,\frac{\pi(\by)\pi(\bg'(\by)=\bzero|\bD_n,\by)}{\pi(\tilde\bx^{(t)})\pi(\bg'(\tilde\bx^{(t)})=\bzero|\bD_n,\tilde\bx^{(t)})}\right\}.
	                  %\label{eq:acc1}
                        \end{equation*}
		\item[(iii)] Set $\bx^{(t)}=\by$ with probability $\alpha_3$, else set $\bx^{(t)}=\tilde\bx^{(t)}$. 

	\end{enumerate}
\item If $U\geq q$, then
	\begin{enumerate}
		\item[(i)] Generate $\varepsilon\sim U(-1,1)$ and $\tilde U\sim U(0,1)$. If $\tilde U<1/2$, set
			$y_j=\tilde x^{(t)}_j\varepsilon$ for $j=1,\ldots,d$ and $|J|=|\varepsilon|^d$, else set $y_j=\tilde x^{(t)}_j/\varepsilon$ for $j=1,\ldots,d$ and 
			$|J|=|\varepsilon|^{-d}$.
		\item[(ii)] Evaluate 
               \begin{equation*}
	       \alpha_4=\min\left\{1,\frac{\pi(\by)\pi(\bg'(\by)=\bzero|\bD_n,\by)}{\pi(\tilde\bx^{(t)})\pi(\bg'(\tilde\bx^{(t)})=\bzero|\bD_n,\tilde\bx^{(t)})}\times|J|\right\}.
	       %\label{eq:acc1}
               \end{equation*}
	        \item[(iii)] Set $\bx^{(t)}=\by$ with probability $\alpha_4$, else set $\bx^{(t)}=\tilde\bx^{(t)}$. 

	\end{enumerate}
			\end{enumerate}
	\end{enumerate}
\item[(3)] Store the realizations $\left\{\bx^{(t)};t=0,1,\ldots,N\right\}$ for Bayesian inference. 
\end{itemize}
\end{algorithm}

In our case, $d=2$, and we choose $a^{(1)}_j=a^{(2)}_j=0.05$, for $j=1,2$; we also set $p=q=1/2$. 
However, in spite of such a sophisticated TMCMC algorithm, we failed to achieve excellent
mixing in this example, even with long TMCMC runs, discarding the first $10^6$ iterations and storing every $10$-th realization in the next $5\times 10^6$ iterations.
The trace plots of all stored $5\times 10^5$ realizations shown in Figure \ref{fig:ex4} indeed demonstrate that the TMCMC chain does not have excellent mixing properties. 
\begin{figure}
	\centering
	\subfigure [TMCMC for MLE: first co-ordinate.]{ \label{fig:ex4_1}
	\includegraphics[width=7.5cm,height=6.5cm]{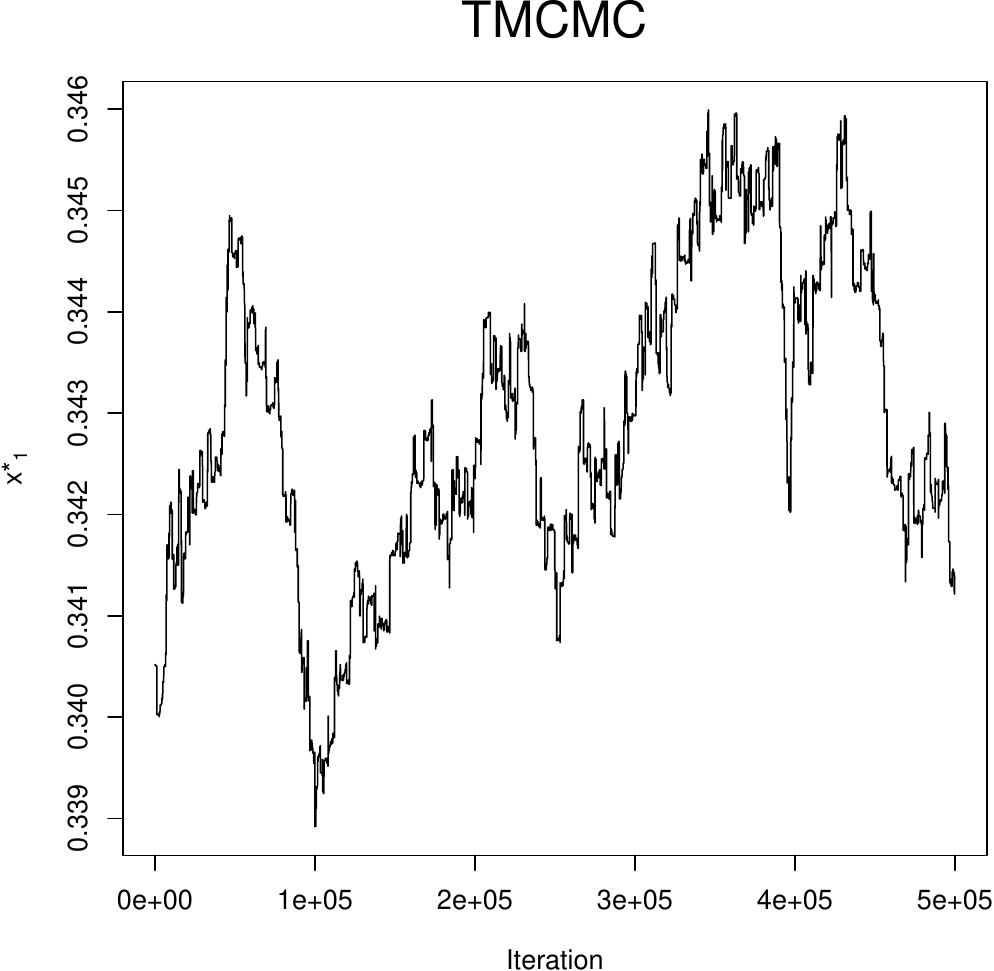}}
	\hspace{2mm}
	\subfigure [TMCMC for MLE: second co-ordinate.]{ \label{fig:ex4_2}
	\includegraphics[width=7.5cm,height=6.5cm]{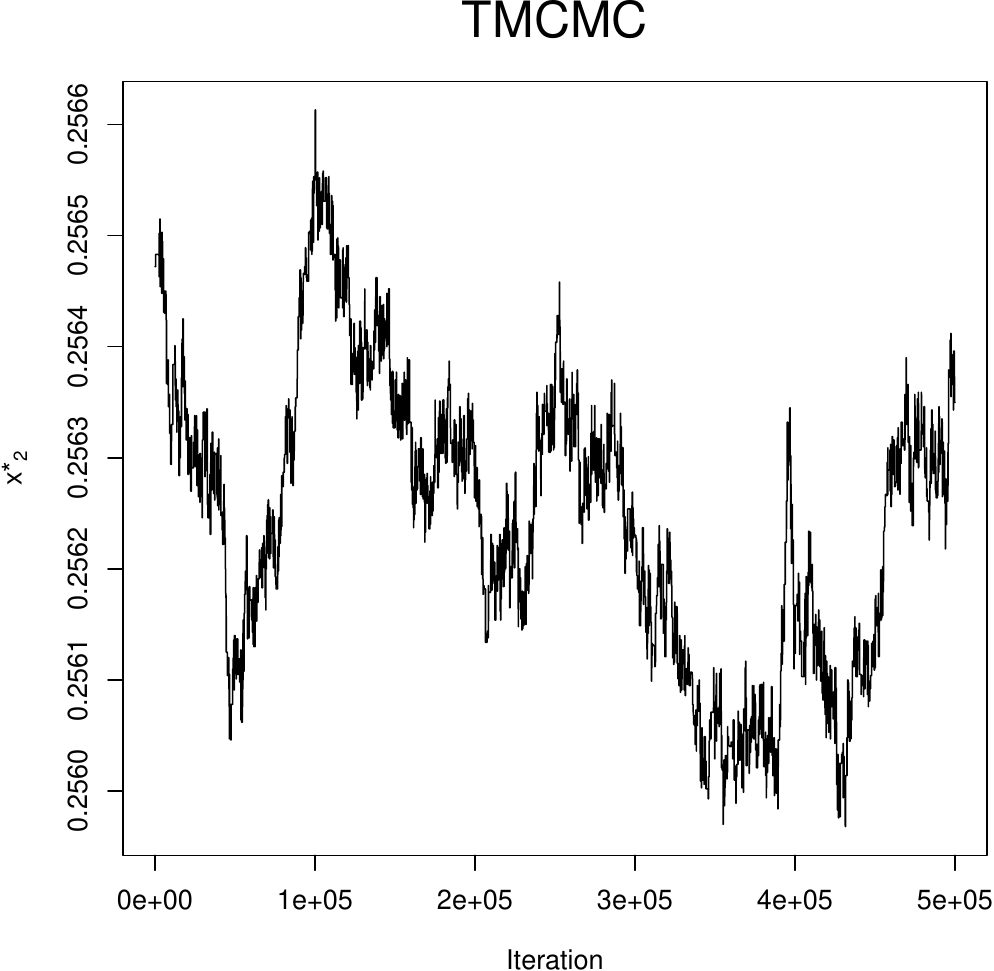}}
	\caption{TMCMC trace plots for Example 4 for finding MLE.}
	\label{fig:ex4}
\end{figure}
In fact, the trace plots correspond to a reasonable initial value, chosen as the optimum obtained by the Fisher scoring method, which we implemented in R
independently of \ctn{Lange10}. 
%Using later Fisher scoring iterates as initial values led to increasingly improved performance, but 
However, our goal here is to demonstrate that even when TMCMC mixing is less than ideal, the estimates obtained by our method can still significantly outperform existing techniques.

Since this is a concave maximization problem, step (2) of Algorithm~\ref{algo:algo1} is unnecessary. Following Remark~\ref{eq:remark3}, we set $i^* = \arg\min\{f(\mathbf{x}_i^*) : i = 1,\ldots,N\}$ and report $\mathbf{x}_{i^*}^*$ as the approximate maximizer of $f(\cdot)$. With the $N = 5\times10^5$ TMCMC realizations shown in Figure~\ref{fig:ex4}, we obtain the estimate $\hat{\mathbf{x}}_{\text{MLE}} = (0.339627,\,0.256524)$, for which $\|\mathbf{f}'(\hat{\mathbf{x}}_{\text{MLE}})\|_2 = 0.005108$. In contrast, the MLE obtained by Fisher's scoring (reported to four digits by \ctn{Lange10}) is $\hat{\boldsymbol{\beta}}_{\text{MLE}} = (0.339634,\,0.256524)$, with a gradient norm of $0.011791$ — substantially farther from zero. %(We implemented Fisher's scoring in R.)

We also compared our approach with other popular optimization methods implemented in R: the BFGS (Broyden–Fletcher–Goldfarb–Shanno) method, simulated annealing followed by BFGS refinement, and a multi‑start quasi‑Newton method with 10 random starts uniformly sampled from $[-2,2]^2$. The BFGS and simulated annealing results were obtained using R's \texttt{optim} function with default settings. For the multi‑start quasi‑Newton method, we report the mean of the estimates from the 10 runs; we deliberately did not select the best solution because the other competing methods rely on single starts only. Thus, the multi‑start results are comparable to an average, not to a minimum‑norm solution over multiple starts.

Up to the first six decimal places, Fisher's scoring, BFGS, and simulated annealing produced identical optima. Table~\ref{tab:ex4_compare} shows that in every case our Bayesian procedure achieves a substantially smaller gradient norm. All methods completed in a few seconds, so we do not report run times; gradient norms were computed analytically.

Thus, our Bayesian approach proves more reliable than existing methods. This may be a general phenomenon: by starting from good initial values (for instance, optima obtained by popular methods), TMCMC can explore regions of $\mathcal{X}$ that almost surely contain points with smaller gradient norms than those found by the initial method. It is also encouraging that for this example, TMCMC took well under a minute to generate $6\times10^6$ iterations on a single processor of our VMWare system.

%\subsubsection{Comparison with standard optimization methods}
%To further demonstrate the effectiveness of our approach, we compare our results with those obtained by several widely used optimization algorithms. Table \ref{tab:ex4_compare} reports the estimates, gradient norms, and computational times for the AIDS deaths data example.
%
\begin{table}[ht]
\centering
\caption{Comparison of optimization methods for Example 4 (Poisson regression MLE). The true gradient norm at the optimum is zero.}
\label{tab:ex4_compare}
\begin{tabular}{lcccc}
\hline
	Method & $\hat\beta_0$ & $\hat\beta_1$ & $\|\nabla f\|_2$ & $f(\hat\beta_0,\hat\beta_1)$\\ \hline
	Fisher's scoring & 0.339634 & 0.256524 & 0.011791 & 472.062548\\
	BFGS  & 0.339634 & 0.256524 & 0.011791 & 472.062548\\
	Simulated annealing  & 0.339634 & 0.256524 & 0.011791 & 472.062548\\
	Multi-start quasi-Newton (10 starts) & 0.339850 & 0.256501 & 0.123481 & 472.062547\\
	Our method (TMCMC + posterior) & 0.339627 & 0.256524 & 0.005108 & 472.062548\\ \hline
\end{tabular}

%\caption{%Standard deviations for multi-start quasi-Newton are shown in parentheses. 
{Our method achieves the smallest gradient norm, indicating closer proximity to the true MLE.}
\end{table}

\subsection{Example 5}
\label{subsec:example5}
\ctn{Hunter00} refer to a nonlinear optimization problem of the form $Y_i\sim N\left(\mu_i,\sigma^2\right)$ for $i=1,\ldots,m$, where
$$\mu_i=\sum_{j=1}^d\left[\exp\left(-z_{ij}\theta^2_j\right)+z_{ij}\theta_{d-j+1}\right];$$
$z_{ij}$ being the $i$-th observation of the $j$-th covariate, for $i=1,\ldots,m$ and $j=1,\ldots,d$.
The goal is to compute the MLE of $\btheta=(\theta_1,\ldots,\theta_d)$, assuming that $\sigma$ is known.
In our notation, the objective is to minimize
$$f(\bx)=\sum_{i=1}^m\left(y_i-\sum_{j=1}^d\left[\exp\left(-z_{ij}x^2_j\right)+z_{ij}x_{d-j+1}\right]\right)^2,$$
with respect to $\bx$.

We consider $5$ simulation experiments in this regard for $d=2,5,10,50,100$. In each case we generate $\theta_{0j}\sim U(-1,1)$ independently for $j=1,\ldots,d$, 
$z_{ij}\sim N(0,1)$ independently, for $i=1,\ldots,m$ and $j=1,\ldots,d$, and set, for $i=1,\ldots,m$, 
$\mu_{0i}=\sum_{j=1}^d\left[\exp\left(-z_{ij}\theta^2_{0j}\right)+z_{ij}\theta_{0,d-j+1}\right]$. We finally generate the response data by simulating
$Y_i\sim N\left(\mu_{0i},\sigma^2_0\right)$ independently for $i=1,\ldots,m$, where we set $\sigma^2_0=0.1$. For $d=2,5,10,50,100$, we generate datasets of
sizes $m=10,10,20,75,200$.

Note that this is not a convex minimization problem and the matrix of second derivatives $\bSigma''(\cdot)$ plays an important role, along with $\|\bof'(\cdot)\|_d$.
Thus it is important to check positive definiteness of $\bSigma''(\cdot)$ for any dimension $d$.
We use the LAPACK library function $``dpotrf"$ for Cholesky decomposition of $\bSigma''(\cdot)$, which contains a parameter $``info"$. Given any $\bx$, $info=0$ indicates 
positive definiteness of $\bSigma''(\bx)$, while other values of $info$ rules out positive definiteness.  
Note that since this is not a convex minimization problem, step (2) of Algorithm \ref{algo:algo1} is necessary, unlike in Example 4.

We implement the mixture TMCMC algorithm \ref{algo:tmcmc} for all values of $d$, with $p=q=1/2$ and set $a^{(1)}_j=a^{(2)}_j=0.05$ for $j=1,\ldots,d$.

\subsubsection{Case 1: $d=2$}
\label{subsubsec:case1}
In the TMCMC step, we set $\epsilon=1$ in the restriction $\|\bof'(\cdot)\|_2<\epsilon$. The trace plots, shown in Figure \ref{fig:ex5_nlm2}, exhibit adequate mixing.
Running step (2) of Algorithm \ref{algo:algo1} till $S=40$ stages with $\eta_k=1/(10+k-1)$ for $k=1,\ldots,S$, 
yielded the estimate of the MLE to be $\hat\bx_{MLE}=(0.678854,0.293575)$, for which 
$\|\bof'(\hat\bx_{MLE})\|_2=0.012514$. 
The exercise takes $3$ minutes on our VMWare, implemented in parallel on $100$ cores. 

However, examination of the samples obtained by importance resampling 
at the different stages of step (2) of Algorithm \ref{algo:algo1} did not reveal any evidence of multimodality, and hence it is pertinent to consider that estimate
which corresponds to the minimum of $\{\|\bof'(\bx^*_i)\|_2:i=1,\ldots,N\}$, where $\bx^*_i$ are the original TMCMC samples, with $N=50000$. As such,
we modify the previous estimate to $\hat\bx_{MLE}=(0.678912,0.293809)$, which yields $\|\bof'(\hat\bx_{MLE})\|_2=0.007221$, which is somewhat closer to zero
compared to that for the previous estimate. Note however, that the two estimates of MLE are quite close to each other.  
\begin{figure}
	\centering
	\subfigure [TMCMC for MLE: first co-ordinate.]{ \label{fig:ex5_nlm2_1}
	\includegraphics[width=7.5cm,height=6.5cm]{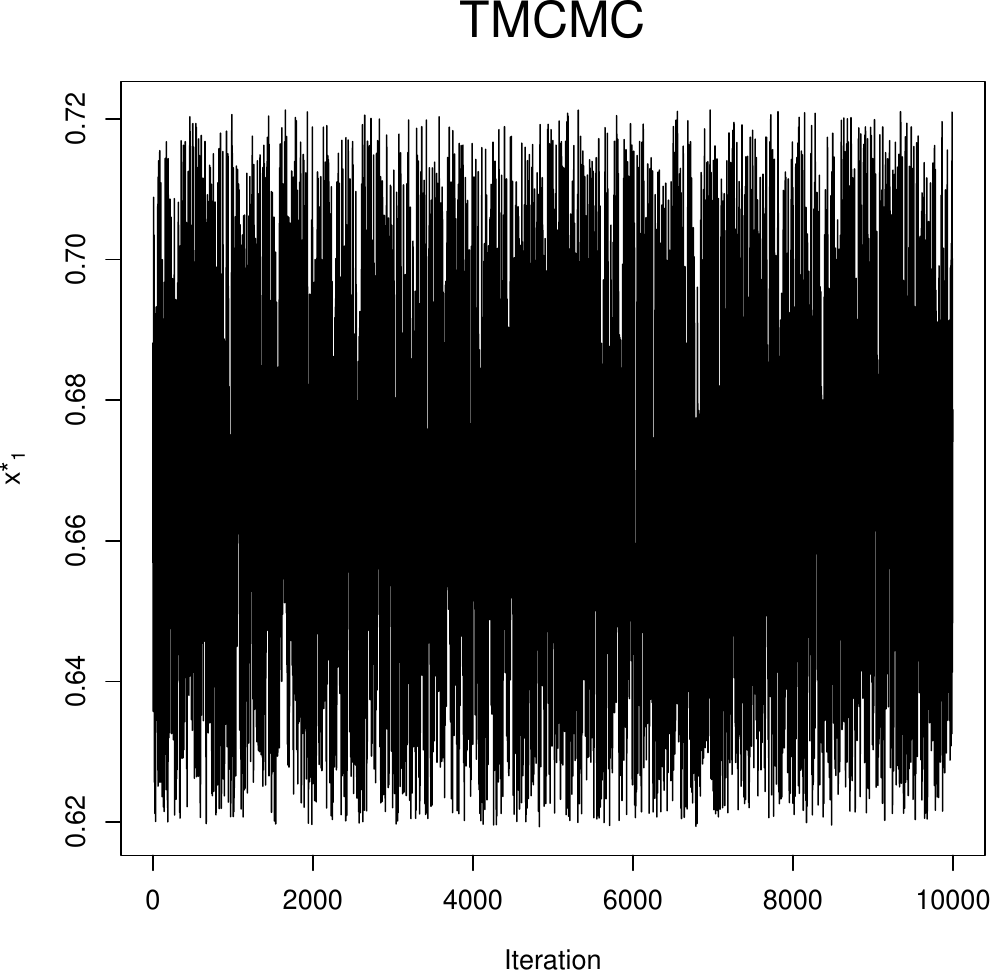}}
	\hspace{2mm}
	\subfigure [TMCMC for MLE: second co-ordinate.]{ \label{fig:ex5_nlm2_2}
	\includegraphics[width=7.5cm,height=6.5cm]{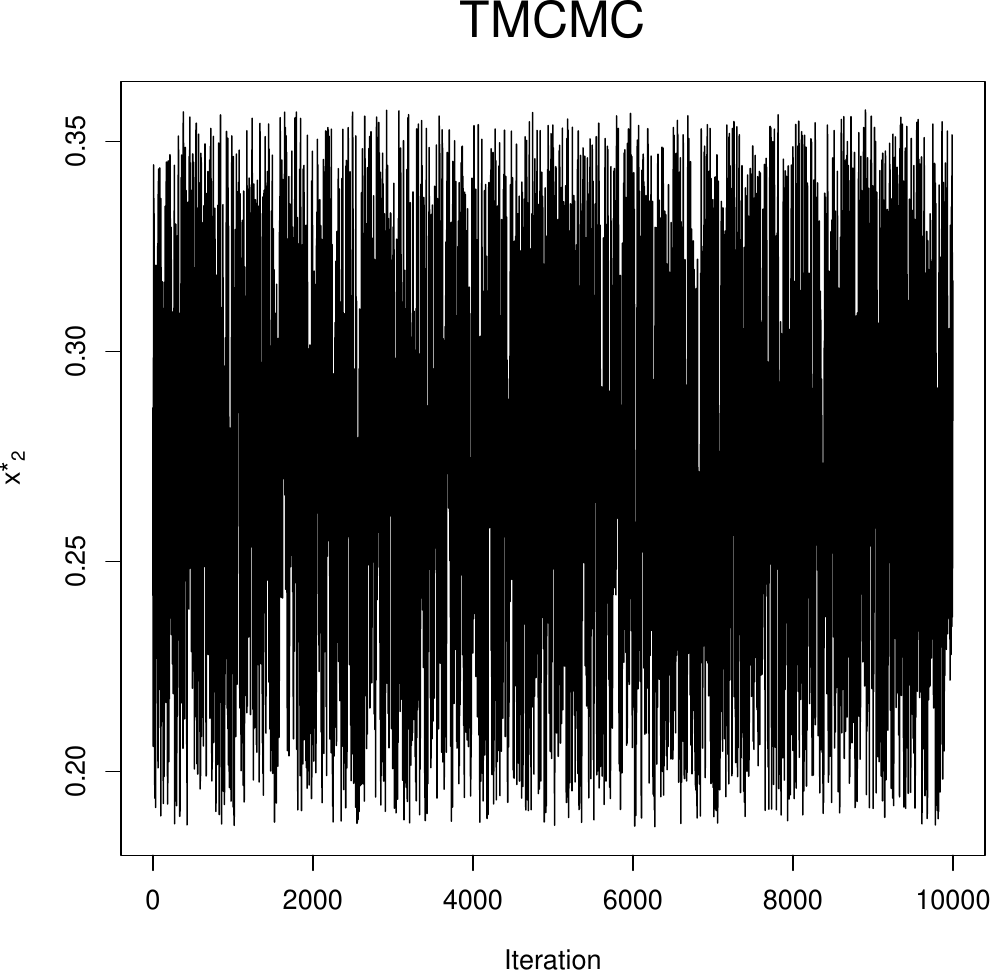}}
	\caption{TMCMC trace plots for Example 5 for finding MLE for dimension $d=2$.}
	\label{fig:ex5_nlm2}
\end{figure}

\subsubsection{Case 2: $d=5$}
\label{subsubsec:case2}

Here, for the $5$-dimensional TMCMC, we set $\epsilon=3$ in the restriction $\|\bof'(\cdot)\|_5<\epsilon$, as smaller values of $\epsilon$ led to poor convergence. 
Note that the Euclidean norm increases with dimension (see, for example, \ctn{Giraud15}), and so it is a natural requirement to increase $\epsilon$ 
as dimension increases. Similarly, we had to increase $\eta_k$ to $\eta_k=1.5/\log(10+k)$ for implementing step (2) of Algorithm \ref{algo:algo1}.  
The rest of the parameters of Algorithm \ref{algo:tmcmc} remain the same as for $d=2$.

The trace plots exhibited reasonable mixing; those for the
first and the last ($5$-th) co-ordinate of $\bx^*$ are depicted in Figure \ref{fig:ex5_nlm5}.
Running step (2) of Algorithm \ref{algo:algo1} till $S=40$ stages %with $\eta_k=1.5/\log(10+k-1)$ for $k=1,\ldots,S$, 
we obtain $\hat\bx_{MLE}=(0.663327, 0.431669, 0.045598, 0.239091, 0.301665)$, which corresponds to $\min\{\|\bof'(\bx^*_i)\|_5:i=1,\ldots,N\}\\=0.254217$, with $N=50000$. 
Note the increase in $\|\bof'(\cdot)\|_5$ compared to those of the smaller dimensions. Again, this is clearly to be expected because of the curse of dimensionality.
The exercise takes $4$ minutes to complete on our VMWare.
\begin{figure}
	\centering
	\subfigure [TMCMC for MLE: first co-ordinate.]{ \label{fig:ex5_nlm5_1}
	\includegraphics[width=7.5cm,height=6.5cm]{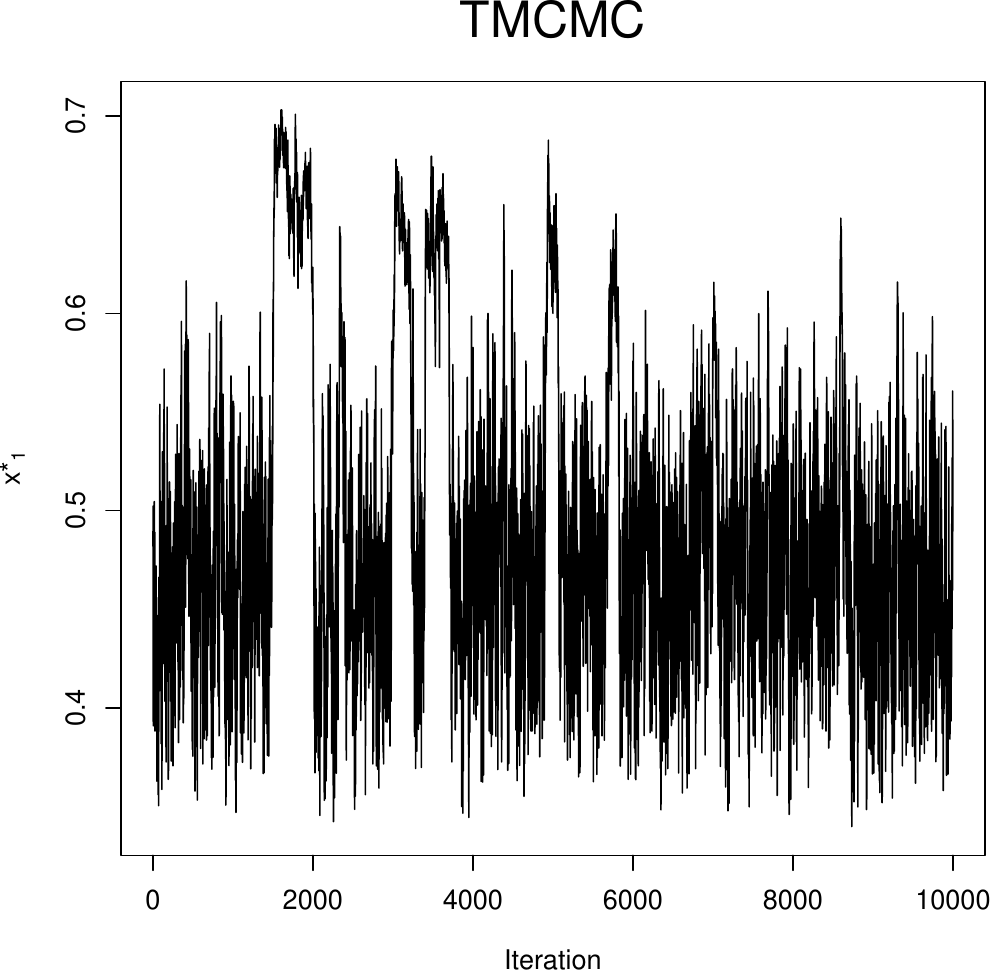}}
	\hspace{2mm}
	\subfigure [TMCMC for MLE: fifth co-ordinate.]{ \label{fig:ex5_nlm5_2}
	\includegraphics[width=7.5cm,height=6.5cm]{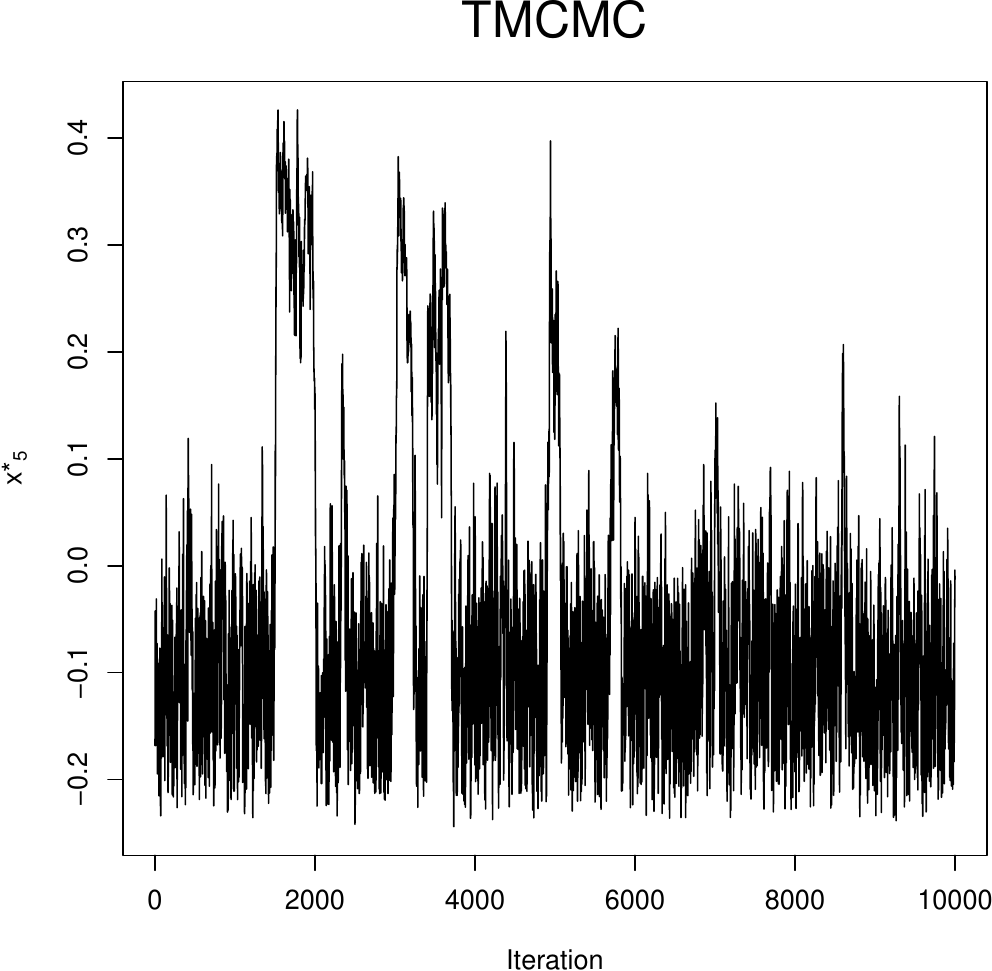}}
	\caption{TMCMC trace plots for Example 5 for finding MLE for dimension $d=5$.}
	\label{fig:ex5_nlm5}
\end{figure}

\subsubsection{Case 3: $d=10$}
\label{subsubsec:case3}
Here, for $d=10$, we had to set $\epsilon=6$ for the restriction $\|\bof'(\cdot)\|_{10}<\epsilon$ in TMCMC for adequate convergence.
We also set $\eta_k=7/\log(10+k-1)$ for implementing step (2) of Algorithm \ref{algo:algo1}.  

Our investigation shows that the mixing of TMCMC is not inadequate. The trace
plots of the first and the last co-ordinate of $\bx^*$ shown in Figure \ref{fig:ex5_nlm10} also bear evidence to this.
After implementing step (2) of Algorithm \ref{algo:algo1} till $S=40$ stages, we obtain
$\hat\bx_{MLE}=(0.472309, 0.10124, 0.079194, 0.108072, 0.096287,\\ -0.511566, 0.517567, -0.887624, 0.637796, -0.317721)$ with
$\|\bof'(\hat\bx_{MLE})\|_{10}=1.917746$. On the other hand, $\min\{\|\bof'(\bx^*_i)\|_{10}:i=1,\ldots,50000\}=1.902788$, which corresponds to
$\hat\bx_{MLE}=(0.471200, 0.100131, 0.078085, 0.101934, 0.095178, -0.504813, 4 0.516459, -0.888733, 0.631659,\\ -0.323859)$. Thus, both the estimates as well as the
corresponding gradients are quite close to each other. Again note the increase in $\|\bof'(\cdot)\|_{10}$ compared to those of the smaller dimensions.
\begin{figure}
	\centering
	\subfigure [TMCMC for MLE: first co-ordinate.]{ \label{fig:ex5_nlm10_1}
	\includegraphics[width=7.5cm,height=6.5cm]{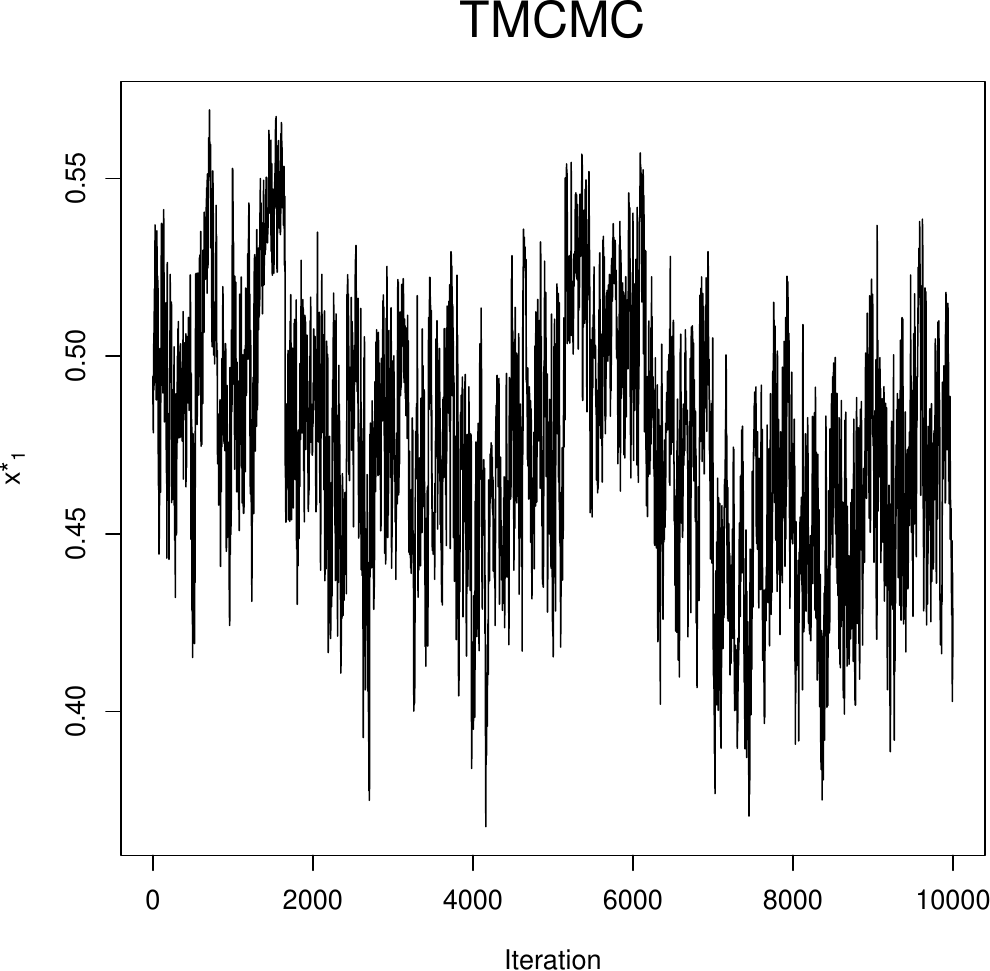}}
	\hspace{2mm}
	\subfigure [TMCMC for MLE: $10$-th co-ordinate.]{ \label{fig:ex5_nlm10_2}
	\includegraphics[width=7.5cm,height=6.5cm]{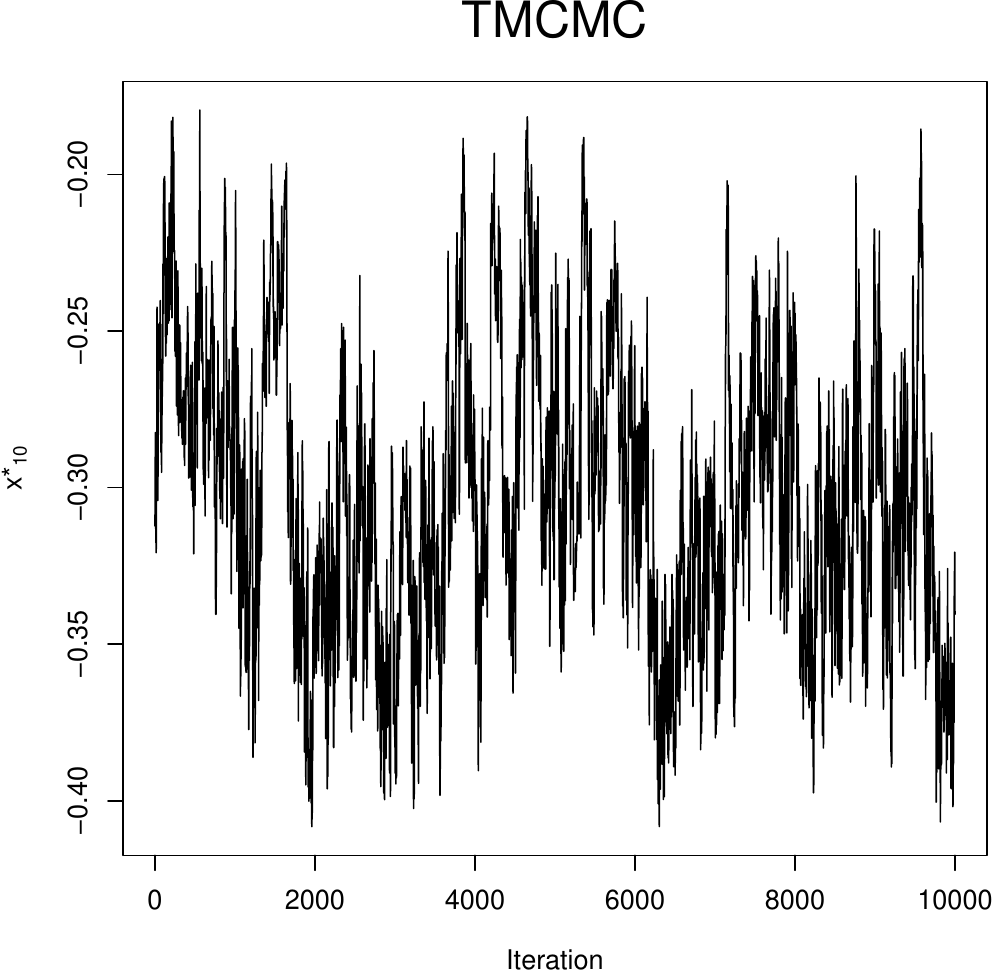}}
	\caption{TMCMC trace plots for Example 5 for finding MLE for dimension $d=10$.}
	\label{fig:ex5_nlm10}
\end{figure}
The entire exercise takes $17$ minutes on our VMWare to complete.

\subsubsection{Case 4: $d=50$}
\label{subsubsec:case4}

For this somewhat large dimension, we had to set $\epsilon=100$ and $\eta_k=200/\log(10+k-1)$.
Figure \ref{fig:ex5_nlm50}, which displays all stored $50000$ TMCMC realizations for $x^*_1$ and $x^*_{50}$ do not indicate excellent mixing, in spite
of the sophistication of Algorithm \ref{algo:tmcmc}. However, due to the high-dimension and complexity of the posterior this is not unexpected. As demonstrated
by Example 4, we can still expect to get closer to the MLE compared to other optimization methods. 
\begin{figure}
	\centering
	\subfigure [TMCMC for MLE: first co-ordinate.]{ \label{fig:ex5_nlm50_1}
	\includegraphics[width=7.5cm,height=6.5cm]{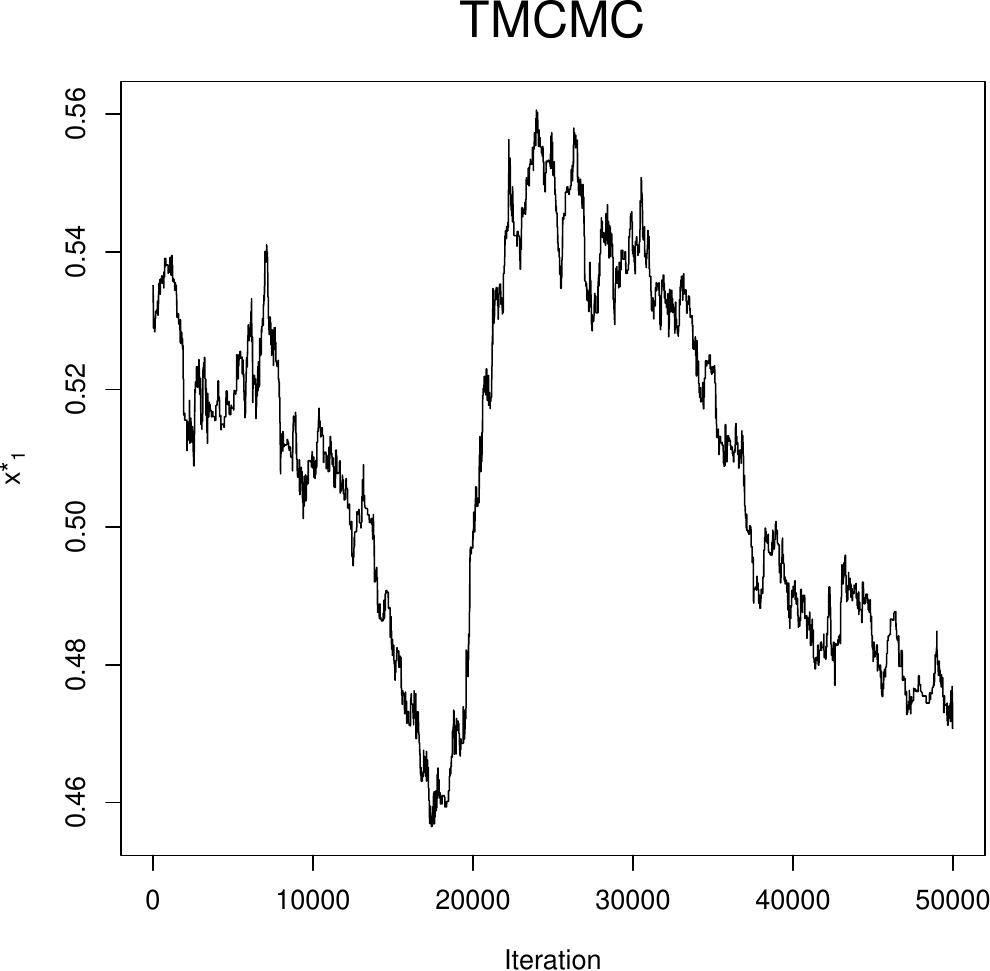}}
	\hspace{2mm}
	\subfigure [TMCMC for MLE: $50$-th co-ordinate.]{ \label{fig:ex5_nlm50_2}
	\includegraphics[width=7.5cm,height=6.5cm]{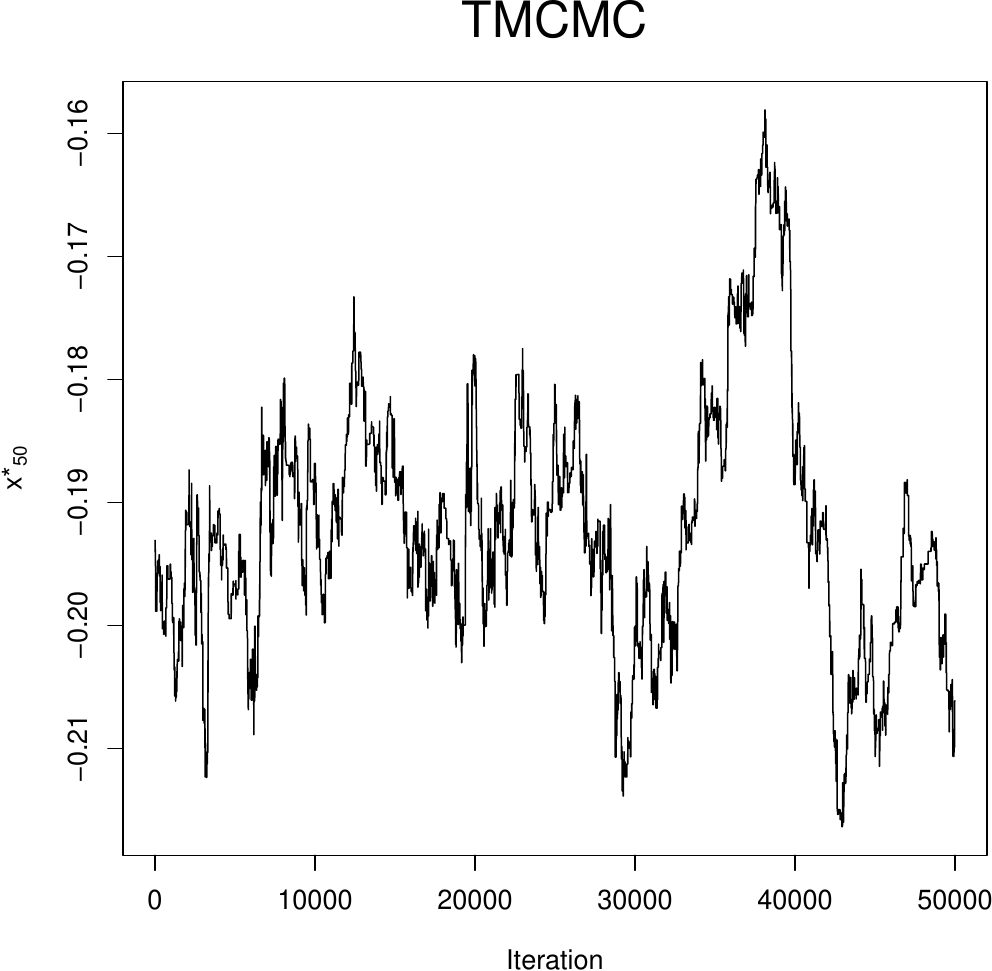}}
	\caption{TMCMC trace plots for Example 5 for finding MLE for dimension $d=50$.}
	\label{fig:ex5_nlm50}
\end{figure}
In this case, we obtain $\min\{\|\bof'(\bx^*_i)\|_{50}:i=1,\ldots,50000\}=73.05261$ while step (2) of Algorithm \ref{algo:algo1} implemented for $S=100$ stages,
of which only the first four stages consisted of positive $n_k$,
yielded $\|\bof'(\hat\bx_{MLE})\|_{50}=76.28244$. Thus, the norms of the gradients have increased considerably in this high dimension, compared to the previous
$d=2,5,10$. Given the high dimension in this problem, the above gradient values $73.05261$ and $76.28244$ are quite close. 

The TMCMC implementation took about $12$ hours on a single core in our VMWare and step (2) of Algorithm \ref{algo:algo1} took additional $2$ hours on $100$ cores. 

\subsubsection{Case 5: $d=100$}
\label{subsubsec:case5}
The curse of dimensionality now forced us to set $\epsilon=400$ for TMCMC convergence and $\eta_k=850/\log(10+k-1)$. 
It took around $15$ hours to complete the TMCMC run; the trace plots shown in Figure \ref{fig:ex5_nlm100} do not bear evidence of non-convergence, even though mixing is
expectedly inadequate in such high dimension. 
\begin{figure}
	\centering
	\subfigure [TMCMC for MLE: first co-ordinate.]{ \label{fig:ex5_nlm100_1}
	\includegraphics[width=7.5cm,height=6.5cm]{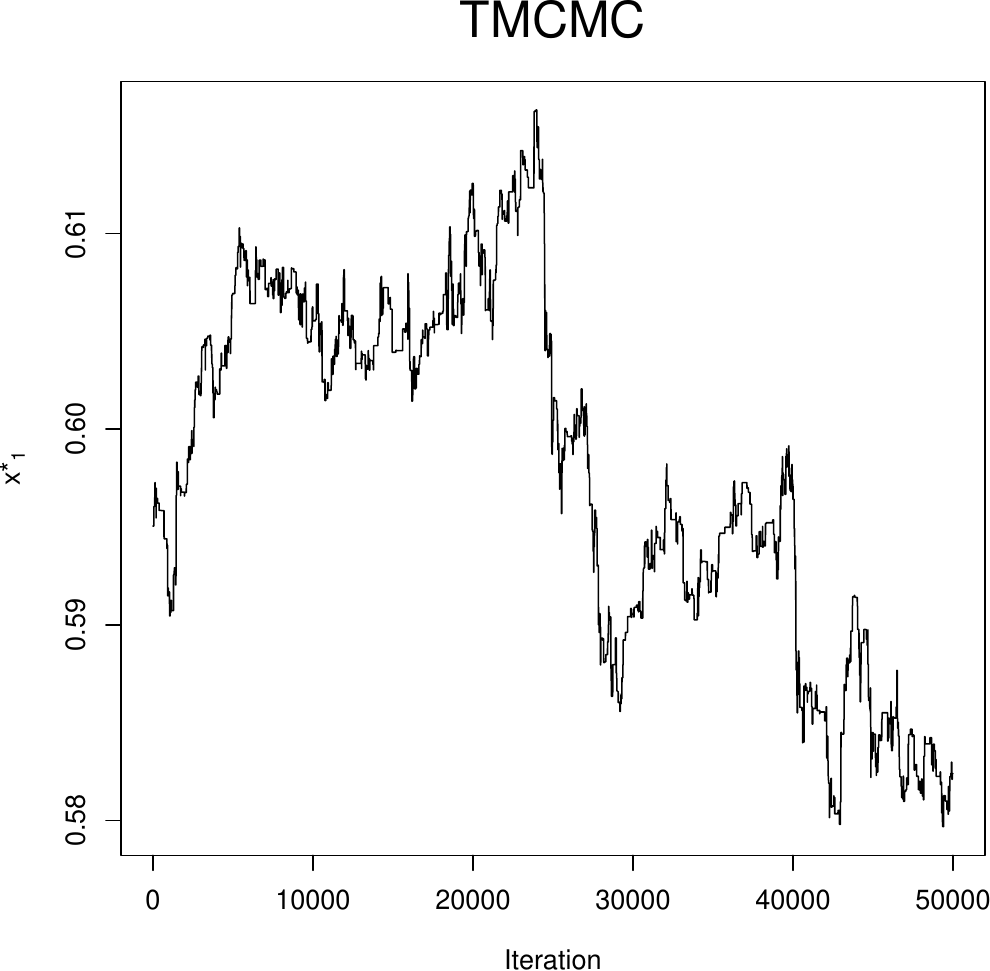}}
	\hspace{2mm}
	\subfigure [TMCMC for MLE: $100$-th co-ordinate.]{ \label{fig:ex5_nlm100_2}
	\includegraphics[width=7.5cm,height=6.5cm]{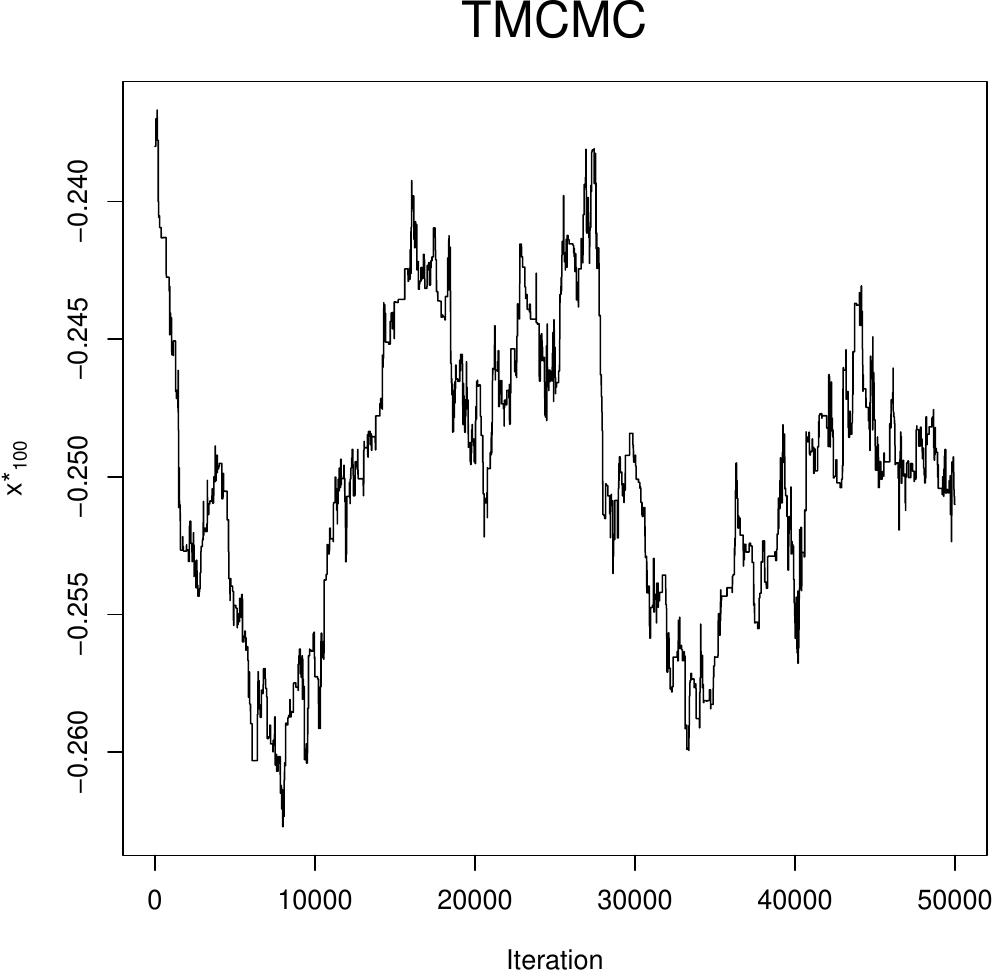}}
	\caption{TMCMC trace plots for Example 5 for finding MLE for dimension $d=100$.}
	\label{fig:ex5_nlm100}
\end{figure}
Here we obtain $\min\{\|\bof'(\bx^*_i)\|_{50}:i=1,\ldots,50000\}=330.4483$, while step (2) of Algorithm \ref{algo:algo1} implemented for till $S=100$ stages, 
of which only the first three stages yielded positive $n_k$, produced $\hat\bx_{MLE}$ such that $\|\bof'(\hat\bx_{MLE})\|_{100}=341.2009$. Given the high dimension, 
this value is quite close to the above minimum gradient.  
Implementation of step (2) of Algorithm \ref{algo:algo1} took $2$ hours $36$ minutes on $100$ cores on our VMWare.

\section{Summary and conclusion}
\label{sec:conclusion}

\subsection{The key ideas in a nutshell}

In this article, we have proposed and developed a novel Bayesian algorithm for general function optimization that judiciously exploits the derivatives of the 
objective function in conjunction with the posterior Gaussian derivative process, given data consisting of input points from the function domain and their 
corresponding function evaluations. The posterior simulation approach inherent to our method ensures improved accuracy of the obtained solutions compared to 
existing optimization algorithms. Another important feature is that, for any desired degree of accuracy, Bayesian credible regions of the optima are readily 
available at any stage of the algorithm.

Under an appropriate fixed‑domain infill asymptotics setup, we prove that the algorithm converges almost surely to the true optima. Along the way, we establish 
almost sure uniform convergence of the posteriors corresponding to the Gaussian process and its derivative process to the objective function and its derivatives, 
and we provide rates of convergence for a specific infill design. Furthermore, we present a Bayesian characterization of the number of optima of the objective 
function by exploiting the information contained in our algorithm.

\subsection{High-dimensional and mixing issues – i.i.d. sampling as an alternative to TMCMC}

The application of our Bayesian optimization algorithm to a variety of examples, ranging from simple to challenging, has produced encouraging and insightful results. The choice of initial values for the TMCMC sampler can affect mixing; however, as argued in the paper, the optima obtained by good existing optimization algorithms can serve as effective initial values for TMCMC. Moreover, we have shown that even when mixing is less than ideal, our Bayesian algorithm can still significantly outperform popular optimization methods. Consequently, the TMCMC mixing issue may not be critical, at least for low‑dimensional problems.

Dimensionality poses a much more serious challenge. Our experimental results demonstrate that as the dimension increases, the accuracy of the algorithm deteriorates. High dimensionality also severely impairs TMCMC mixing by excessively restricting the input space through the prior constraints. The reason is that the prior requires the gradient norm to be small and the Hessian positive definite; in high dimensions, the region of the parameter space that satisfies these conditions is extremely small, making it difficult for a Markov chain to explore effectively.

Fortunately, the TMCMC mixing problem can be completely circumvented by replacing TMCMC with the general i.i.d. sampling procedure developed by \ctn{bhattacharya2025iid} (see also \ctn{bhattacharya2021multimodal}, \ctn{bhattacharya2021doubly}, \ctn{bhattacharya2022dirichlet}, \ctn{bhattacharya2025flows} for extensions). Using i.i.d. sampling, sampling accuracy can be maintained almost perfectly even as the dimension increases, which is expected to dramatically improve the quality of the optima obtained by our method. This approach eliminates the need for careful tuning of proposal distributions and the risk of poor mixing, making it particularly attractive for high‑dimensional optimization problems.

\subsection{Towards black‑box optimization: a natural extension}

Although the current work assumes that the first and second partial derivatives of the objective function are available, a natural and promising extension is to adapt the algorithm to settings where \(f\) is a black box – only function evaluations are possible, and derivatives are unavailable. In this case, one can replace the true derivatives \(\mathbf{f}'\) and \(\mathbf{\Sigma}''\) in the prior and the refinement step with their Gaussian process surrogates \(\mathbf{g}'_n\) and the GP‑based Hessian. Because Theorems~\ref{theorem:infill2} and~\ref{theorem:theorem1} establish that \(\mathbf{g}'_n\) converges uniformly to \(\mathbf{f}'\) as the dataset grows, the modified algorithm would still converge to the true stationary points of \(f\). The refinement condition \(\|\mathbf{f}'(\mathbf{x})\|_d < \eta_k\) would be replaced by \(\|\mathbf{g}'_n(\mathbf{x})\|_d < \eta_k\), and the Hessian check would use the GP posterior Hessian. This adaptation turns our derivative‑based method into a general‑purpose, provably convergent black‑box optimization algorithm – a rare combination in the Bayesian optimization literature. We intend to pursue this direction in future work, providing rigorous convergence guarantees and demonstrating its performance on real‑world black‑box problems.

\subsection{Discussion of scalability and potential extensions}

The cubic complexity \(O(n^3)\) of the Gaussian process posterior, which arises from the inversion of the \(n\times n\) correlation matrix \(\boldsymbol{\Sigma}_{22}\), limits the method to moderate dataset sizes \(n\) (typically \(n \le 500\) in our experiments). For high‑dimensional problems, the number of refinement stages \(k\) is small and the dataset size \(n\) remains modest. To scale the method to very large input dimensions or very large dataset sizes, several extensions could be explored.

One promising direction is the use of sparse Gaussian process approximations, such as the fully independent training conditional (FITC) method of \ctn{Quinonero05}. By introducing a small set of \(m\) inducing points with \(m \ll n\), the complexity of the GP posterior reduces from \(O(n^3)\) to \(O(n m^2)\), making it feasible to handle larger datasets. Another approach is to employ additive kernels, where the covariance function is assumed to decompose as a sum of one‑dimensional kernels: \(c(\mathbf{x}, \mathbf{y}) = \sum_{k=1}^d c_k(x_k, y_k)\). This additive structure allows separate Gaussian process modeling per dimension, effectively reducing the dimensionality of the problem and enabling more efficient computation; see \ctn{Duvenaud11}.

Random Fourier features offer an alternative that approximates the covariance kernel using a finite set of random basis functions; see \ctn{Rahimi07}. This approximation yields linear complexity in the number of data points \(n\) and is particularly suitable for high‑dimensional problems, as the number of features needed to achieve a given accuracy does not depend exponentially on \(d\). Finally, if the covariance kernel is separable, the posterior of the derivative process factorizes across dimensions. This factorization can be exploited for parallel or distributed computation, where each dimension is processed independently on a separate processor, leading to significant speedups in high dimensions.

These extensions, together with the black‑box adaptation and i.i.d. sampling, constitute a rich agenda for future research. They would enable the application of our Bayesian optimization method to very high‑dimensional optimization problems and to problems with large numbers of data points, while retaining the theoretical guarantees that are the hallmark of our approach.

\section*{Acknowledgment}
We sincerely thank the reviewer whose comments have led to significant enhancement of the quality of our manuscript.

\section*{Conflict of interest}
The authors declare no conflicts of interest.

\section*{Data and code availability statement}
Our MPI-based parallel C code for implementing the AIDS data problem is available at \url{https://github.com/Sourabh-Bhattacharya/FUNCTION_OPT_GDP}.
The data is contained in the code itself, and is also available in \ctn{Lange10}.

\appendix
\section{Computational complexity of Algorithm \ref{algo:algo1}}
\label{app:complexity}

\subsection{Preliminary notation}

\begin{itemize}
    \item $d$ – dimension of the input space $\mathcal{X}$.
    \item $n_0$ – size of the initial dataset $\mathcal{D}_{n_0}$.
    \item $N_{\text{TMCMC}}$ – number of TMCMC iterations performed in step (1) (including burn‑in and thinning).
 \item $N$ – number of stored samples after thinning (for example, $N=50000$).
 \item $M$ – number of points resampled at each stage of step (2) (for example, $M=1000$).
    \item $K$ – total number of refinement stages.
    \item $L$ – maximum number of points augmented per stage (typically $L=5$).
    \item $n_k$ – number of points actually augmented at stage $k$ ($0\le n_k\le L$).
    \item $N_k$ – dataset size before stage $k$; $N_1 = n_0$, $N_{k+1}=N_k+n_k$.
    \item $N_{\max} = \max_k N_k \le n_0 + LK$.
    \item $K_{\text{aug}} = \#\{k: n_k>0\}$ – number of stages where augmentation occurs.
\end{itemize}

\subsection{Assumptions}

\begin{enumerate}
	\item[(i)] \textit{Dense linear algebra.} All matrix operations (multiplication, inversion, Cholesky decomposition) use standard algorithms. Multiplication of a $p\times q$ matrix by a $q\times r$ matrix costs $O(pqr)$; inversion of an $m\times m$ matrix costs $O(m^3)$; Cholesky decomposition costs $O(m^3)$.
	\item[(ii)] \textit{Analytical derivatives.} The first and second partial derivatives of the objective function $f$ are available in closed form. Evaluating the gradient $\mathbf{f}'(\mathbf{x})$ costs $O(d)$, assembling the Hessian $\mathbf{\Sigma}''(\mathbf{x})$ costs $O(d^2)$, and testing positive definiteness via Cholesky decomposition costs $O(d^3)$.
	\item[(iii)] \textit{Dataset augmentation.} When new points are added to the dataset, the correlation matrix $\boldsymbol{\Sigma}_{22}$ and its inverse are recomputed from scratch (no low‑rank updates).
	\item[(iv)] \textit{TMCMC implementation.} In each TMCMC iteration, %(Algorithm~2), 
	    the density of the current point is stored. Hence only one density evaluation (for the proposal) is performed per iteration.
    We assume additive TMCMC for simplicity.
%    \item[(v)] \textit{Parameter bounds.} The initial dataset size $n_0$ and the maximum number of augmented points per stage $L$ are fixed constants (for example, 
%	    $n_0=10$, $L=5$). Consequently, the total number of refinement stages $K$ is bounded, and the maximum dataset size $N_{\max}=n_0+LK$ is also bounded (or at most $O(K)$). For asymptotic analysis we treat $K$ as a constant.
\end{enumerate}

\subsection{Cost of one posterior density evaluation}
\subsubsection{Precomputed quantities}
Before the TMCMC loop, the following matrices are computed once:
\[
\mathbf{P} = \bigl(\mathbf{H}^T \boldsymbol{\Sigma}_{22}^{-1} \mathbf{H} + \boldsymbol{\Sigma}_0^{-1}\bigr)^{-1},
\]
and the vector \(\mathbf{v} = \boldsymbol{\Sigma}_{22}^{-1}(\mathbf{f}_n - \mathbf{H}\hat{\boldsymbol{\beta}})\), where \(\hat{\boldsymbol{\beta}}\) is the posterior mean of \(\boldsymbol{\beta}\). These quantities do not depend on the candidate point \(\mathbf{x}\) and are reused throughout.
\begin{lemma}
	\label{lemma:app1}
    For a fixed dataset $\mathcal{D}_n$ of size $n$ and a point $\mathbf{x}\in\mathcal{X}$, the computation of
    \[
    \pi\bigl(\mathbf{g}'(\mathbf{x})=\mathbf{0}\mid \mathcal{D}_n,\mathbf{x}\bigr)
    \]
    requires $O\bigl(d n^2 + d^2 n + d^3\bigr)$ floating‑point operations, assuming that all quantities that do not depend on $\mathbf{x}$ (for example, $\boldsymbol{\Sigma}_{22}^{-1}$, $\mathbf{P}$, $\mathbf{A}\hat{\boldsymbol{\beta}}$) have been precomputed.
\end{lemma}
\begin{proof}
	The evaluation follows equations~\ref{eq:postpred4}, \ref{eq:postpred_mean1}, \ref{eq:postpred_var1}, \ref{eq:postmean1} and \ref{eq:postvar1} of the main paper. 
	We list the steps and their costs.

    \begin{enumerate}
	    \item[(1)] Compute $\boldsymbol{\Sigma}_{12}(\mathbf{x})$ ($d\times n$). Each column requires $O(d)$ work $\implies$ $O(dn)$.
	    \item[(2)] Compute $\hat{\boldsymbol{\mu}}'(\mathbf{x}) = \mathbf{A}\hat{\boldsymbol{\beta}} + \boldsymbol{\Sigma}_{12}(\mathbf{x})\,\mathbf{v}$, where $\mathbf{v}$ is a precomputed $n$-vector.  
              $\boldsymbol{\Sigma}_{12}\mathbf{v}$ costs $O(dn)$; addition of the constant $O(d)$.
      \item[(3)] Compute $\mathbf{M} = \boldsymbol{\Sigma}_{12}(\mathbf{x})\,\boldsymbol{\Sigma}_{22}^{-1}$: $d\times n$ times $n\times n$ $\implies$ $O(d n^2)$.
      \item[(4)] Compute $\tilde{\boldsymbol{\Sigma}} = \boldsymbol{\Sigma}_{11} - \mathbf{M}\,\boldsymbol{\Sigma}_{12}(\mathbf{x})^T$:  
              $\mathbf{M}\boldsymbol{\Sigma}_{12}^T$ is $d\times n$ times $n\times d$ $\implies$ $O(d^2 n)$; subtraction $O(d^2)$.
      \item[(5)] Compute $\mathbf{Q} = \mathbf{A} - \mathbf{M}\mathbf{H}$:  
              $\mathbf{M}\mathbf{H}$ is $d\times n$ times $n\times(d+1)$ $\implies$ $O(d^2 n)$; subtraction $O(d^2)$.
      \item[(6)] Compute $\mathbf{Q}\mathbf{P}\mathbf{Q}^T$:  
              $\mathbf{Q}\mathbf{P}$: $d\times(d+1)$ times $(d+1)\times(d+1)$ $\implies$ $O(d^3)$.  
              Then $(\mathbf{Q}\mathbf{P})\mathbf{Q}^T$: $d\times(d+1)$ times $(d+1)\times d$ $\implies$ $O(d^3)$.  
              Add to $\tilde{\boldsymbol{\Sigma}}$: $O(d^2)$.
      \item[(7)] Invert $\hat{\boldsymbol{\Sigma}}$ (size $d\times d$): $O(d^3)$ (Cholesky or LU).
      \item[(8)] Compute the quadratic form $\hat{\boldsymbol{\mu}}'^T \hat{\boldsymbol{\Sigma}}^{-1} \hat{\boldsymbol{\mu}}'$:  
              matrix‑vector product $O(d^2)$, dot product $O(d)$.
      \item[(9)] Combine with the precomputed constant denominator: $O(1)$.
    \end{enumerate}
    Summing the dominating terms gives $O(dn^2 + d^2 n + d^3)$. 
\end{proof}

\subsection{Cost of one TMCMC iteration (step (1))}

\begin{lemma}
	\label{lemma:app2}
    Each iteration of the TMCMC loop %(Algorithm~2) 
	executed in step (1) of Algorithm~\ref{algo:algo1} has complexity
    \[
    O\bigl(d^3 + d^2 n_0 + d n_0^2\bigr).
    \]
\end{lemma}
\begin{proof}
	An iteration proposes a new point $\mathbf{y}$ (cost $O(d)$), evaluates its density using Lemma~\ref{lemma:app1} with $n=n_0$ (cost as above), and performs an acceptance test (constant time). The density of the current point is stored from the previous iteration. Hence the total cost is dominated by the density evaluation. 
\end{proof}

\subsection{Cost of one stage of the refinement loop (step (2))}

\begin{lemma}
	\label{lemma:app3}
	Consider stage $k$ of step (2), where before the stage the dataset size is $N_k$ (with $N_1=n_0$) and after augmentation it becomes $N_{k+1}=N_k+n_k$ ($0\le n_k\le L$). The following are available from the previous stage: all precomputed matrices for dataset size $N_k$, and the stored denominator densities $\pi(\mathbf{g}'(\mathbf{x}_i^*)=\mathbf{0}\mid\mathcal{D}_{N_k},\mathbf{x}_i^*)$ for all stored samples $\mathbf{x}_i^*$. Then the computational cost of stage $k$ is
    \[
    \begin{aligned}
    C_k = O\Bigl(&N\bigl(d^3 + d^2 N_k + d N_k^2\bigr) \quad\text{(density evaluations for weights)}\\
    &+ N \quad\text{(weight multiplication and storage)}\\
    &+ (N+M) \quad\text{(resampling)}\\
    &+ M d^3 \quad\text{(gradient/Hessian checks)}\\
    &+ \mathbf{1}_{\{n_k>0\}}\bigl(N_{k+1}^3 + d^3 + d^2 N_{k+1} + n_k C_f\bigr) \quad\text{(augmentation)}\Bigr),
    \end{aligned}
    \]
    where $N$ is the number of stored samples, $M$ the resample size, and $C_f$ the cost of evaluating the objective function $f$ at a single point.
\end{lemma}
\begin{proof}
    The work is broken down as follows.

    \begin{itemize}
	    \item[(1)] \textit{Weight update (density evaluation).} For each of the $N$ stored samples we need the new density $\pi(\mathbf{g}'(\mathbf{x}_i^*)=\mathbf{0}\mid\mathcal{D}_{N_k},\mathbf{x}_i^*)$ (numerator). By Lemma~\ref{lemma:app1} this costs $N\cdot C_{\text{eval}}(N_k) = O(N(d^3 + d^2 N_k + d N_k^2))$.
	    \item[(2)] \textit{Weight update (multiplication).} Each new density is multiplied by the stored previous weight $w_{k-1}(\mathbf{x}_i^*)$ to obtain $w_k(\mathbf{x}_i^*)$, which is then stored. This costs $O(N)$.
	    \item[(3)] \textit{Resampling.} 
		    To draw $M$ distinct indices without replacement with probabilities proportional to the normalized weights $w_k(\bx^*_i)$, we can use the alias method (\ctn{Walker1974,Walker1977,Vose1991}) or systematic resampling (\ctn{Kitagawa1996,Carpenter1999}). Both methods require $O(N)$ preprocessing: building the alias table for the alias method or computing cumulative sums for systematic resampling. After preprocessing, each sample can be drawn in $O(1)$ time using the alias method, while systematic resampling draws all $M$ samples in $O(M)$ time. Consequently, the total cost of resampling is $O(N + M)$.
	    \item[(4)] \textit{Gradient and Hessian checks.} For each of the $M$ resampled points, we compute $\|\mathbf{f}'(\mathbf{x})\|_d$ ($O(d)$), assemble the Hessian $\mathbf{\Sigma}''(\mathbf{x})$ ($O(d^2)$), and perform Cholesky decomposition to test positive definiteness ($O(d^3)$). Total $O(M d^3)$.
	    \item[(5)] \textit{Augmentation (if $n_k>0$).} We evaluate $f(\mathbf{x})$ for the $n_k$ points that satisfy the gradient condition. The cost per evaluation is $C_f$, so total $O(n_k C_f)$. Then we recompute all precomputed quantities for the new dataset size $N_{k+1}$: forming $\boldsymbol{\Sigma}_{22}$ ($O(N_{k+1}^2)$) and its inverse ($O(N_{k+1}^3)$); recomputing $\mathbf{P}=(\mathbf{H}^T\boldsymbol{\Sigma}_{22}^{-1}\mathbf{H}+\boldsymbol{\Sigma}_0^{-1})^{-1}$ ($O(d^3+d^2 N_{k+1})$); and other auxiliary matrices ($O(d^2 N_{k+1})$). The dominant cost is $O(n_k C_f+N_{k+1}^3 + d^3 + d^2 N_{k+1})$.
    \end{itemize}
    Summing all contributions yields the stated bound. 
\end{proof}

\section*{Total complexity of Algorithm~\ref{algo:algo1}}

\begin{theorem}
   Let Algorithm~1 be run with initial dataset size $n_0$, $N_{\text{TMCMC}}$ TMCMC iterations, $N$ stored samples, resample size $M$, at most $K$ refinement stages, and at most $L$ points augmented per stage. %where $K$ and $L$ are fixed constants (independent of $d$, $N$, and $N_{\text{TMCMC}}$). 
Let $C_f$ denote the cost of evaluating $f$ at a single point. Then the total time complexity is
    \[
    \begin{aligned}
    T_{\text{total}} = O\Bigl(&n_0^3 + d^3 + d^2 n_0 \quad\text{(precomputation)}\\
    &+ N_{\text{TMCMC}}(d^3 + d^2 n_0 + d n_0^2) \quad\text{(TMCMC)}\\
    &+ \sum_{k=1}^{K} N(d^3 + d^2 N_k + d N_k^2) \quad\text{(density evaluations for weights)}\\
    &+ K N \quad\text{(weight multiplication)}\\
    &+ K(N+M) \quad\text{(resampling)}\\
    &+ K M d^3 \quad\text{(gradient/Hessian checks)}\\
    &+ \sum_{k: n_k>0} (N_{k+1}^3 + d^3 + d^2 N_{k+1} + n_k C_f) \quad\text{(augmentation)}\Bigr),
    \end{aligned}
    \]
    where $N_1=n_0$, $N_{k+1}=N_k+n_k$, and $0\le n_k\le L$.

    Since $N_{\max}= \max_k N_k \le n_0+ LK$, the total complexity simplifies to
%    \[
%    T_{\text{total}} = O\!\left( N_{\text{TMCMC}}(d^3 + d^2 n_0 + d n_0^2) \;+\; K N (d^3 + d^2 N_{\max} + d N_{\max}^2) \;+\; K(N + M d^3) \;+\; K_{\text{aug}} (N_{\max}^3 + L C_f) \right),
%    \]
    \[
\begin{aligned}
T_{\text{total}} = O\!\Bigl( & N_{\text{TMCMC}}(d^3 + d^2 n_0 + d n_0^2) + K N (d^3 + d^2 N_{\max} + d N_{\max}^2) \\
& + K(N + M d^3) + K_{\text{aug}} (N_{\max}^3 + L C_f) \Bigr),
\end{aligned}
\]
    where $K_{\text{aug}}=\#\{k:n_k>0\}\le K$.
\end{theorem}
\begin{proof}
    The total cost is the sum of:
    \begin{enumerate}
	    \item[(1)] \textit{Precomputation (once)}. Form $\boldsymbol{\Sigma}_{22}$ ($O(n_0^2)$), its inverse ($O(n_0^3)$), $\mathbf{P}$ ($O(d^3+d^2 n_0)$), and other constant‑time operations. Hence $O(n_0^3 + d^3 + d^2 n_0)$.
	    \item[(2)] \textit{Step~1 (TMCMC)}: $N_{\text{TMCMC}}$ iterations, each costing $O(d^3+d^2 n_0+d n_0^2)$ by Lemma~\ref{lemma:app2}.
	    \item[(3)] \textit{Step~2}: sum of the costs of $K$ stages given by Lemma~\ref{lemma:app3}. Summing over $k=1,\dots,K$ yields the terms shown. The term $\sum_{k} N(d^3+d^2 N_k+d N_k^2)$ appears because each stage requires $N$ density evaluations with dataset size $N_k$. The additional $K N$ accounts for the multiplications. The $K(N+M)$ and $K M d^3$ come from resampling and gradient/Hessian checks. The last sum accounts for the recomputation after augmentation.
    \end{enumerate}
    Using $N_k\le N_{\max}$, we obtain the simplified bound. The cost $C_f$ appears only in stages where $n_k>0$, and each such stage contributes at most $L C_f$, so the total contribution from function evaluations is $O(K_{\text{aug}} L C_f)$. 
\end{proof}

\begin{remark}
    The cost $C_f$ is problem‑dependent. In all examples of this paper, $f$ is a sum of elementary functions (polynomials, exponentials, etc.) evaluated at $O(md)$ terms, where $m$ is the number of data points. Because $m$ is at most $200$ when $d=100$, we have $C_f = O(d^2)$ (polynomial in $d$). The dominant term in the total complexity is then $N_{\text{TMCMC}} d^3$, since $N_{\text{TMCMC}}$ is large (for example, $5\times10^6$). This matches the empirical runtimes reported in Section~\ref{sec:exps}.
\end{remark}

\begin{remark}
  If the objective function $f$ were such that $C_f$ grows faster than any polynomial (for example, exponentially) in $d$, the total complexity could be dominated by 
$K_{\text{aug}} L C_f$. However, for the functions considered in this paper (and for most practical optimization problems where derivatives are available), 
$C_f$ is at most polynomial in $d$, and the TMCMC term remains the primary bottleneck for moderate to large $d$.
\end{remark}

\newpage
\normalsize
\bibliographystyle{natbib}
\bibliography{irmcmc}

\end{document}